\documentclass[12pt]{amsart}

\usepackage{longtable}
\usepackage{enumerate}

\allowdisplaybreaks

\usepackage[english]{babel}

\usepackage[%pagebackref,
pdftex,hyperfootnotes]{hyperref}
\hypersetup{
 colorlinks=true,
 linkcolor=NavyBlue, 
 urlcolor=RoyalPurple,
 citecolor=OliveGreen,
 pdftitle={Oscillation, multiple orthogonality and Markov chains},
 bookmarks=true,
 pdfpagemode=FullScreen,
}
\usepackage[usenames,dvipsnames,svgnames,table,x11names]{xcolor}
\usepackage{pgfplots}
\usepackage{tikz}
\usepackage{tikz-3dplot}
\usetikzlibrary{automata,quotes, chains,matrix,calc,shadows,shapes.callouts,shapes.geometric,shapes.misc,positioning,patterns,decorations.shapes,
decorations.pathmorphing,decorations.markings,decorations.fractals,decorations.pathreplacing,shadings,fadings,arrows.meta,bending}
\usepackage{nicematrix}
\usepackage[utf8]{inputenc}
\usepackage{polski}
\usepackage{comment}
\usepackage[textwidth=17.5cm,textheight=22.275cm,
height=%24cm,
24.275cm,
width=18cm]{geometry}

\usepackage{fnpct}
\usepackage{rotating,multirow,pdflscape}
\usepackage{amssymb,latexsym,amsmath,amsthm,bm}
\usepackage{mathrsfs}
\usepackage{mathtools,arydshln,mathdots}
\usepackage{bigints}
\makeatletter
\renewcommand*\env@matrix[1][*\c@MaxMatrixCols c]{%
 \hskip -\arraycolsep
 \let\@ifnextchar\new@ifnextchar
 \array{#1}}
\makeatother

\mathtoolsset{centercolon}
\usepackage[table]{xcolor}
\usepackage[toc,page]{appendix}

\usepackage{tocvsec2}
\usepackage{Baskervaldx}
\usepackage[]{newtxmath}

\newtheorem{coro}{Corollary}

\newtheorem{defi}{Definition}
\newtheorem{teo}{Theorem}

\newtheorem{pro}{Proposition}
\newtheorem{lemma}{Lemma}

\newtheorem{rem}{Remark}

\renewcommand{\d}{\operatorname{d}}

\newcommand{\diag}{\operatorname{diag}}

\newcommand{\C}{\mathbb{C}}

\newcommand{\B}{\mathbb{B}}

\newcommand{\N}{\mathbb{N}}
\newcommand{\R}{\mathbb{R}}

\allowdisplaybreaks[1]

\makeatletter
\DeclareRobustCommand{\gaussk}{\DOTSB\gaussk@\slimits@}
\newcommand{\gaussk@}{\mathop{\vphantom{\sum}\mathpalette\bigcal@{K}}}

\newcommand{\bigcal@}[2]{%
 \vcenter{\m@th
 \sbox\z@{$#1\sum$}%
 \dimen@=\dimexpr\ht\z@+\dp\z@
 \hbox{\resizebox{!}{0.8\dimen@}{$\mathcal{K}$}}%
}%
}
\newcommand{\cfracplus}{\mathbin{\cfracplus@}}
\newcommand{\cfracplus@}{%
 \sbox\z@{$\dfrac{1}{1}$}%
 \sbox\tw@{$+$}%
 \raisebox{\dimexpr\dp\tw@-\dp\z@\relax}{$+$}%
}
\newcommand{\cfracdots}{\mathord{\cfracdots@}}
\newcommand{\cfracdots@}{%
 \sbox\z@{$\dfrac{1}{1}$}%
 \sbox\tw@{$+$}%
 \raisebox{\dimexpr\dp\tw@-\dp\z@\relax}{$\cdots$}%
}
\makeatother

\makeatletter
\newcommand*{\relrelbarsep}{.386ex}
\newcommand*{\relrelbar}{%
	\mathrel{%
		\mathpalette\@relrelbar\relrelbarsep
	}%
}
\newcommand*{\@relrelbar}[2]{%
	\raise#2\hbox to 0pt{$\m@th#1\relbar$\hss}%
	\lower#2\hbox{$\m@th#1\relbar$}%
}
\providecommand*{\rightrightarrowsfill@}{%
	\arrowfill@\relrelbar\relrelbar\rightrightarrows
}
\providecommand*{\leftleftarrowsfill@}{%
	\arrowfill@\leftleftarrows\relrelbar\relrelbar
}
\providecommand*{\xrightrightarrows}[2][]{%
	\ext@arrow 0359\rightrightarrowsfill@{#1}{#2}%
}
\providecommand*{\xleftleftarrows}[2][]{%
	\ext@arrow 3095\leftleftarrowsfill@{#1}{#2}%
}
\makeatother
\begin{document}

\title[Oscillation, multiple orthogonality and Markov chains]{Oscillatory banded Hessenberg matrices, multiple orthogonal polynomials and Markov chains}

\author[A Branquinho]{Amílcar Branquinho$^{1}$}
\address{$^1$CMUC, Departamento de Matemática,
 Universidade de Coimbra, 3001-454 Coimbra, Portugal}
\email{ajplb@mat.uc.pt}

\author[A Foulquié]{Ana Foulquié-Moreno$^{2}$}
\address{$^2$CIDMA, Departamento de Matemática, Universidade de Aveiro, 3810-193 Aveiro, Portugal}
\email{foulquie@ua.pt}

\author[M Mañas]{Manuel Mañas$^{3}$}
\address{$^3$Departamento de Física Teórica, Universidad Complutense de Madrid, Plaza Ciencias 1, 28040-Madrid, Spain \&
 Instituto de Ciencias Matematicas (ICMAT), Campus de Cantoblanco UAM, 28049-Madrid, Spain}
\email{manuel.manas@ucm.es}



%\enlargethispage{1.25cm}

\begin{abstract}
A spectral Favard theorem for bounded banded lower Hessenberg matrices that admit a positive bidiagonal factorization is found. The large knowledge on the spectral and factorization properties of oscillatory matrices leads to this spectral Favard theorem in terms of sequences of multiple orthogonal polynomials of types I and II with respect to a set of positive Lebesgue--Stieltjes~measures. Also a multiple Gauss quadrature is proven and corresponding degrees of precision are found. 

This spectral Favard theorem is applied to Markov chains with $(p+2)$-diagonal transition matrices, i.e. beyond birth and death, that admit a positive stochastic bidiagonal factorization.
In the finite case, the Karlin--McGregor spectral representation is given.
It is shown that the Markov chains are recurrent and explicit expressions in terms of the orthogonal polynomials for the stationary distributions are given.
Similar results are obtained for the countable infinite Markov chain. Now the Markov chain is not necessarily recurrent, and it is characterized in terms of the first measure.
Ergodicity of the Markov chain is discussed in terms of the existence of a mass at $1$, which is an eigenvalue corresponding to the right and left eigenvectors. 
\end{abstract}
 
 \maketitle
 
% \clearpage
 
% \tableofcontents

\thispagestyle{empty}
 \section{Introduction}
In functional analysis of linear operators it is known that a self-adjoint operator defined in a Hilbert space with a cyclic vector $u$ can be represented as a symmetric Jacobi operator in a suitable orthonormal basis, see for instance \cite[Theorem 5.3.1]{Simon}.
 
 The infinite matrix representative of these Jacobi operators is tridiagonal which leads to a three term recurrence relation definition of an orthonormal polynomial basis with respect to the spectral measure of the Jacobi operator.
 In this way the theory of orthogonal polynomials is instrumental to understand the spectrality of a self-adjoint operator in a Hilbert space \cite{Simon}.
 The construction of the spectral measure for these Jacobi operators can be reached by means of a Markov theorem, a uniform limit of ratio for polynomials satisfying the three term recurrence relations coming from the Jacobi operator. The spectral measure is in fact a measure of orthogonality for the sequence of orthogonal polynomials coming from the Jacobi operator, as they share its moments. Spectral theorems hold beyond self-adjointness for normal operators (the operator commutes with its adjoint).
In the case of banded Hessenberg operators, the self-adjointness or normality no more takes place. Hence, the mentioned well established spectral theory is lost. Nevertheless, these banded Hessenberg operators case can be understood as higher order ($4$ or more terms) recursion relations and have associated recursion polynomials, known as type I and II.
 
Moreover, a biorthogonality can be derived between these two systems that stands as duality relations. In fact, this is a sort of \emph{nonspectral} Favard theorem for banded Hessenberg operators. Several authors, in particular Sorokin and Van~Iseghem have studied these nonspectral Favard type theorems \cite{Sorokin_Van_Iseghem_1,Sorokin_Van_Iseghem_2,Sorokin_Van_Iseghem_3}.
 Unfortunately these nonspectral Favard type theorems for high order recurrence relations only contemplate the existence of functionals (and define the moments in terms of the recursion coefficients) that happened to be biorthogonal. These sequences of functions takes part of a theory of multiple orthogonal polynomials, type I and II \cite{nikishin_sorokin}. For more on multiple orthogonal polynomials see \cite[Chapter 23]{Ismail} and also \cite{afm,bfm}.
 
 \enlargethispage{.25cm}
 
 A natural question arises from these facts for the operator theory side: Do we have a spectral Favard theorem for a subclass of bounded lower Hessenberg matrices? Or, from the orthogonal polynomial side:
 Do multiple orthogonal polynomials can be of any use in the spectral description of a banded Hessenberg operators? That is, do they give an interpretations for the spectral points, as well as of the spectral measures of the corresponding operator?
 
 A first attempt tackle these problems has been made by Kalyagin in \cite{Kalyagin,Kaliaguine}
 where the author defines a class of operators related with the Hermite--Padé approximants and connects their spectral analysis with the asymptotic properties of polynomials defined by systems of orthogonality relations (that coincides with the notion of multiple orthogonality). Some years later, this author in a joint work with Aptekarev and Van Iseghem 
 \cite{Aptekarev_Kaliaguine_VanIseghem} made the analysis of banded Hessenberg operators with one upper diagonal based on the analysis of the genetic sums formulas for the moments of the operator.
 Regarding these achievements in \cite[Conclusions]{Coussement-VanAssche} that Coussement and Van Assche say:
\emph{``However, finding necessary and sufficient conditions on the recurrence coefficients to have positive orthogonality measures on the real axis is still an open problem''}.
See also \cite{VanAssche2}. 

In \cite{Aptekarev_Kaliaguine_Lopez} general properties and limit behavior of the coefficients in the recurrence relation satisfied by multiple orthogonal polynomials were studied, including as particular cases Angelescu and Nikishin systems. There it was noticed that for the Hessenberg operators considered in this paper:
\emph{
``This is a non-symmetric $(m + 2)$ diagonal lower Hessenberg matrix. A very important point is that the spectral properties of this operator are closely connected with the asymptotic properties of the multiple orthogonal polynomials (see [4])
In particular, one can consider the system of measures $\mu_1 , \mu_2 , \dots , \mu_m$ generating the relations (3) as the system of spectral measures of the associated Hessenberg operator. This implies the possibility of stating and investigating direct and inverse spectral and scattering problems for this class of operators using advanced results for multiple orthogonal polynomials.''}
Higher order recurrence relations, associated with similar Hessenberg matrices with all diagonals being zero but for the first superdiagonal and the lowest subdiagonal where discussed in \cite{Aptekarev_Kaliaguine_Saff}, associated with a star-like set $S$. There, positive measures $\d\mu_j$ supported in this star-like set and corresponding multiple orthogonal polynomials where found. Also in \cite{Beckermann_Osipov} certain operator classes represented by banded infinite matrices where studied and a characterization of the corresponding resolvent set in terms of polynomial solutions of the underlying higher order recurrence relations was given. That allowed for the description of some asymptotic behavior of the corresponding systems of multiple orthogonal polynomials. 
Finally in \cite{KaliaguineII} a number of Markov functions of positive weights supported on disjoint real intervals, satisfying a Szeg\H{o} condition, and the associated set of Angeslescu multiple orthogonal polynomials was considered and it was shown that the corresponding recurrence Hessenberg operator is a perturbation of an operator with a purely periodic matrix. 

 Since the publication of the keystone monograph \cite{Gantmacher-Krein} it is well known that the feature of being oscillatory of a Jacobi matrix is important in the study of the spectral points of a Jacobi operator. They are also instrumental in the study of the orthogonal polynomials sequences used to define the spectral measure of the operator.

%\enlargethispage{.25cm}

%\subsection{Oscillatory matrices}

%\textcolor{red}{Given a matrix $A$ we have the natural concepts of submatrix and minor.
%For two subsets of indexes $\boldsymbol\alpha=\{ \alpha_1, \alpha_2,\dots, \alpha_r\}$ and $\boldsymbol\beta=\{\betT_1,\betT_2,\dots,\betT_s\}$ the submatrix $A[\boldsymbol \alpha,\boldsymbol \beta]$ is the one obtained from $A$ by keeping only those entries at the crossing of the rows indicated by $\boldsymbol{\alpha}$ and columns by $\boldsymbol{\beta}$. When $\boldsymbol\alpha=\boldsymbol{\beta}$, we denote the submatrix by $A[\boldsymbol \alpha]$ and call it a principal submatrix. A contiguous subset is of the form $\boldsymbol\alpha=\{ \alpha_1, \alpha_2,\dots, \alpha_r\}$,
%%\subset\{0,1,\dots,N-1\}$, i.e., 
%$ \alpha_k= \alpha_1+k-1$ for $k\in\{1,\dots,r\})$. Leading principal submatrices are $A[\{1,2,\dots,j\}]$. The determinants of these submatrices are called minors (obviously we need $r=s$), principal minors and leading principal minors, respectively. The complementary submatrix $A(\boldsymbol \alpha,\boldsymbol \beta)$ is obtained from $A$ by removing the rows indicated by $\boldsymbol{\alpha}$ and columns by $\boldsymbol{\beta}$. When $\boldsymbol\alpha=\boldsymbol{\beta}$, we denote the complementary submatrix~$A(\boldsymbol \alpha)$.}
%

In this paper we will be dealing with totally non negative matrices and oscillatory matrices and, consequently, we require of some definitions and properties that we are about to present succinctly.

Totally nonnegative (TN) matrices are those with all their minors nonnegative \cite{Fallat-Johnson,Gantmacher-Krein}, and the set of nonsingular TN matrices is denoted by InTN. Oscillatory matrices \cite{Gantmacher-Krein} are totally nonnegative, irreducible \cite{Horn-Johnson} and nonsingular. Notice that the set of oscillatory matrices is denoted by IITN (irreducible invertible totally nonnegative) in \cite{Fallat-Johnson}. An oscillatory matrix~$A$ is equivalently defined as a totally nonnegative matrix $A$ such that for some $n$ we have that $A^n$ is totally positive (all minors are positive). From Cauchy--Binet Theorem one can deduce the invariance of these sets of matrices under the usual matrix product. Thus, following \cite[Theorem 1.1.2]{Fallat-Johnson} the product of matrices in InTN is again InTN (similar statements hold for TN or oscillatory matrices).

We have the important result:
\begin{teo}[Gantmacher--Krein Criterion] \cite[Chapter 2, Theorem 10]{Gantmacher-Krein}. \label{teo:Gantmacher--Krein Criterion}
 A~totally non negative matrix $A$ is oscillatory if and only if it is nonsingular and the elements at the first subdiagonal and first superdiagonal are positive.
\end{teo}

The following theorems are extracted from \cite{Fallat-Johnson}, see also \cite{Gantmacher-Krein}.

%The first results deals with LU factorization
%\begin{teo}[Whitney]\label{theorem:Whitney} \cite[Theorem 2.2.1]{Fallat-Johnson}
% Let $A=(T_{i,j})$ be such that $T_{k,1}>0$ and $T_{l,1}=0$ for $l\geqslant k+2$. Then , $A$ is totally nonnegative if and only if $L_{k+1}(-\frac{T_{k+1,1}}{T_{k,1}})A$ is nonnegative.
%\end{teo}

%\begin{teo}[Standard elementary bidiagonal factorization]\label{theorem:bidiagonal factorization}\cite[Theorem 2.2.2]{Fallat-Johnson}
% Any invertible totally nonnegative matrix $A\in\R^{N\times N}$ can be factorized as
% \begin{multline}\label{eq:fac}
% A= \big(L_N(\mathscr l_k)L_{N-1}(\mathscr l_{k-1})\cdots L_2(\mathscr l_{k-N+2})\big)\big(L_N(\mathscr l_{k-N+1})\cdots L_3(\mathscr l_{k-2N+4})\big)\cdots L_N(\mathscr l_1) \\D U_N(u_1) \big(U_{N-1}(u_2)U_{N}(u_3)\big)\cdots\big(U_2(u_{k-N+2})\cdots U_{N-1}(u_{k-1})U_N(u_k))\big),
% \end{multline}
% with $k=\binom{N}{2}$, $\mathscr l_j,u_j\geqslant 0$ for all $i,j\in\{1,\dots,k\}$, and $D=\diag(d_1\dots,d_N)$, with all $d_i>0$.
% A matrix $A\in\R^{N\times N}$ is totally positive if it can be written as in \eqref{eq:fac} where now $\mathscr l_j,u_j> 0$ for all $i,j\in\{1,\dots,k\}$ and $D$ is positive.
%\end{teo}
%One finds the following spectral properties\cite[Theorem 5.2.1]{Fallat-Johnson}:
\begin{teo}[Eigenvalue] \cite[Theorem 5.2.1]{Fallat-Johnson}\label{teo:eigennvalues_TN} 
Given an oscillatory matrix $A\in\R^{N\times N}$~ the eigenvalues of $A$ are $N$ distinct positive numbers. 
\end{teo}

We also need to introduce the following notation.
We define the total sign variation of a totally nonzero vector (no entry of the vector $u$ is zero) as $v(u)=\text{cardinal}\{i\in\{1,\dots,n-1\} : u_iu_{i+1}<0\}$.
 For a general vector $u\in\R^n$ we define $v_m(u)$ ($v_M(u)$) as the minimum (maximum) value $v(y)$ among all totally nonzero vectors $y$ that coincide with $u$ in its nonzero entries. For $v_m(u)=v_M(u)$ we write $v(u)\coloneq v_m(u)=v_M(u)$.

\begin{teo}[Eigenvectors] 
\label{teo:eigenvectorII}
 Let $A\in\R^{N\times N}$ be an oscillatory matrix, and $u^{(k)}$ ($w^{(k)}$) the right (left) eigenvector corresponding to $ \lambda_k$, the $k$-th largest eigenvalue of $A$. Then
\begin{enumerate}
 \item \cite[Theorem 5.3.3]{Fallat-Johnson} We have $v_m(u^{(k)})=v_M(u^{(k)})= v(u^{(k)})=k-1$ ($v_m(w^{(k)})=v_M(w^{(k)})= v(w^{(k)})=k-1$). Moreover, the first and last entry of $u^{(k)}$ ($w^{(k)})$ are nonzero, and $u^{(1)}$ and $u^{(N)}$ ($w^{(1)}$ and $w^{(N)}$) are totally nonzero; the other vectors may have a zero entry.
 \item From Perron--Frobenius theorem we know that $u^{(1)}$ ($w^{(1)}$) can be chosen to be entrywise positive, and that the other eigenvectors $u^{(k)}$ ($w^{(k)}$), $k=2,\dots,n$ have at least one entry sign change. In fact, $u^{(N)}$ ($w^{(N)}$) strictly alternates the sign of its entries.
% \item \cite[Theorem 5.3.5]{Fallat-Johnson} The $k$ nodes of $u^{(k+1)}$ ($w^{(k+1)}$) alternate with the $k-1$ nodes of $u^{(k)}$ ($w^{(k)}$).

\end{enumerate}
\end{teo}

 In this paper we prove a spectral Favard theorem for a bounded banded Hessenberg matrix that admits a positive bidiagonal factorization (being therefore an oscillatory matrix) and we show the existence of spectral positive measures for which the recursion polynomials happen to be multiple orthogonal polynomials. Thus, we find sufficient conditions for the existence of spectral positive measures. 
However, the Jacobi--Piñeiro's multiple orthogonal polynomials show that these are sufficient but not necessary conditions. Also a multiple Gauss quadrature is proven and degrees of precision are determined explicitly. 
%We present this result for simplicity the case of two measures but clearly the same idea holds for any finite numbers of measures.
 
 Our motivation for this work initiated in our previous research on Markov chains beyond birth and death and multiple orthogonal polynomials, see the papers \cite{nuestro2,nuestro1}, in where the transition matrices were Hessenberg matrices with more than three diagonals. In those papers we studied two specific cases, in \cite{nuestro2} we dealt with the Jacobi--Piñeiro multiple orthogonal polynomials, see \cite[Section 23.3.2]{Ismail} and in \cite{nuestro1} the hypergeometric multiple orthogonal polynomials recently presented in \cite{Lima-Loureiro}. 
 
 However, the approach of those two papers was the reverse to that initially presented by Karlin and McGregor in their seminal paper \cite{Karlin-McGregor}. From the recursion matrix of the given examples of multiple orthogonal polynomials we gave a procedure to find a transition matrix and studied the corresponding Markov chains. While in \cite{Karlin-McGregor}, for birth and death Markov chains, a tridiagonal transition matrix was presented and the previously mentioned spectral approach was used to construct a set of {random walk orthogonal polynomials} and give spectral representation formulas in terms of the corresponding polynomials for different important probabilistic objects. Indeed, this was a very important question, pointed to the authors by Alberto Grünbaum, that deserved an answer. What about the spectral theory for our examples of {random walk multiple orthogonal polynomials}? Our answer to this question is the spectral Favard theorem we present in this paper for bounded banded Hessenberg matrices that admit a positive bidiagonal factorization. Let us notice that all the cases we have mentioned, the tridiagonal case, and the Jacobi--Piñeiro in the semi-band and hypergeometric cases fit in this subset of oscillatory matrices.
 Consequently, using this new tool we start as Karlin and McGregor from the transition matrix, assuming that has a positive stochastic bidiagonal factorization and derive from it the spectral representation of the Markov chain.
 
The paper is organized is two parts. In the first part, that is divided in two sections, we will present the main result of the paper, a spectral Favard theorem for bounded banded lower Hessenberg matrices which admit a positive bidiagonal factorization, and in the second part we present the application of this spectral Favard theorem for Markov chains with $(p+2)$-diagonal transition matrices, i.e. beyond birth and death for $p>1$, that admit a positive stochastic bidiagonal factorization.

\section{Banded Hessenberg matrices and recursion polynomials}

\subsection{Recursion polynomials of type I and II for banded Hessenberg matrices}
Let us consider the following banded monic lower Hessenberg semi-infinite matrix
\begin{align}\label{eq:monic_Hessenberg}
 T&=
% \left[\begin{NiceMatrix}
% T_{0,0} & 1& 0 & \Cdots &\\
% T_{1,0} & T_{1,1} & 1&\ddots&& \\
% \Vdots & \Ddots & \Ddots & \Ddots \\
% T_{p,0}& \Cdots & &T_{p,p}& 1 &\\
% \Vdots & \Ddots & \Ddots & \Ddots \\
%0 & T_{n,n-p}& \Cdots & & T_{n,n} &1\\
%\Vdots
% & 
%\Ddots 
% &\Ddots&&\Ddots&\Ddots
% \end{NiceMatrix}\right], 
\begin{bNiceMatrix}%[columns-width = .5cm]
 T_{0,0} & 1& 0 & \Cdots& &&&\phantom{X}\\
 T_{1,0} & T_{1,1} & 1&& &&&\phantom{X}\\
 \Vdots & & & &&&&\phantom{X}\\%[5pt]
 T_{p,0}& & & & &&&\phantom{X}\\%[10pt]
 0&&&&&&&\phantom{X}\\
 \Vdots & & &&&&&\phantom{X}\\
 \phantom{X} &\phantom{X} &\phantom{X}&\phantom{X}&\phantom{X}& \phantom{X}& \phantom{X}&\phantom{X}
\CodeAfter\line{4-1}{7-4} \line{5-1}{7-3}\line{1-3}{5-8}\line{2-3}{6-8}\line{2-2}{7-8}
%\SubMatrix[{1-1}{6-8}]
\end{bNiceMatrix}
\end{align}
and assume that $T_{n,n-p}\not =0, n=p, p+1, \dots$.
The term \emph{monic} refers to the normalization to unity in the first superdiagonal,

\begin{defi}[Recursion polynomials of type II]\label{def:typeII}
The type II recursion vector of polynomials
\begin{align*}
B(x)&=
\left[\begin{NiceMatrix}[columns-width = auto]
B_0(x) & 
B_1(x) & \Cdots 
\end{NiceMatrix}\right]^\top, & \deg B_n=n,
\end{align*}
is determined by the following eigenvalue equation
\begin{align*}%\label{eq:recurrence_J}
TB(x)=x B(x).
\end{align*}
Uniqueness is ensured by taking as initial condition $B_0=1$.
We call the components $B_n$ type II recursion polynomials. 
%One obtains that
%$B_1=x- T_{0,0}$, $B_2=(x-T_{1,1})(x-T_{0,0})-T_{1,0}$, and 
Higher degree recursion polynomials are constructed by means of the $(p+2)$-term recurrence relation
\begin{align}\label{eq:recurrence_B}
B_{n+1}&=(x-T_{n,n})B_n-T_{n,n-1}B_{n-1}-\dots -T_{n,n-p} B_{n-p}, & n&\in\N_0, 
\end{align}
where $T_{n,j}=0$ when $j<0$.
\end{defi}

\begin{defi}[Recursion polynomials of type I]\label{def:typeI}
Dual to the polynomial vector $B(x)$ we consider the following polynomial dual vectors
\begin{align*}
A^{(a)}(x)&=\left[\begin{NiceMatrix}
 A^{(a)}_0(x) & A^{(a)}_1(x)& \Cdots
 \end{NiceMatrix}\right]
, 
&a&\in\{ 1, \dots, p\},
\end{align*}
that are left eigenvectors of the semi-infinite matrix $T$, i.e.,
\begin{align*}
A^{(a)}(x) T &=x A^{(a)}(x), &a&\in\{ 1, \dots, p\}.
\end{align*} 
The initial conditions, that determine these polynomials uniquely, are taken as
\begin{align}
\label{eq:initcondtypeI}
\begin{cases}
A^{(1)}_0=1 , \\
A^{(1)}_1= \nu^{(1)}_1 , \\
 \hspace{.895cm} \vdots \\
A^{(1)}_{p-1}=\nu^{(1)}_{p-1} ,
\end{cases}
&&
\begin{cases}
A^{(2)}_0=0 , \\
A^{(2)}_1= 1 , \\
A^{(2)}_2= \nu^{(2)}_2 , \\
 \hspace{.895cm} \vdots \\
A^{(2)}_{p-1}=\nu^{(2)}_{p-1} ,
\end{cases}
&& \cdots &&
\begin{cases}
A^{(p)}_0 =0 , \\
 \hspace{.915cm} \vdots \\
A^{(p)}_{p-2} = 0 , \\
A^{(p)}_{p-1} = 1,
\end{cases}
\end{align}
with $\nu^{(i)}_{j}$ being arbitrary constants. 
We will define the initial condition matrix 
\begin{align}
\label{eq:mic}
\nu\coloneq \begin{bNiceMatrix}%[columns-width=2pt]
 1& 0 & \Cdots& && 0 \\
\nu^{(1)}_1 & 1 & \Ddots&& & \Vdots \\
 \Vdots & \Ddots & \Ddots& && \\
&&&& &\\&&&&&0\\
\nu^{(1)}_{p-1} &\Cdots& && \nu^{(p-1)}_{p-1}& 1 
 \end{bNiceMatrix} .
\end{align}
%Then, from the first the relation
%\begin{align*}
% T_{0,0}A^{(a)}_0+T_{1,0}A^{(a)}_{1}+\dots + T_{p,0}A^{(a)}_{p}&=xA^{(a)}_0, & a&\in\{1,\ldots, p\},
%\end{align*}
%we get
%$A^{(1)}_2=\frac{x}{T_2}-\frac{c_0+b_1\nu}{T_2}$ and
%$A^{(2)}_2=-\frac{b_1}{T_2}$.
The other polynomials in these sequences are determined by the following $(p+2)$-term recursion relation
\begin{align}\label{eq:recursion_dual_A}
A^{(a)}_{n-1} +T_{n,n}A^{(a)}_{n} + \cdots + T_{n+p,n}A^{(a)}_{n+p} = x A^{(a)}_{n}, && n \in\{0,1,\ldots\}, && a \in\{1,\ldots, p\}, && A_{-1}^{(a)} = 0.
\end{align}
\end{defi}

%For example, one finds
%\begin{align*}
%A^{(1)}_3 T_3&=-A^{(1)}_2b_2+A^{(1)}_1(x-c_1)-A^{(1)}_0=
% -b_2\Big(\frac{x}{T_2}-\frac{c_0+b_1\nu}{T_2}\Big)+\nu (x-c_1)-1,\\
%A^{(2)}_3 T_3 &=-A^{(2)}_2b_2+A^{(2)}_1(x-c_1)-A^{(2)}_0=
% b_2\frac{b_1}{T_2}+x-c_1. 
%\end{align*}

%\begin{enumerate}
%\item
%One can check that, $\deg A^{(1)}_{kp+j} = \deg A^{(2)}_{2n+1}=n$, and that 
%$\deg A^{(2)}_{2n-1}=\deg A^{(1)}_{2n}=n-1$.
%
%\item
%Frequently the following notation is used for these recursion polynomials of type I
%\begin{align*}
%A^{(1)}_{k+1,k,\dots,k}&\coloneq A^{(1)}_{kp}, &A^{(2)}_{k+1,k+1,k,\dots,k}&\coloneq A^{(2)}_{kp},& &A^{(p)}_{k+1,k+1,k,\dots,k}&\coloneq A^{(2)}_{kp},&
%A^{(1)}_{n+1,n}&\coloneq A^{(1)}_{2n}, &A^{(2)}_{n,n}&\coloneq A^{(2)}_{2n}.
%\end{align*} 
%\end{enumerate}

Let us introduce the products $H_0=1$, $H_1=T_{p,0}$, $H_2=T_{p,0}T_{p+1,1}$, $H_3=T_{p,0}T_{p+1,1}T_{p+2,2}$, and in general 
\begin{align*}
H_n\coloneq T_{p,0}T_{p+1,1}\cdots T_{n-1+p,n-1}.
\end{align*}
Hence, we have for the entries in the lowest nonzero subdiagonal
\begin{align*}
 T_{n+p,n}=\frac{H_{n+1}}{H_n}.
\end{align*}
We also use the notation
\begin{align*}
 h_n\coloneq(-1)^{(p-1)n}H_n,
\end{align*}
so that $ h_{n+1}= (-1)^{p-1} h_{n} T_{n+p,n} $,
$ n \in\N_0$,
$h_0 \coloneq 1$.

These dual polynomials are, in a sense, more fundamental to the banded Hessenberg matrix~$T$ than the sequence $\{B_n\}_{n=0}^\infty$, as we can reconstruct this sequence with determinantal formulas involving the sequences $\big\{A_n^{(1)}\big\}_{n=0}^\infty,\dots, \big\{A_n^{(p)}\big\}_{n=0}^\infty$.
\begin{pro}\label{pro:determinatal_B_A} \cite[Lemma 2.4]{Coussement-VanAssche}
 For $n\in\N_0$, the following formula
 \begin{align}\label{eq:B_n-A_n}
 B_n&=h_n\begin{vNiceMatrix}
 A^{(1)}_n & \Cdots & A^{(p)}_n \\[2pt]
 \Vdots & \Ddots & \Vdots \\[2pt]
 A^{(1)}_{n+p-1} & \Cdots & A^{(p)}_{n+p-1}
 \end{vNiceMatrix},
%, 
%&
% h_{n+1}\coloneq (-1)^{p-1} h_{n} T_{n+p,n} ,
% && n = 0,1,\ldots,
% && h_0 \coloneq 1 ,
 \end{align}
holds.
\end{pro}

\begin{proof}
For $n=0$, we have
\begin{align*}
B_0
=1
=h_0
\begin{vNiceMatrix}
 A^{(1)}_0 & \Cdots & A^{(p)}_0 \\[2pt]
 \Vdots & \Ddots & \Vdots \\[2pt]
 A^{(1)}_{p-1} & \Cdots & A^{(p)}_{p-1}
 \end{vNiceMatrix}
=
h_0
\begin{vNiceMatrix}
	1& 0 & \Cdots& && 0 \\
	\nu^{(1)}_1 & 1 & \Ddots&& & \Vdots \\
	\Vdots & \Ddots & \Ddots& && \\
	&&&& &\\&&&&&0\\
	\nu^{(1)}_{p-1} &\Cdots& && \nu^{(p-1)}_{p-1}& 1 
 \end{vNiceMatrix}
 ,
\end{align*}
while $n=1$ we find
\begin{align*}
 B_1&=
 (x-T_{0,0}) B_0 =
 (x-T_{0,0}) h_0 
 \begin{vNiceMatrix}
 A^{(1)}_0 & \Cdots & A^{(p)}_0 \\[2pt]
 \Vdots & \Ddots & \Vdots \\[2pt]
 A^{(1)}_{p-1} & \Cdots & A^{(p)}_{p-1}
 \end{vNiceMatrix}
 =
h_0 
 \begin{vNiceMatrix}
 (x-T_{0,0}) A^{(1)}_0 & \Cdots & (x-T_{0,0}) A^{(p)}_0 \\[2pt]
 \Vdots & \Ddots & \Vdots \\[2pt]
 A^{(1)}_{p-1} & \Cdots & A^{(p)}_{p-1}
 \end{vNiceMatrix} \\
& = h_0 
 \begin{vNiceMatrix}
 T_{p,0} A^{(1)}_p & \Cdots & T_{p,0} A^{(p)}_p \\[2pt]
 A^{(1)}_{1} & \Cdots & A^{(p)}_{1} \\[2pt]
 \Vdots & \Ddots & \Vdots \\[2pt]
 A^{(1)}_{p-1} & \Cdots & A^{(p)}_{p-1}
 \end{vNiceMatrix}
 = h_0 T_{p,0} (-1)^{p-1}
 \begin{vNiceMatrix}
 A^{(1)}_{1} & \Cdots & A^{(p)}_{1} \\[2pt]
 \Vdots & \Ddots & \Vdots \\[2pt]
 A^{(1)}_{p-1} & \Cdots & A^{(p)}_{p-1} \\[2pt]
 A^{(1)}_p & \Cdots & A^{(p)}_p
 \end{vNiceMatrix}
 = h_1
 \begin{vNiceMatrix}
 A^{(1)}_{1} & \Cdots & A^{(p)}_{1} \\[2pt]
 \Vdots & \Ddots & \Vdots \\[2pt]
%A^{(1)}_{p-1} & \Cdots & A^{(p)}_{p-1} \\[2pt]
 A^{(1)}_p & \Cdots & A^{(p)}_p
 \end{vNiceMatrix}
 .
\end{align*}
Once we have proven the cases $n=0,1$, let us apply induction. Hence, we assume that \eqref{eq:B_n-A_n} holds for all nonnegative integers up to $n$ and proceed to show that it also holds for $n+1$. From the direct recursion relation \eqref{eq:recurrence_B} and our induction assumption we find
\begin{align*}
B_{n+1}=
 (x-T_{n,n}) h_n\begin{vNiceMatrix}[small]
 A^{(1)}_n & \Cdots & A^{(p)}_n \\[2pt]
 \Vdots & \Ddots & \Vdots \\[2pt]
 A^{(1)}_{n+p-1} & \Cdots & A^{(p)}_{n+p-1}
 \end{vNiceMatrix}
 - T_{n,n-1}
h_{n-1}\begin{vNiceMatrix}[small]
 A^{(1)}_{n-1} & \Cdots & A^{(p)}_{n-1} \\[2pt]
 \Vdots & \Ddots & \Vdots \\[2pt]
 A^{(1)}_{n+p-2} & \Cdots & A^{(p)}_{n+p-2}
 \end{vNiceMatrix}
 - \cdots -
 T_{n,n-p}
h_{n-p}\begin{vNiceMatrix}[small]
 A^{(1)}_{n-p} & \Cdots & A^{(p)}_{n-p} \\[2pt]
 \Vdots & \Ddots & \Vdots \\[2pt]
 A^{(1)}_{n-1} & \Cdots & A^{(p)}_{n-1}
 \end{vNiceMatrix}.
\end{align*}
We notice that, using the recursion \eqref{eq:recursion_dual_A}, we get
\begin{align*}
 (x-T_{n,n}) h_n\begin{vNiceMatrix}[small]
 A^{(1)}_n & \Cdots & A^{(p)}_n \\[2pt]
 \Vdots & \Ddots & \Vdots \\[2pt]
 A^{(1)}_{n+p-1} & \Cdots & A^{(p)}_{n+p-1}
 \end{vNiceMatrix}
 &=
 h_n\begin{vNiceMatrix}[small]
 (x-T_{n,n}) A^{(1)}_n & \Cdots & (x-T_{n,n}) A^{(p)}_n \\[2pt]
 A^{(1)}_{n+1} & \Cdots & A^{(p)}_{n+1} \\[2pt]
 \Vdots & \Ddots & \Vdots \\[2pt]
 A^{(1)}_{n+p-1} & \Cdots & A^{(p)}_{n+p-1}
 \end{vNiceMatrix} =
 h_n\begin{vNiceMatrix}[small]
 A^{(1)}_{n-1} + T_{n+p,n} A^{(1)}_{n+p} & \Cdots & A^{(p)}_{n-1} + T_{n+p,n} A^{(p)}_{n+p} \\[2pt]
 A^{(1)}_{n+1} & \Cdots & A^{(p)}_{n+1} \\[2pt]
 \Vdots & \Ddots & \Vdots \\[2pt]
 A^{(1)}_{n+p-1} & \Cdots & A^{(p)}_{n+p-1}
 \end{vNiceMatrix}
 \\
 & =
h_n T_{n+p,n} (-1)^{p-1} \begin{vNiceMatrix}[small]
 A^{(1)}_{n+1} & \Cdots & A^{(p)}_{n+1} \\[2pt]
 \Vdots & \Ddots & \Vdots \\[2pt]
 A^{(1)}_{n+p-1} & \Cdots & A^{(p)}_{n+p-1} \\[2pt]
A^{(1)}_{n+p} & \Cdots & A^{(p)}_{n+p} \\[2pt]
 \end{vNiceMatrix}
 +
 h_n \begin{vNiceMatrix}[small]
 A^{(1)}_{n-1} & \Cdots & A^{(p)}_{n-1} \\[2pt]
 A^{(1)}_{n+1} & \Cdots & A^{(p)}_{n+1} \\[2pt]
 \Vdots & \Ddots & \Vdots \\[2pt]
 A^{(1)}_{n+p-1} & \Cdots & A^{(p)}_{n+p-1}
 \end{vNiceMatrix} \\
 & = h_{n+1}
 \begin{vNiceMatrix}[small]
 A^{(1)}_{n+1} & \Cdots & A^{(p)}_{n+1} \\[2pt]
 \Vdots & \Ddots & \Vdots \\[2pt]
% A^{(1)}_{n+p-1} & \Cdots & A^{(p)}_{n+p-1} \\[2pt]
A^{(1)}_{n+p} & \Cdots & A^{(p)}_{n+p} \\[2pt]
 \end{vNiceMatrix}
 +
 h_{n-1}
 \begin{vNiceMatrix}[small]
 T_{n-1+p,n-1} A^{(1)}_{n+p-1} & \Cdots & T_{n-1+p,n-1} A^{(p)}_{n+p-1} \\[2pt]
 A^{(1)}_{n-1} & \Cdots & A^{(p)}_{n-1} \\[2pt]
 A^{(1)}_{n+1} & \Cdots & A^{(p)}_{n+1} \\[2pt]
 \Vdots & \Ddots & \Vdots \\[2pt]
 A^{(1)}_{n+p-2} & \Cdots & A^{(p)}_{n+p-2}
 \end{vNiceMatrix} .
\end{align*}
We continue with
\begin{multline*}
%\begin{align*}\hspace*{-1cm}
 \begin{vNiceMatrix}[small]
 T_{n-1+p,n-1} A^{(1)}_{n+p-1} & \Cdots & T_{n-1+p,n-1} A^{(p)}_{n+p-1} \\[2pt]
 A^{(1)}_{n-1} & \Cdots & A^{(p)}_{n-1} \\[2pt]
 A^{(1)}_{n+1} & \Cdots & A^{(p)}_{n+1} \\[2pt]
 \Vdots & \Ddots & \Vdots \\[2pt]
 A^{(1)}_{n+p-2} & \Cdots & A^{(p)}_{n+p-2}
 \end{vNiceMatrix}
 \\ =
 \begin{vNiceMatrix}[small]
 - T_{n,n-1} A^{(1)}_{n} - A^{(1)}_{n-2} & \Cdots & - T_{n,n-1} A^{(p)}_{n} - A^{(p)}_{n-2} \\[2pt]
 A^{(1)}_{n-1} & \Cdots & A^{(p)}_{n-1} \\[2pt]
 A^{(1)}_{n+1} & \Cdots & A^{(p)}_{n+1} \\[2pt]
 \Vdots & \Ddots & \Vdots \\[2pt]
 A^{(1)}_{n+p-2} & \Cdots & A^{(p)}_{n+p-2}
 \end{vNiceMatrix}
= T_{n,n-1} 
 \begin{vNiceMatrix}[small]
 A^{(1)}_{n-1} & \Cdots & A^{(p)}_{n-1} \\[2pt]
 A^{(1)}_{n} & \Cdots & A^{(p)}_{n} \\[2pt]
 A^{(1)}_{n+1} & \Cdots & A^{(p)}_{n+1} \\[2pt]
 \Vdots & \Ddots & \Vdots \\[2pt]
 A^{(1)}_{n+p-2} & \Cdots & A^{(p)}_{n+p-2} \\[2pt]
 \end{vNiceMatrix}
 -
 \begin{vNiceMatrix}[small]
 A^{(1)}_{n-2} & \Cdots & A^{(p)}_{n-2} \\[2pt]
 A^{(1)}_{n-1} & \Cdots & A^{(p)}_{n-1} \\[2pt]
 A^{(1)}_{n+1} & \Cdots & A^{(p)}_{n+1} \\[2pt]
 \Vdots & \Ddots & \Vdots \\[2pt]
 A^{(1)}_{n+p-2} & \Cdots & A^{(p)}_{n+p-2}
 \end{vNiceMatrix} .
%\end{align*}
\end{multline*}
Now, using recursively these kind of arguments we deduce that
\begin{multline*}
-h_{n-1}
 \begin{vNiceMatrix}[small]
 A^{(1)}_{n-2} & \Cdots & A^{(p)}_{n-2} \\[2pt]
 A^{(1)}_{n-1} & \Cdots & A^{(p)}_{n-1} \\[2pt]
 A^{(1)}_{n+1} & \Cdots & A^{(p)}_{n+1} \\[2pt]
 \Vdots & \Ddots & \Vdots \\[2pt]
 A^{(1)}_{n+p-2} & \Cdots & A^{(p)}_{n+p-2}
 \end{vNiceMatrix}
 =
T_{n,n-2} h_{n-2} 
 \begin{vNiceMatrix}[small]
 A^{(1)}_{n-2} & \Cdots & A^{(p)}_{n-2} \\[2pt]
% A^{(1)}_{n-1} & \Cdots & A^{(p)}_{n-1} \\[2pt]
% A^{(1)}_{n} & \Cdots & A^{(p)}_{n} \\[2pt]
 \Vdots & \Ddots & \Vdots \\[2pt]
 A^{(1)}_{n+p-3} & \Cdots & A^{(p)}_{n+p-3} \\[2pt]
 \end{vNiceMatrix}
 +\cdots\\
 +T_{n,n-p+2} h_{n-p+2}
 \begin{vNiceMatrix}[small]
 A^{(1)}_{n-p+2} & \Cdots & A^{(p)}_{n-p+2} \\[2pt]
% A^{(1)}_{n-1} & \Cdots & A^{(p)}_{n-1} \\[2pt]
% A^{(1)}_{n} & \Cdots & A^{(p)}_{n} \\[2pt]
 \Vdots & \Ddots & \Vdots \\[2pt]
 A^{(1)}_{n+1} & \Cdots & A^{(p)}_{n+1} \\[2pt]
 \end{vNiceMatrix}
 + (-1)^{p-2}
h_{n-p+2}
\begin{vNiceMatrix}[small]
 A^{(1)}_{n-p+1} & \Cdots & A^{(p)}_{n-p+1} \\[2pt]
% A^{(1)}_{n-1} & \Cdots & A^{(p)}_{n-1} \\[2pt]
% A^{(1)}_{n} & \Cdots & A^{(p)}_{n} \\[2pt]
 \Vdots & \Ddots & \Vdots \\[2pt]
 A^{(1)}_{n-1} & \Cdots & A^{(p)}_{n-1} \\[2pt]
 A^{(1)}_{n+1} & \Cdots & A^{(p)}_{n+1} \\[2pt]
 \end{vNiceMatrix}
\end{multline*}
and recalling the recursion $A^{(a)}_{n-p} + \cdots +T_{n,n-p+1} A^{(a)}_{n} +T_{n+1,n-p+1} A^{(a)}_{n+1}=x A^{(a)}_{n-p+1}$
we get 
\begin{multline*}
-h_{n-1}
 \begin{vNiceMatrix}[small]
 A^{(1)}_{n-2} & \Cdots & A^{(p)}_{n-2} \\[2pt]
 A^{(1)}_{n-1} & \Cdots & A^{(p)}_{n-1} \\[2pt]
 A^{(1)}_{n+1} & \Cdots & A^{(p)}_{n+1} \\[2pt]
 \Vdots & \Ddots & \Vdots \\[2pt]
 A^{(1)}_{n+p-2} & \Cdots & A^{(p)}_{n+p-2}
 \end{vNiceMatrix} =
T_{n,n-2} h_{n-2} 
 \begin{vNiceMatrix}[small]
 A^{(1)}_{n-2} & \Cdots & A^{(p)}_{n-2} \\[2pt]
% A^{(1)}_{n-1} & \Cdots & A^{(p)}_{n-1} \\[2pt]
% A^{(1)}_{n} & \Cdots & A^{(p)}_{n} \\[2pt]
 \Vdots & \Ddots & \Vdots \\[2pt]
 A^{(1)}_{n+p-3} & \Cdots & A^{(p)}_{n+p-3} \\[2pt]
 \end{vNiceMatrix}
 +\cdots
 \\ 
 +T_{n,n-p+2} \uppi_{n-p+2}
 \begin{vNiceMatrix}[small]
 A^{(1)}_{n-p+2} & \Cdots & A^{(p)}_{n-p+2} \\[2pt]
% A^{(1)}_{n-1} & \Cdots & A^{(p)}_{n-1} \\[2pt]
% A^{(1)}_{n} & \Cdots & A^{(p)}_{n} \\[2pt]
 \Vdots & \Ddots & \Vdots \\[2pt]
 A^{(1)}_{n+1} & \Cdots & A^{(p)}_{n+1} \\[2pt]
 \end{vNiceMatrix}
 +
(-1)^{p-1}\uppi_{n-p+1}
\begin{vNiceMatrix}[small]
 A^{(1)}_{n-p} + T_{n,n-p+1} A^{(1)}_{n}& \Cdots & A^{(p)}_{n-p} + T_{n,n-p+1} A^{(p)}_{n} \\[2pt]
% A^{(1)}_{n-1} & \Cdots & A^{(p)}_{n-1} \\[2pt]
% A^{(1)}_{n} & \Cdots & A^{(p)}_{n} \\[2pt]
 \Vdots & \Ddots & \Vdots \\[2pt]
 A^{(1)}_{n-1} & \Cdots & A^{(p)}_{n-1} \\[2pt]
 \end{vNiceMatrix}.
\end{multline*} 
%\enlargethispage{.25cm}
Finally, calculating the last determinant we get
\begin{multline*}
(-1)^{p-1}h_{n-p+1}
\begin{vNiceMatrix}[small]
 A^{(1)}_{n-p} + T_{n,n-p+1} A^{(1)}_{n}& \Cdots & A^{(p)}_{n-p} + T_{n,n-p+1} A^{(p)}_{n} \\[2pt]
% A^{(1)}_{n-1} & \Cdots & A^{(p)}_{n-1} \\[2pt]
% A^{(1)}_{n} & \Cdots & A^{(p)}_{n} \\[2pt]
 \Vdots & \Ddots & \Vdots \\[2pt]
 A^{(1)}_{n-1} & \Cdots & A^{(p)}_{n-1} \\[2pt]
 \end{vNiceMatrix} \\
=
h_{n-p+1} T_{n,n-p+1}
\begin{vNiceMatrix}[small]
A^{(1)}_{n-p+1} & \Cdots & A^{(p)}_{n-p+1} \\[2pt]
% A^{(1)}_{n} & \Cdots & A^{(p)}_{n} \\[2pt]
 \Vdots & \Ddots & \Vdots \\[2pt]
 A^{(1)}_{n-1} & \Cdots & A^{(p)}_{n-1} \\[2pt]
A^{(1)}_{n}& \Cdots & A^{(p)}_{n} \\[2pt]
 \end{vNiceMatrix}
 +h_{n-p+1}T_{n,n-p}
\begin{vNiceMatrix}[small]
 A^{(1)}_{n-p} & \Cdots & A^{(p)}_{n-p} \\[2pt]
% A^{(1)}_{n-1} & \Cdots & A^{(p)}_{n-1} \\[2pt]
% A^{(1)}_{n} & \Cdots & A^{(p)}_{n} \\[2pt]
 \Vdots & \Ddots & \Vdots \\[2pt]
 A^{(1)}_{n-1} & \Cdots & A^{(p)}_{n-1} \\[2pt]
 \end{vNiceMatrix}
\end{multline*}
\end{proof}

\subsection{Discrete multiple orthogonality}
\setcounter{MaxMatrixCols}{15}
\begin{defi}\label{def: Jtruncated}
 Let us denote by $ T^{[N]} \in\R^{(N+1)\times(N+1)}$ the $(N+1)$-th leading principal submatrix of the banded Hessenberg matrix $T$:
 
\begin{align}\label{eq:Hessenberg_truncation}
 T^{[N]}\coloneq
\begin{bNiceMatrix}%[columns-width = .5cm]
 T_{0,0} & 1& 0 & \Cdots& &&&0\\
 T_{1,0} & T_{1,1} & 1&\Ddots& &&&\Vdots\\
 \Vdots & & \Ddots& \Ddots&&&&\\
T_{p,0}& & & & &&&\\[7pt]
0&\Ddots&&&&&\\
 \Vdots &\Ddots & &&&&&0\\
 \Vdots & & &&&&&1 \\
0 & \Cdots&&0&T_{N,N-p}& \Cdots & &T_{N,N} 
 %\CodeAfter\line{4-1}{8-5} \line{5-1}{8-4} 
 \end{bNiceMatrix}.
\end{align}
%Recall that we are taking all $T_{p+j,j} \not =0 ,j=0, 1, \dots$. 
%Let us denote the corresponding leading principal minors~as
%$$\delta^{[N]}\coloneq \det T^{[N]}.$$
\end{defi}
We define the vector $ B^{\langle N\rangle} \in\R^{N+1}[x]$ by 
\begin{align*}
 B^{\langle N\rangle}&\coloneq 
 \begin{bNiceMatrix}[columns-width = auto]
 B_0 
 & B_1 
 & \Cdots 
 & B_N
 \end{bNiceMatrix}^\top .
\end{align*}
In what follows the $e_n$ denotes the semi-infinite vector with $0$ all its entries but for a $1$ in its $n$-th~entry.

\begin{pro}[Determinantal expression]\label{pro:determintal_second_kind}
% \begin{enumerate} \item
For the type II recursion polynomials we have the determinantal expression
 \begin{align}\label{eq:det_B}
 B_{N+1}(x)&=\det\big(x I_{N+1}- T^{[N]}\big).
% =\begin{vNiceMatrix}%[columns-width = 10pt]
%x- T_{0,0} & -1& 0 & \Cdots& &&&0\\
%-T_{1,0} & x-T_{1,1} & -1&\Ddots& &&&\Vdots\\
%\Vdots & & \Ddots& \Ddots&&&&\\
%-T_{p,0}& \Ddots& & & &&&\\[5pt]
%0&&&&&&\\
%\Vdots & \Ddots& &&&&&0\\
%\Vdots & & &&&&&-1 \\
%0 & \Cdots&&0&-T_{N,N-p}& \Cdots & &x-T_{N,N} 
%%\CodeAfter\line{4-1}{8-5} \line{5-1}{8-4} 
% \end{vNiceMatrix}.
 \end{align}
Hence, they are the characteristic polynomials of the leading principal submatrices $T^{[N]}$.
% \item For the recursion polynomials of type II of the second kind, $B^{(1)}_{N+1} $ and $B^{(2)}_{N+1} $, we have the following adjugate and determinantal expressions 
% \begin{align*}
% B^{(1)}_{N+1} &= e_1^\top \operatorname{adj}\big(x I_{N+1}-T^{[N]}\big)e_1=\begin{vNiceMatrix}[columns-width = 10pt]
% x-c_1 & -1&0&\Cdots&&&0\\[-3pt]
% -b_2&x-c_2&-1&\Ddots&&&\Vdots\\
% -T_3&-b_3&x-c_3&-1&&&\\
% 0&-T_4&-b_4&x-c_4&-1&&\\
% \Vdots&\Ddots&\Ddots&\Ddots&\Ddots& \Ddots&0\\
% &&&-T_{N-1}&-b_{N-1}&x-c_{N-1}&-1\\
% 0&\Cdots&&0&-T_{N}&-b_{N}&x-c_N
% \end{vNiceMatrix},\\
% B^{(2)}_{N+1} &= e_1^\top \operatorname{adj}\big(x I_{N+1}-T^{[N]}\big)^\top (e_2 - \nu e_1)
% =b^{(1)}_{N+1}-\nu B^{(1)}_{N+1} , \\ b^{(1)}_{N+1}&= e_{1}^\top \operatorname{adj}\big(x I_{N+1}-T^{[N]}\big)^\top e_{2}=
% \begin{vNiceMatrix}[columns-width = 10pt]
% x-c_2 & -1&0&\Cdots&&&0\\[-3pt]
% -b_3&x-c_3&-1&\Ddots&&&\Vdots\\
% -T_4&-b_4&x-c_4&-1&&&\\
% 0&-T_5&-b_5&x-c_5&-1&&\\
% \Vdots&\Ddots&\Ddots&\Ddots&\Ddots& \Ddots&0\\
% &&&-T_{N-1}&-b_{N-1}&x-c_{N-1}&-1\\
% 0&\Cdots&&0&-T_{N}&-b_{N}&x-c_N
% \end{vNiceMatrix}.
% \end{align*}
% and we take $b^{(1)}_{2}:=1$.
% \end{enumerate}
\end{pro}
\begin{proof}
 All these determinants satisfy the recurrence \eqref{eq:recurrence_B}. For that aim we expand the determinant along the last column, and the second determinant that appears after the expression is expanded along the last row. Notice also that the first determinant coincide with $B_1(x)$, so the statement holds.
% $\deg B_{N+1}^{(1)}=\deg B_{N+1}^{(2)}=N$. One can check that the \emph{initial conditions}, that is the first determinants coincide with the expressions found in \eqref{eq:first_typeII}, \eqref{eq:first_second_kind(1)} and \eqref{eq:first_second_kind(2)} . Consequently, these determinantal expressions coincide with the polynomials of type II, and of type II of second kind. 
\end{proof}

We will assume that the eigenvalues of $T^{[N]}$ are simple.
Then, one has the following basic result:

\begin{pro}
 Let $\Big\{\lambda^{[N]}_n\Big\}_{n=1}^{N+1}$ be the set of zeros of the polynomial $B_{N+1}(x)$. Then, $\lambda^{[N]}_n$ are the eigenvalues of $ T^{[N]}$ with associated right eigenvectors
 \begin{align*}
 u^{\langle N\rangle}_n&\coloneq 
 B^{\langle N\rangle}
 \big|_{x=\lambda^{[N]}_{n}}, &n&\in\{1,\dots, N+1\}.
 \end{align*}
% ; i.e.,
%\begin{align}
% \label{eq:eigenvalue}
% T^{[N]} B^{[N]}\big(\lambda^{[N]}_{n}\big)=\lambda^{[N]}_{n} B^{[N]}\big(\lambda^{[N]}_{n}\big).
%\end{align}
\end{pro}
\begin{proof}
 According to \eqref{eq:det_B} the eigenvalues of $T^{[N]}$ are the zeros of $B_{N+1}$. 
From
 \begin{align}\label{eq:truncated_Hessenberg}
 T^{[N]} B^{\langle N\rangle}(x)+\begin{bNiceMatrix}
 0\\\Vdots\\0\\B_{N+1} (x)
 \end{bNiceMatrix}=x B^{\langle N\rangle}(x).
 \end{align}
We get that $u^{\langle N\rangle}_n$ is a right eigenvector. 
\end{proof}

\begin{defi}
 Let us introduce the polynomials
\begin{align}\label{eq:QNn}
 Q_{n,N}\coloneq\begin{vNiceMatrix}
 A^{(1)}_{n} & \Cdots & A^{(p)}_{n} \\[2pt]
 A^{(1)}_{N+1} & \Cdots & A^{(p)}_{N+1} \\[2pt]
 \Vdots & \Ddots & \Vdots \\[2pt]
 A^{(1)}_{N+p-1} & \Cdots & A^{(p)}_{N+p-1}
 \end{vNiceMatrix},
\end{align}
and the semi-infinite row matrix
$ Q_N\coloneq\left[\begin{NiceMatrix}
 Q_{0,N} &Q_{1,N} &\Cdots
 \end{NiceMatrix}\right]$.
Additionally, we consider
\begin{align*}
 Q^{\langle N\rangle}\coloneq \begin{bNiceMatrix}
 Q_{0,N} &Q_{1,N}&\Cdots & Q_{N,N}
 \end{bNiceMatrix}.
\end{align*}
\end{defi}

\begin{pro}\label{pro:properties_Q}
For the polynomials $Q_{n,N}$ we find
 \begin{enumerate}
 \item $ Q_{N+1,N}=\cdots=Q_{N+p-1,N}=0$.
\item $Q_{N,N}=h_N^{-1}B_N$ and $Q_{N+p,N}=(-1)^{p-1}h_{N+1}^{-1}B_{N+1}$.
\item $Q_NT=xQ_N$.
\item \begin{align}\label{eq:dual}
 Q^{\langle N\rangle} T^{[N]}+
 \begin{bNiceMatrix}
 0& \Cdots& 0& T_{N+p,N}Q_{N+p,N}
 \end{bNiceMatrix}&=x Q^{\langle N\rangle}.
\end{align}
\end{enumerate}
\end{pro}
\begin{proof}
 \begin{enumerate}
 \item As $Q_{n,N}$ is the determinant \eqref{eq:QNn} we see that it vanishes whenever two rows are equal, that happens precisely in the indicated cases.
 \item It follows from Proposition \ref{pro:determinatal_B_A}. 
 \item It is a direct consequence of the fact that each row in the determinant in \eqref{eq:QNn} satisfies such recurrence.
 \item It follows from the previous points i) and iii).
 \end{enumerate}
\end{proof}

Now we are ready to give a set of left eigenvectors of the matrix $T^{[N]}$.
\begin{pro}
For $k\in\{1,\dots,N+1\}$, the vectors
\begin{align*}
w^{\langle N\rangle}_k&\coloneq Q^{\langle N\rangle}\big|_{x=\lambda^{[N]}_k}, &k&\in\{1,\dots,N+1\},
\end{align*}
are left eigenvectors of $T^{[N]}$.
\end{pro}
\begin{proof}
Properties ii) and iv) in Proposition \ref{pro:properties_Q} and an evaluation at $\lambda^{[N]}_k$ leads to the result.
\end{proof}
Now we now give some Christoffel--Darboux formulas. Notice that \eqref{eq:CD2} and \eqref{eq:CD2_confluent} appeared in 
\cite[Theorem 2.5]{Coussement-VanAssche}.
\begin{pro}[Christoffel--Darboux type formulas]\label{theorem:CD}
\begin{enumerate}
% \item For the truncated polynomials the following Christoffel--Darboux type relation holds
% \begin{align}\label{eq:CD1}
% \sum_{n=0}^{N}B_{N+1}^{[n+1]}(x)B_{n}(y)=\frac{B_{N+1}(x)-B_{N+1}(y)}{x-y}.
% \end{align}
%\item The following confluent Christoffel--Darboux type formula is satisfied
%\begin{align}\label{eq:CD1_confluent}
% \sum_{n=0}^{N}B_{N+1}^{[n+1]}B_{n}=B'_{N+1}.
%\end{align}
\item For the polynomials $Q_{n,N}$ introduced in \eqref{eq:QNn} we get the following Christoffel--Darboux formula
\begin{align}\label{eq:CD2}
\sum_{n=0}^{N}Q_{n,N}(x)B_n(y)= \frac{1}{h_N}\frac{B_{N+1}(x)B_{N}(y)-B_{N}(x)B_{N+1}(y)}{x-y}.
\end{align}
\item The following confluent Christoffel--Darboux relation is fulfilled 
\begin{align}\label{eq:CD2_confluent}
\sum_{n=0}^{N}Q_{n,N}B_n= \frac{1}{h_N}\big(B'_{N+1}B_{N}-B'_{N}B_{N+1}\big).
\end{align}
\end{enumerate}
\end{pro}
\begin{proof}
%\begin{enumerate}
% \item It follows from \eqref{eq:truncated_Hessenberg} and \eqref{eq:left_eigenvectors} that
% \begin{align*}
% \begin{bNiceMatrix}
% B^{[1]}_{N+1}(x)& \Cdots & B^{[N+1]}_{N+1}(x)
% \end{bNiceMatrix}T^{[N]}+ \begin{bNiceMatrix}
% B^{[0]}_{N+1} (x)& 0&\Cdots & 0
% \end{bNiceMatrix}&=x \begin{bNiceMatrix}
% B^{[1]}_{N+1}(x) & \Cdots & B^{[N+1]}_{N+1}(x)
% \end{bNiceMatrix},\\
% T^{[N]}\begin{bNiceMatrix} B_0(y)\\\Vdots\\
% B_N(y)
% \end{bNiceMatrix}+\begin{bNiceMatrix}
% 0\\\Vdots\\0\\B_{N+1} (y)
%\end{bNiceMatrix}&=y \begin{bNiceMatrix} B_0(y)\\\Vdots\\
% B_N(y)
%\end{bNiceMatrix},
%\end{align*}
%so that
%\begin{multline*}
% \begin{bNiceMatrix}
% B^{[0]}_{N+1} (x)& 0&\Cdots & 0
% \end{bNiceMatrix}\begin{bNiceMatrix} B_0(y)\\\Vdots\\
% B_N(y)
%\end{bNiceMatrix}- \begin{bNiceMatrix}
%B^{[1]}_{N+1}(x)& \Cdots & B^{[N+1]}_{N+1}(x)
%\end{bNiceMatrix}\begin{bNiceMatrix}
%0\\\Vdots\\0\\B_{N+1} (y)
%\end{bNiceMatrix}\\=(x-y ) \begin{bNiceMatrix}
%B^{[1]}_{N+1}(x)& \Cdots & B^{[N+1]}_{N+1}(x)
%\end{bNiceMatrix}\begin{bNiceMatrix} B_0(y)\\\Vdots\\
%B_N(y)
%\end{bNiceMatrix},
%\end{multline*}
%and \eqref{eq:CD1} follows immediately.
%\item To get \eqref{eq:CD1_confluent} take the limit $x\to y$ in \eqref{eq:CD1}.
%
%\item
We use \eqref{eq:truncated_Hessenberg} and \eqref{eq:dual} to get 
\begin{align*}
 & -Q^{\langle N\rangle}(x)
 \begin{bNiceMatrix}
 0\\ \Vdots\\ 0 \\ B_{N+1} (y)
 \end{bNiceMatrix}+\begin{bNiceMatrix}
 0& \Cdots& 0& T_{N+p,N}Q_{N+p,N}(x)
 \end{bNiceMatrix} \begin{bNiceMatrix} B_0(y)\\\Vdots\\
 B_N(y)
 \end{bNiceMatrix}=(x-y) Q^{\langle N\rangle}(x)\begin{bNiceMatrix} B_0(y)\\\Vdots\\
 B_N(y)
 \end{bNiceMatrix},
\end{align*}
Now, recalling $ Q_{N,N}= h^{-1}_N B_N$, $ Q_{N+p,N}= (-1)^{p-1}h^{-1}_{N+1} B_{N+1}$ and $ h_{N+1}= (-1)^{p-1} T_{N+p,N}h_N $ we obtain \eqref{eq:CD2}. Finally, \eqref{eq:CD2_confluent} appears as a limit in \eqref{eq:CD2}.
%\end{enumerate}
\end{proof}

\begin{pro}[Spectral properties]Assume that $B_{N+1}$ has simple zeros at the set $\big\{\lambda^{[N]}_k\big\}_{k=1}^{N+1}$, so that
the vectors $u_k^{\langle N\rangle}\coloneq B^{\langle N\rangle}\big(\lambda^{[N]}_k\big)$ and $\tilde w^{\langle N\rangle}_k\coloneq Q^{\langle N\rangle} \big(\lambda^{[N]}_k\big)$ are right and left eigenvectors of $ T^{[N]}$, respectively, $k=1,\dots,N+1$. Then:
 \begin{enumerate}
 \item The normalized basis of left eigenvectors $\big\{w^{\langle N\rangle}_k\big\}_{k=1}^{N+1}$, which is biorthogonal to the basis of right eigenvectors $\big\{u^{\langle N\rangle}_k\big\}_{k=1}^{N+1}$, is given by
 \begin{align*}
 w^{\langle N\rangle}_{k}=\frac{ Q^{\langle N\rangle}\big(\lambda^{[N]}_k\big)}{\sum_{l=0}^{N}Q_{l,N}\big(\lambda^{[N]}_k\big)B_{l}\big(\lambda^{[N]}_k\big)}.
 \end{align*}
\item The following expression holds
\begin{align}
\label{eq:kcomponentelefteigenI}
 w^{\langle N\rangle}_{k,n}&=\frac{Q_{n-1,N}\big(\lambda^{[N]}_k\big)}{\sum_{l=0}^{N}Q_{l,N}\big(\lambda^{[N]}_k\big)B_{l}\big(\lambda^{[N]}_k\big)}=\frac{ Q_{n-1,N}\big(\lambda^{[N]}_k\big)
}{
Q_{N,N}\big(\lambda^{[N]}_k\big)B'_{N+1}\big(\lambda^{[N]}_k\big)}.
\end{align}
\item Alternatively, we can write
\begin{align}\label{eq:discrete_linear_form_typeI}
 w^{\langle N\rangle}_{k,n}= A_{n-1}^{(1)}\big(\lambda^{[N]}_k\big)\mu^{[N]}_{k,1} +\cdots +A_{n-1}^{(p)}\big(\lambda^{[N]}_k\big)\mu^{[N]}_{k,p},
\end{align}
with masses or Christoffel numbers defined as
\begin{align*}
 \mu_{k,1}^{[N]}&\coloneq 
 \frac{\begin{vNiceMatrix}
 A^{(2)}_{n+1} \big(\lambda^{[N]}_k\big)& \Cdots & A^{(p)}_{n+1}\big(\lambda^{[N]}_k\big) \\[2pt]
% A^{(1)}_{n-1} & \Cdots & A^{(p)}_{n-1} \\[2pt]
% A^{(1)}_{n} & \Cdots & A^{(p)}_{n} \\[2pt]
 \Vdots & \Ddots & \Vdots \\[2pt]
 A^{(2)}_{n+p-1} \big(\lambda^{[N]}_k\big)& \Cdots & A^{(p)}_{n+p-1}\big(\lambda^{[N]}_k\big) \\[2pt]
 \end{vNiceMatrix}}{\sum_{l=0}^{N}Q_{l,N}\big(\lambda^{[N]}_k\big)B_{l}\big(\lambda^{[N]}_k\big)},\\
 \mu^{[N]}_{k,2}&\coloneq - \frac{\begin{vNiceMatrix}
 A^{(1)}_{n+1} \big(\lambda^{[N]}_k\big)& A^{(3)}_{n+1} \big(\lambda^{[N]}_k\big)&\Cdots & A^{(p)}_{n+1}\big(\lambda^{[N]}_k\big) \\[2pt]
% A^{(1)}_{n-1} & \Cdots & A^{(p)}_{n-1} \\[2pt]
% A^{(1)}_{n} & \Cdots & A^{(p)}_{n} \\[2pt]
 \Vdots & \Vdots & \Ddots & \Vdots \\[2pt]
 A^{(1)}_{n+p-1} \big(\lambda^{[N]}_k\big)& A^{(3)}_{n+p-1} \big(\lambda^{[N]}_k\big)&\Cdots & A^{(p)}_{n+p-1}\big(\lambda^{[N]}_k\big) \\[2pt]
 \end{vNiceMatrix}}{\sum_{l=0}^{N}Q_{l,N}\big(\lambda^{[N]}_k\big)B_{l}\big(\lambda^{[N]}_k\big)},\\
 &\hspace*{8pt}\vdots \\
 \mu^{[N]}_{k,p}&\coloneq (-1)^{p-1} \frac{\begin{vNiceMatrix}
 A^{(1)}_{n+1} \big(\lambda^{[N]}_k\big)&\Cdots & A^{(p-1)}_{n+1}\big(\lambda^{[N]}_k\big) \\[2pt]
% A^{(1)}_{n-1} & \Cdots & A^{(p)}_{n-1} \\[2pt]
% A^{(1)}_{n} & \Cdots & A^{(p)}_{n} \\[2pt]
 \Vdots & \Ddots & \Vdots \\[2pt]
 A^{(1)}_{n+p-1} \big(\lambda^{[N]}_k\big)&\Cdots & A^{(p-1)}_{n+p-1}\big(\lambda^{[N]}_k\big) \\[2pt]
 \end{vNiceMatrix}}{\sum_{l=0}^{N}Q_{l,N}\big(\lambda^{[N]}_k\big)B_{l}\big(\lambda^{[N]}_k\big)}.
 \end{align*}
%\item The following alternative expression, in terms of truncated polynomials, for the components of the normalized left eigenvectors 
%\begin{align}\label{eq:discrete_linear_form_truncations}
% w^{\langle N\rangle}_{k,n}=\frac{B_{N+1}^{[n]}\big(\lambda^{[N]}_k\big)}{\sum_{l=0}^{N}B_{N+1}^{[l+1]}\big(\lambda^{[N]}_k\big)B_{l}\big(\lambda^{[N]}_k\big)}=\frac{B_{N+1}^{[n]}\big(\lambda^{[N]}_k\big)}{B'_{N+1}\big(\lambda^{[N]}_k\big)},
%\end{align}
%holds. Moreover, for the Christoffel coefficients we have \cite[Chapter 4, Equation 3.11]{nikishin_sorokin}, 
%\begin{align}\label{eq:masses}
% \mu_{k,1}^{[N]}&=\frac{B_{N+1}^{(1)}\big(\lambda^{[N]}_k\big)}{B'_{N+1}\big(\lambda^{[N]}_k\big)}, &
% \mu^{[N]}_{k,2}&=\frac{B_{N+1}^{(2)}\big(\lambda^{[N]}_k\big)}{B'_{N+1}\big(\lambda^{[N]}_k\big)}.
%\end{align}
\item It holds that
 \begin{align}
 \label{eq:kcomponentelefteigenII}
 \begin{bNiceMatrix}
 \mu^{[N]}_{k,1} \\[5pt]
 \mu^{[N]}_{k,2} \\
 \Vdots
 \\
 \mu^{[N]}_{k,p}
 \end{bNiceMatrix}
&= \begin{bNiceMatrix}%[columns-width=2pt]
	1& 0 & \Cdots& && 0 \\
	\nu^{(1)}_1 & 1 & \Ddots&& & \Vdots \\
	\Vdots & \Ddots & \Ddots& && \\
	&&&& &\\&&&&&0\\
	\nu^{(1)}_{p-1} &\Cdots& && \nu^{(p-1)}_{p-1}& 1 
\end{bNiceMatrix}^{-1} \begin{bNiceMatrix}
w^{\langle N\rangle}_{k,1} \\[5pt]
w^{\langle N\rangle}_{k,2} \\
\Vdots \\
w^{\langle N\rangle}_{k,p}
 \end{bNiceMatrix}.
\end{align} 
 \item The corresponding matrices $\mathscr U$ (with columns the right eigenvectors $u_k$ arranged in the standard order) and 
 $\mathscr W$ and (with rows the left eigenvectors $w_k$ arranged in the standard order) satisfy
 \begin{align}\label{eq:UW=I}
\mathscr U\mathscr W=\mathscr W\mathscr U=I_{N+1}.
 \end{align}
\item In terms of the diagonal matrix $D=\diag\big(\lambda^{[N]}_1,\dots,\lambda^{[N]}_{N+1}\big)$ we have
 \begin{align}\label{eq:UDnW=Jn}
 \mathscr UD^n\mathscr W&=\big(T^{[N]}\big)^n, & n&\in\N_0.
\end{align}
\end{enumerate}
\end{pro}
\begin{proof}
\begin{enumerate}
 \item As the zeros are simple we have that left and right eigenvectors are orthogonal, i.e., $\tilde w^{\langle N\rangle}_k u^{\langle N\rangle}_l=\delta_{k,l} \sum_{r=0}^{N}Q_{r,N}\big(\lambda^{[N]}_k\big)B_{r}\big(\lambda^{[N]}_k\big)$. Hence, we divide by $\sum_{r=0}^{N}Q_{r,N}\big(\lambda^{[N]}_k\big)B_{r}\big(\lambda^{[N]}_k\big)$ to get normalized left eigenvectors.
\item It follows from the previous result and Equation \eqref{eq:CD2_confluent}. 
\item 
In Equation \eqref{eq:QNn} expand the determinant in $Q_{n-1,N}$ along its first row.
\item Use \eqref{eq:discrete_linear_form_typeI} for the first $p$ entries
\begin{align*}
 \begin{bNiceMatrix}
w^{\langle N\rangle}_{k,1} \\
\Vdots \\
w^{\langle N\rangle}_{k,p}
 \end{bNiceMatrix} 
&= \begin{bNiceMatrix}
 A^{(1)}_{0} \big(\lambda^{[N]}_k\big)&\Cdots & A^{(p)}_{0}\big(\lambda^{[N]}_k\big) \\[2pt]
% A^{(1)}_{n-1} & \Cdots & A^{(p)}_{n-1} \\[2pt]
% A^{(1)}_{n} & \Cdots & A^{(p)}_{n} \\[2pt]
 \Vdots & \Ddots & \Vdots \\[2pt]
 A^{(1)}_{p} \big(\lambda^{[N]}_k\big)&\Cdots & A^{(p)}_{p}\big(\lambda^{[N]}_k\big) \\[2pt]
 \end{bNiceMatrix} \begin{bNiceMatrix}
 \mu^{[N]}_{k,1} \\
 \Vdots
 \\
 \mu^{[N]}_{k,p}
 \end{bNiceMatrix}
\end{align*} 
and the initial conditions~\eqref{eq:initcondtypeI}
 \begin{align*}
 \begin{bNiceMatrix}
 A^{(1)}_{0} \big(\lambda^{[N]}_k\big)&\Cdots & A^{(p)}_{0}\big(\lambda^{[N]}_k\big) \\[2pt]
 A^{(1)}_{1} \big(\lambda^{[N]}_k\big)&\Cdots & A^{(p)}_{1}\big(\lambda^{[N]}_k\big) \\[2pt]
% A^{(1)}_{n-1} & \Cdots & A^{(p)}_{n-1} \\[2pt]
% A^{(1)}_{n} & \Cdots & A^{(p)}_{n} \\[2pt]
 \Vdots & \Ddots & \Vdots \\[2pt]
 A^{(1)}_{p} \big(\lambda^{[N]}_k\big)&\Cdots & A^{(p)}_{p}\big(\lambda^{[N]}_k\big) \\[2pt]
 \end{bNiceMatrix} = \begin{bNiceMatrix}%[columns-width=2pt]
 1& 0 & \Cdots& && 0 \\
 \nu^{(1)}_1 & 1 & \Ddots&& & \Vdots \\
 \Vdots & \Ddots & \Ddots& && \\
 &&&& &\\&&&&&0\\
 \nu^{(1)}_{p-1} &\Cdots& && \nu^{(p-1)}_{p-1}& 1 
\end{bNiceMatrix}
 \end{align*}
to obtain the result.
 \item It follows from the biorthogonality of the left and right eigenvectors.
 \item Notice that $\mathscr UD^n =\big(T^{[N]}\big)^n\mathscr U $ and use $\mathscr U^{-1}=\mathscr W$ to get $\mathscr UD^n\mathscr W=\big(T^{[N]}\big)^n$ as desired.
 \end{enumerate}

% \begin{align*}
% w^{\langle N\rangle}_{k,1}&=\frac{A_{N+1}^{(2)}\big(\lambda^{[N]}_k\big)}{\sum\limits_{l=1}^{N+1}\Big(A_{N+1}^{(2)}\big(\lambda^{[N]}_k\big) A_{l-1}^{(1)}\big(\lambda^{[N]}_k\big) -A_{N+1}^{(1)}\big(\lambda^{[N]}_k\big) A_{l-1}^{(2)}\big( \lambda_k^{[N]}\big)\Big)B_{l-1}\big(\lambda^{[N]}_k\big)},\\
% w^{\langle N\rangle}_{k,2}&=\frac{A_{N+1}^{(2)}\big(\lambda^{[N]}_k\big)\nu-A_{N+1}^{(1)}
% \big(\lambda^{[N]}_k\big)}{\sum\limits_{l=1}^{N+1}\Big(A_{N+1}^{(2)}\big(\lambda^{[N]}_k\big) A_{l-1}^{(1)}\big(\lambda^{[N]}_k\big) -A_{N+1}^{(1)}\big(\lambda^{[N]}_k\big) A_{l-1}^{(2)}\big( \lambda_k^{[N]}\big)\Big)B_{l-1}\big(\lambda^{[N]}_k\big)},
% \end{align*}
% and, recalling Theorem \ref{theorem:CD}, the result follows.
\end{proof}

\begin{rem}
	Multiple quadrature formulas have been studied in \cite{Coussement-VanAssche,Ulises-Illan-Guillermo}. In particular, in \cite{Ulises-Illan-Guillermo} the convergence properties of simultaneous quadrature rules of a given function with respect to different weights is studied.
%\begin{enumerate}
%% \item If we compare the expressions for the masses in \eqref{eq:mass1} and \eqref{eq:mass2}, in this banded Hessenberg scenario, with the simpler Jacobi case, we immediately see that positivity is not ensured in this case. 
%%Hence, to analyze its positivity it seems more promising \eqref{eq:masses}, and as we will show that will be the case for oscillatory matrices.
%\item Multiple quadrature formulas have been studied in \cite{Ulises-Illan-Guillermo,Coussement-VanAssche}. In particular, in \cite{Ulises-Illan-Guillermo} the convergence properties of simultaneous quadrature rules of a given function with respect to different weights is studied.
%\end{enumerate}
\end{rem}

Let us consider $p$ %\textcolor{blue}{singular} \textcolor{red}{discrete}
 measures supported at the zeros of $B_{N+1}$,
\begin{align}\label{eq:discrete_mesures}
 \mu^{[N]}_{a}&\coloneq \sum_{j=1}^{N+1}\mu^{[N]}_{j,a}\delta\big(x-\lambda^{[N]}_j\big), &a&\in\{ 1, \dots, p\}.
\end{align}
Let us show that the recursion polynomials of type II, $\{B_n\}_{n=0}^N$, and of type I, $\big\{A^{(a)}_n\big\}_{n=0}^N$, $a = 1, \dots , p$, are finite sequences of multiple discrete orthogonal polynomials with respect to these measures.
\begin{teo}[Multiple discrete biorthogonalities]\label{theorem:birothoganality}
 Assume that the recursion polynomials $B_{N+1}$ have simple zeros $\{\lambda_k^{[N]}\}_{k=1}^{N+1}$.
The following biorthogonal relations hold
 \begin{align*}
 \left\langle A^{(1)}_{k}\mu^{[N]}_1+\cdots + A^{(p)}_{k}\mu^{[N]}_p,B_l\right\rangle&=\delta_{k,l}, &k,l&\in\{0,\dots,N\}.
 \end{align*}
\end{teo}
\begin{proof}
The matrices
 \begin{align}\label{eq:UW}
 \mathscr U&=\begin{bNiceMatrix}
 B_0\big(\lambda^{[N]}_1 \big)&\Cdots & B_0\big(\lambda^{[N]}_{N+1} \big)\\
 \Vdots & \Ddots & \Vdots\\
 B_N\big(\lambda^{[N]}_1 \big)& \Cdots & B_N\big(\lambda^{[N]}_{N+1}\big)
 \end{bNiceMatrix}, &
 \mathscr W&=\begin{bNiceMatrix}
 w^{[N]}_{1,1} & \Cdots & w^{[N]}_{1,N+1}\\
 \Vdots & \Ddots & \Vdots\\
 w^{[N]}_{N+1,1}& \Cdots & w^{[N]}_{N+1,N+1}
 \end{bNiceMatrix}.
\end{align}
satisfy $UW=I$, and using \eqref{eq:discrete_linear_form_typeI} biorthogonality follows immediately.
\end{proof}

The previous discrete biorthogonality implies a set of multiple orthogonal relations of types I and II.

We can check, using the initial conditions \eqref{eq:initcondtypeI} and the recursion relation \eqref{eq:recursion_dual_A} that 
\begin{align*}
%\deg A^{(1)}_{0}=0, \deg A^{(2)}_{0}=0 \dots \deg A^{(p)}_{0}=0\\
%\vdots \\
%\deg A^{(1)}_{p-1}=0, \deg A^{(2)}_{p-1}=0 \dots \deg A^{(p)}_{p-1}=0\\
\deg A^{(1)}_{p}=1, \deg A^{(2)}_{p}=0 , \dots , \deg A^{(p)}_{p}=0 , \\
\deg A^{(1)}_{p+1}=1, \deg A^{(2)}_{p+1}=1 , \dots , \deg A^{(p)}_{p+1}=0 .
\end{align*}
In general, it holds that $\deg A^{(r)}_{kp+j} = k$, for $r = 1, \dots, j+1$ and 
$\deg A^{(r)}_{kp+j}=k-1$ for $r = j+2, \dots , p$.

 \begin{coro}[Multiple discrete orthogonalities]
 Assume that the recursion polynomials $B_{N+1}$ have simple zeros $\big\{\lambda_k^{[N]}\big\}_{k=1}^{N+1}$.
 For $k$ such that $kp+j \leqslant N$, the following type II multiple orthogonal conditions are satisfied
 \begin{align*}
 \left \langle\mu^{[N]}_r, x^m B_{kp +j}\right\rangle&=0, & m&=0,\dots,k,& r&= 1, \dots, j,\\
 \left \langle\mu^{[N]}_r, x^m B_{kp+j}\right\rangle&=0, & m&=0,\dots,k-1,& r&= j+1 ,\dots, p,
 \end{align*}
with $j = 0 , \ldots , p-1$.
For the recursion polynomials of type I we have the following discrete type I multiple orthogonality 
\begin{align*}
\left\langle A^{(1)}_{kp+j}\mu^{[N]}_1+\cdots +A^{(p)}_{kp+j}\mu^{[N]}_p ,x^n\right\rangle&=0, & n&\in\{0,1,\dots,kp+j-1\}, & k\in\{1,\dots, N\}.
\end{align*}
 \end{coro}

Frequently, see for example \cite{VanAssche2}, the following notation is used for these recursion polynomials of types I and II, respectively
%\textcolor{red}{latex!!!}
\begin{align*}
%%A^{(r)}_{(\underset{\mathclap{j+1 \text{ times}}}{k+1,\dots, k+1},k, \dots,k)}
A^{(a)}_{
		\begin{pNiceMatrix}[small]
			k+1, & \Cdots, & k+1, & k, & \Cdots & k 
%			\CodeAfter
%			\UnderBrace[shorten]{1-1}{1-3}{\mathclap{\tiny \text{$(j+1)$-times}}}
	\end{pNiceMatrix}
}
&\coloneq A^{(a)}_{kp+j}, &
 a& \in \{1, \dots , p\}, & B_{
 	\begin{pNiceMatrix}[small]
 		k+1, & \Cdots, & k+1, & k, & \Cdots & k 
% 		\CodeAfter
% 		\UnderBrace[shorten]{1-1}{1-3}{\mathclap{ \text{\tiny $j$-times}}}
 	\end{pNiceMatrix}
 }\coloneq B_{kp+j}.
\end{align*}
In the first expression $k+1$ is repeated $j+1$ times, while in the second $j$ times.
%EL VECTOR TIENE $j+1$, $k+1$de manera que el grado total es $kp+j+1$.
%and for the type II recursion polynomials
%$$. %&& r \in\{1, \dots , p\}

%EL VECTOR TIENE $j$, $k+1$de manera que el grado total es $kp+j$.
\begin{pro}[Lebesgue--Stieltjes representation of the measures]
%CAMBIAR EL NOMBRE DE LA MATRIZ CONSTANTE $\nu$
In terms of the following piecewise continuous functions
\begin{align*}
 \psi^{[N]}_{a}&\coloneq \begin{cases}
 0, & x<\lambda^{[N]}_{N+1},\\[2pt]
 \mu^{[N]}_{1,a}+\cdots+\mu^{[N]}_{k,a}, & \lambda^{[N]}_{k+1}\leqslant x\leqslant \lambda^{[N]}_{k}, \quad k\in\{1,\dots,N\},\\[2pt]
 \mu^{[N]}_{1,a}+\cdots+\mu^{[N]}_{N+1,a} = (\nu^{-\top})_{1,a},
 & x > \lambda^{[N]}_{1},
 \end{cases}
\end{align*}
we can write 
$\mu_{a} ^{[N]}=\d\psi^{[N]}_{a}$ for $a\in\{1, \dots, p\}$.
\end{pro}
\begin{proof}
It is clear that $\d\psi^{[N]}_a=\mu_{a} ^{[N]}$.

Let us write \eqref{eq:kcomponentelefteigenII} in the alternative form
 \begin{align}\label{eq:mu_W_e_nu}
 \begin{bNiceMatrix}
 \mu^{[N]}_{1,1} &\Cdots&\mu^{[N]}_{1,p} \\
 \Vdots & & \Vdots \\
 \mu^{[N]}_{N+1,1} &\Cdots&\mu^{[N]}_{N+1,p}
 \end{bNiceMatrix}
 &= \begin{bNiceMatrix}
 w^{\langle N\rangle}_{1,1} &\Cdots &w^{\langle N\rangle}_{1,p} \\
 \Vdots & & \Vdots \\
 w^{\langle N\rangle}_{N+1,1} &\Cdots &w^{\langle N\rangle}_{N+1,p}
 \end{bNiceMatrix} \begin{bNiceMatrix}%[columns-width=2pt]
 1& 0 & \Cdots& && 0 \\
 \nu^{(1)}_1 & 1 & \Ddots&& & \Vdots \\
 \Vdots & \Ddots & \Ddots& && \\
 &&&& &\\&&&&&0\\
 \nu^{(1)}_{p-1} &\Cdots& && \nu^{(p-1)}_{p-1}& 1 
\end{bNiceMatrix}^{-\top} .
\end{align} 
Multiplying on the left this equation by $ \begin{bNiceMatrix}1 &\Cdots &1 \end{bNiceMatrix}$ and using the biorthogonality $UW = I$, we get
\begin{align*}
\begin{bNiceMatrix}1 &\Cdots &1 \end{bNiceMatrix} \begin{bNiceMatrix}
 \mu^{[N]}_{1,1}&\Cdots&\mu^{[N]}_{1,p} \\
 \Vdots & & \Vdots \\
 \mu^{[N]}_{N+1,1} &\Cdots&\mu^{[N]}_{N+1,p}
 \end{bNiceMatrix} = \begin{bNiceMatrix}1 & 0 & \Cdots &0 \end{bNiceMatrix}
\begin{bNiceMatrix}%[columns-width=2pt]
	1& 0 & \Cdots& && 0 \\
	\nu^{(1)}_1 & 1 & \Ddots&& & \Vdots \\
	\Vdots & \Ddots & \Ddots& && \\
	&&&& &\\&&&&&0\\
	\nu^{(1)}_{p-1} &\Cdots& && \nu^{(p-1)}_{p-1}& 1 
\end{bNiceMatrix}^{-\top} ,
\end{align*}
so that $ \mu^{[N]}_{1,a}+\cdots+\mu^{[N]}_{N+1,a} = (\nu^{-\top})_{1,a}$.
\end{proof}

\section{Spectral theory}

We now introduce the very important idea of positive bidiagonal factorization. This factorization is very natural for banded matrices as all the subdiagonals may be constructed in terms of simpler bidiagonal matrices. 
\begin{defi}[Positive bidiagonal factorization]
We say that a banded Hessenberg matrix $ T$ as in \eqref{eq:monic_Hessenberg} admits a positive bidiagonal factorization if
%the following factorization holds
 \begin{align}\label{eq:bidiagonal}
 T= L_{1} L_{2} \cdots L_{p} U,
 \end{align}
with bidiagonal matrices given respectively by 
\begin{align}\label{eq:bidiagonal_factors}
\left\{
\begin{aligned}
	L_k&\coloneq \begin{bNiceMatrix}[columns-width =auto]
 1 &0&\Cdots&&&\\
 ( L_{k} )_{1,0}& 1 &\Ddots&& &\\
 0& ( L_{k} )_{2,1}& 1& &&\\
 \Vdots&\Ddots& ( L_{k} )_{3,2}& 1 & &\\
 && &\Ddots & \Ddots&\\
 \phantom{h}&\phantom{h}&\phantom{h}&\phantom{h}&\phantom{h}&\phantom{h}&\phantom{h}
% \CodeAfter\line{3-1}{6-1}\line{3-1}{6-4} \line{4-3}{6-6}\line{4-4}{5-6} 
% \line{1-2}{1-6}
% \line{1-2}{4-6} 
% \SubMatrix[{1-1}{5-6}]
 \end{bNiceMatrix}, & (L_k)_{j+1,j}&=\alpha_{k+1+j(p+1)}, & j&\in\N_0,\\
U& \coloneq
 \left[\begin{NiceMatrix}[columns-width = .9cm]
U_{0,0}& 1 &0&\Cdots&\\
 0& U_{1,1}& 1 &\Ddots&\\
 \Vdots&\Ddots&U_{2,2}&\Ddots&\\
 & & \Ddots &\Ddots &
 \end{NiceMatrix}\right], & U_{j,j}&=\alpha{1+j(p+1)}, & j&\in\N_0,
\end{aligned}\right.
\end{align}
and such that the positivity constraints $\alpha_i >0,$ for $i\in\N$, are satisfied. 
\end{defi}
% The following relations for the matrix entries hold 
% \begin{align}\label{eq:alphas}
% \left\{\begin{aligned}
% c_n & = \alpha_{3n+1} + \alpha_{3n} + \alpha_{3n-1} , \\
% b_n & = \alpha_{3n} \alpha_{3n-2}+ \alpha_{3n-1} \alpha_{3n-2}+ \alpha_{3n-1} \alpha_{3n-3}, \\
% T_n & = \alpha_{3n-1} \alpha_{3n-3} \alpha_{3n-5}.
% \end{aligned}\right.
% \end{align}
\begin{rem}
	Notice that, $L_1,\dots, L_p,U\in \operatorname{InTN}$. %The positivity condition is the requirement that $ \alpha_i>0$, $i\in\N$.
\end{rem}
\begin{pro}
	The above positive bidiagonal factorization of $T$ induces the following positive bidiagonal factorization for the leading principal submatrix $ T^{[N]}$ given in Definition \ref{def: Jtruncated} 
 \begin{align}\label{eq:bidiagonal_HessenbergN}
 T^{[N]}= L^{[N]}_{1} L_{2}^{[N]} \cdots L_{p}^{[N]} U^{[N]},
 \end{align}
 with bidiagonal matrices given by
\begin{align}\label{eq:L1L2U_truncated}
 L^{[N]}_{k}&=
 \begin{bNiceMatrix}
 1 &0&\Cdots&&0\\
 \alpha_{k+1} &\Ddots &\Ddots&&\\
 0& \alpha_{k+1+(p+1)} & & &\Vdots\\
 \Vdots& \Ddots& \Ddots & \Ddots&0\\
 0&\Cdots&0& \alpha_{k+1+(p+1)N}&1
 \end{bNiceMatrix},
%& L_{2}^{[N]}&=
% \begin{bNiceMatrix}
% 1 &0&\Cdots&&0\\
% \alpha_{3} &\Ddots &\Ddots&&\\
% 0& \alpha_{4+p} & & &\Vdots\\
% \Vdots& \Ddots& \Ddots & \Ddots&0\\
% 0&\Cdots&0& \alpha_{(p+1)N+3}&1
% \end{bNiceMatrix},
& 
U^{[N]}& =
 \left[\begin{NiceMatrix}
 \alpha_1 & 1 &0&\Cdots&0\\
 0& \alpha_{1+(p+1)} & \Ddots &\Ddots&\Vdots\\
 \Vdots&\Ddots& \alpha_{1+2(p+1)}&&0\\
 & & \Ddots &\Ddots &1\\
 0 &\Cdots&&0& \alpha_{1+(p+1)N}
 \end{NiceMatrix}\right].
\end{align} 
\end{pro}

\begin{pro}
If $T$ admits a positive bidiagonal factorization then its truncations 	$T^{[N]} $ are oscillatory.
\end{pro}
\begin{proof}
	As all factors are InTN the product matrix is InTN. Moreover, as all parameters in the bidiagonal factors are positive then using Theorem~\ref{teo:Gantmacher--Krein Criterion} we get that the matrix is oscillatory.
\end{proof}

For the tetradiagonal case, $p=2$, it can be proven that if $T^{[N]} $ is an oscillatory matrix there exists such positive bidiagonal factorizations. If for each $N$ the matrix $T^{[N]} $ is oscillatory there exists a positive bidiagonal factorization for the semi-infinite matrix $T$ except for $ \alpha_2 $ that we can only guarantee only to be nonnegative \cite{bidiagonal_factorization_paper}.

\subsection{Interlacing properties and positivity of Christoffel coefficients}
%Let us assume that $T$ given in \eqref{eq:monic_Hessenberg} is a regular oscillatory matrix, so the matrices $T^{[N]}$ are oscillatory. 

%The matrices $T^{[N]}$ are oscillatory so we conclude that 
%all the matrices $ T^{[N,k]}$, $k\in\{0,1,\dots,N+1\}$ are also oscillatory. See Theorem \ref{teo:contigous}. Then, all the eigenvalues $ \lambda_1^{[N,k]}> \lambda_2^{[N,k]}>\cdots > \lambda_{N+1-k}^{[N,k]}>0$ of $ T^{[N,k]}$ are simple and positive. In fact, as we have seen, these eigenvalues are the zeros of the polynomial $B^{[k]}_{N+1}$. If we take out the first column and first row of the matrix $ T^{[N,k]}$ we get the matrix $ T^{[N,k+1]}$, and if we take out the last column and the last row of the matrix $ T^{[N,k]}$ we get the matrix $ T^{[N-1,k]}$, so again using the Theorem \ref{teo:contigous}
%we get the strict interlacing properties
%\begin{align*}
% \lambda_1^{[N,k]}> \lambda_1^{[N,k+1]}> \lambda_2^{[N,k]}> \lambda_2^{[N,k+1]}>\cdots> \lambda_{N-k}^{[N,k+1]}> \lambda_{N+1-k}^{[N,k]}>0,\\
% \lambda_1^{[N,k]}> \lambda_1^{[N-1,k]}> \lambda_2^{[N,k]}> \lambda_2^{[N-1,k]}>\cdots> \lambda_{N-k}^{[N-1,k]}> \lambda_{N+1-k}^{[N,k]}>0.
%\end{align*}
%That is, $B^{[k]}_{N+1}$ interlaces $B^{[k+1]}_{N+1}$ and $B^{[k]}_{N}$.
%

We now explore some consequences that a positive bidiagonal factorization has. For that aim we introduce the idea of Darboux transformation of a banded Hessenberg matrix. Darboux transformations for banded Hessenberg matrices (beyond the tridiagonal situation) were discussed in \cite{dolores}. In \cite{proximo} for the tetradiagonal case, and corresponding multiple orthogonal polynomials in the step-line with two weights, the PBF factorization is given in terms of the values of the orthogonal polynomials of type I and II at $0$ and, consequently, an spectral interpretation of the Darboux transformation is given.
\begin{defi}[Darboux transformations]
	Let us assume that $T$ admits a bidiagonal factorization (not necessarily positive). For each of its truncations $T^{[N]}$ we consider a chain of new auxiliary matrices, called Darboux transformations, given by the consecutive permutation of the triangular matrices in the factorization~\eqref{eq:bidiagonal_HessenbergN},
\begin{align}\label{eq:Darboux}
\left\{\begin{aligned}
 \hat T^{[N,1]}&=L_2 ^{[N]}\cdots L_p ^{[N]} U^{[N]} L_1^{[N]}, \\
 \hat T^{[N,2]}&=L_3 ^{[N]}\cdots L_p ^{[N]} U^{[N]} L_1^{[N]}L_2 ^{[N]}, \\ &\hspace*{5pt} \vdots\\
 \hat T^{[N,p-1]}&=L_p ^{[N]} U^{[N]} L_1^{[N]}L_2 ^{[N]} \cdots L_{p-1} ^{[N]},\\
 \hat T^{[N,p]}&= U^{[N]} L_1^{[N]}L_2 ^{[N]} \cdots L_{p} ^{[N]}.
\end{aligned}\right.
\end{align}
\end{defi}
%and denote by $\hat B_{N+1}$ its corresponding second kind polynomial. Consequently, 
%The auxiliary matrices are ``Darboux transformations'' of the submatrices of the banded Hessenberg matrix. %, see \S \ref{S:Darboux}. 
\begin{lemma}
These Darboux transformations are banded Hessenberg matrices with only its first $p$ subdiagonals different from zero.
\end{lemma}
%\begin{rem}\label{rem:hat T is not tilde T}
% These auxiliary matrices $\hat T^{[N]}$ are not the $m$-th leading principal submatrix of the matrix $\hat T\coloneq L_2 \dots L_p UL_1$. The difference is in the last diagonal entry.
% \end{rem}
%The entries of $\hat T$ are
%\begin{align}\label{eq:entries hat_T}
%&\left\{\begin{aligned}
%\hat c_n & = \alpha_{3n+2} + \alpha_{3n+1} + \alpha_{3n} , \\ 
%\hat b_n & = \alpha_{3n} \alpha_{3n-1}+ \alpha_{3n+1} \alpha_{3n-1}+ \alpha_{3n} \alpha_{3n-2} , \\
%\hat T_n & = \alpha_{3n} \alpha_{3n-2} \alpha_{3n-4} ,
%\end{aligned}\right.&
%\end{align}
%All the entries of the $(N+1)$-th leading principal submatrix of $\,\hat T$ coincide with those of $\hat T^{[N]}$ but for the last diagonal entry, as
%$(\hat T^{[N]})_{N+1,N+1}= \alpha_{3N}+ \alpha_{3N+1}$ while $\hat c_{N+1}= \alpha_{3N}+ \alpha_{3N+1}+ \alpha_{3N+2}$.
%\end{rem}

\begin{lemma}\label{lemma:poly}
Let us assume that the positive bidiagonal factorization \eqref{eq:bidiagonal_HessenbergN} holds.
Then, for $k\in\{1,\dots,p\}$, we find
 \begin{enumerate}
 \item The Darboux transformation $\hat{ T}^{[N,k]}$ is oscillatory.
\item The characteristic polynomial of the Darboux transformation 
$\hat{ T}^{[N,k]}$ is $B_{N+1}$. 
\item If $w$ is a left eigenvector of $ T^{[N]}$, then $\hat w=w L_{1}^{[N]}\cdots L_{k}^{[N]}$ is a left eigenvector of $\hat{ T}^{[N,k]}$.
\end{enumerate}
\end{lemma}
\begin{proof}
\begin{enumerate}
 \item Each bidiagonal factor belongs to InTN. Then, the Darboux transformation $\hat{ T}^{[N,k]}$ is a product of matrices in InTN and, consequently, belongs to InTN. Moreover, the entries in the first superdiagonal are all $1$, while the entries in the first subdiagonal are sum of products of $\alpha$'s. Thus, all entries in these two diagonals are positive. According to Gantmacher--Krein Criterion, Theorem~\ref{teo:Gantmacher--Krein Criterion}, is an oscillatory matrix.
 \item As
${\hat T}^{[N,k]}=( L_{1}^{[N]}\cdots L_{k}^{[N]} )^{-1} T^{[N]} L_{1}^{[N]}\cdots L_{k}^{[N]} $ its characteristic polynomial is $B_{N+1}$. 
%Consequently, from the theory of oscillatory matrices, $B_{N+1}$ also interlaces $\hat B_{N+1}^{[1]}$.

\item
We see that
\begin{align*}
\lambda \hat w=\lambda w L_{1}^{[N]}\cdots L_{k}^{[N]} =w L^{[N]}_{1} \cdots L_{k}^{[N]} L^{[N]}_{k+1} \cdots L_{p}^{[N]} U^{[N]}L_{1}^{[N]}\cdots L_{k}^{[N]} =\hat w \hat{ T}^{[N]}.
\end{align*}
\end{enumerate}
\end{proof}

In order to show the positivity of the Christoffel coefficients we require of some preliminary notation.

\begin{defi}\label{def:L_alpha}
Let us define the matrix
 \begin{align*}
 \mathscr L\coloneq\begin{bNiceMatrix}
 \mathscr L_1 & \mathscr L_2&\Cdots& \mathscr L_p
 \end{bNiceMatrix}\in\R^{p\times p}
 \end{align*}
 with columns
 \begin{align*}
 \mathscr L_1 &\coloneq \begin{bNiceMatrix}
 1\\0\\\Vdots\\0
 \end{bNiceMatrix},& \mathscr L_k&\coloneq\frac{1}{d_k}
 L_1^{[p-1]} \cdots L_{k-1}^{[p-1]} 
 \begin{bNiceMatrix}
 1\\0\\\Vdots\\0
 \end{bNiceMatrix}, & d_k\coloneq %& \alpha_{k+1} \alpha_{k+(p+1)} \alpha_{k-1+2(p+1)} \alpha_{k-2+3(p+1)}\cdots \alpha_{2+(k-1)(p+1)},
 & \alpha_{k} \alpha_{k+p} \alpha_{k+2p} \alpha_{k+3p}\cdots \alpha_{k+(k-2)p},
 & k&\in\{2,\dots, p\}.
 \end{align*}
% and
% \begin{align*}
% d_k\coloneq %& \alpha_{k+1} \alpha_{k+(p+1)} \alpha_{k-1+2(p+1)} \alpha_{k-2+3(p+1)}\cdots \alpha_{2+(k-1)(p+1)},
% & \alpha_{k} \alpha_{k+p} \alpha_{k+2p} \alpha_{k+1+3p}\cdots \alpha_{k+(k-2)p}.
% \end{align*}
\end{defi}

\begin{lemma}
$\mathscr L$ is a non-negative upper unitriangular matrix.
\end{lemma}

Then, we can deduce the following important result given conditions that ensure the Christoffel coefficients positivity.

\begin{teo}[Christoffel coefficients positivity]\label{theorem:Christoffel_positivy}
Let us assume that $T^{[N]}$ has a positive bidiagonal factorization \eqref{eq:bidiagonal_HessenbergN} and that the initial conditions in \eqref{eq:mic}, is such that
\begin{align*}
\nu^{-\top} = \mathscr L\mathscr C
\end{align*}
for some nonnegative upper unitriangular matrix $\mathscr C$.
Then, the Christoffel coefficients \eqref{eq:kcomponentelefteigenII} of the discrete measures given in \eqref{eq:discrete_mesures} for $T^{[N]}$ are positive, i.e.,
\begin{align*}
\mu^{[N]}_{n,a}&>0, & n&\in\{1,\dots,N+1\},& a&\in\{1,\dots,p\}.
\end{align*}
\end{teo}

\begin{proof}
 Let us begin proving that $\mu^{[N]}_{n,1} >0, \, n\in\{1,\dots,N+1\}$.
We consider the left eigenvector 
\begin{align*}
 Q^{\langle N\rangle} \big(\lambda^{[N]}_k\big) =
\begin{bNiceMatrix}
 Q_{0,N}\big(\lambda^{[N]}_k\big) & Q_{1,N} \big(\lambda^{[N]}_k\big)&\Cdots& Q_{N,N}\big(\lambda^{[N]}_k\big)
\end{bNiceMatrix}
\end{align*}
and recall that according to Theorem \ref{teo:eigenvectorII} the last entry of any eigenvector is nonzero, i.e. $ Q_{N,N}\big(\lambda^{[N]}_k\big)\neq 0$, so that 
we can normalize the last entry to $1$ to get the following rescaled left eigenvector of $T^{[N]}$
\begin{align*}
\omega^{\langle N\rangle}_k\coloneq \begin{bNiceMatrix}
 \frac{Q_{0,N}\big(\lambda^{[N]}_k\big)}{Q_{N,N}\big(\lambda^{[N]}_k\big)} &\frac{Q_{1,N} \big(\lambda^{[N]}_k\big)}{Q_{N,N}\big(\lambda^{[N]}_k\big)}& \Cdots & 1 
\end{bNiceMatrix}.
\end{align*}
We observe that this normalization is not the biorthogonal normalization leading to the left eigenvector $w^{\langle N\rangle}_k$.
Observe also that according to Theorem \ref{teo:eigenvectorII} the first entry is not zero $\frac{Q_{0,N}\big(\lambda^{[N]}_k\big)}{Q_{N,N}\big(\lambda^{[N]}_k\big)} \neq 0$.

Now, as the last entry is positive the change sign properties described in Theorem~\ref{teo:eigenvectorII}, leads to
\begin{align*}
 \frac{Q_{0,N}\big(\lambda^{[N]}_1\big)}{Q_{N,N}\big(\lambda^{[N]}_1\big)} &>0, & 
\frac{Q_{0,N}\big(\lambda^{[N]}_2\big)}{Q_{N,N}\big(\lambda^{[N]}_2\big)} &<0,& \frac{Q_{0,N}\big(\lambda^{[N]}_3\big)}{Q_{N,N}\big(\lambda^{[N]}_3\big)} &>0, & 
\end{align*}
and so on, alternating the sign. %Using the interlacing properties of the polynomial $B^{\prime}_{N+1}$ with respect to the polynomial $B_{N+1}$
Recall that as the polynomial $B_{N+1}$ is monic for the derivative $B'_{N+1}$ evaluated at the zeros $\lambda^{[N]}_k$ we have
\begin{align*}
 B_{N+1}'(\lambda^{[N]}_1)&>0, & B_{N+1}'(\lambda^{[N]}_2)&<0, & B_{N+1}'(\lambda^{[N]}_3)&>0,
\end{align*}
and so on, alternating the sign.
Then, it holds that
\begin{align*}
\frac{Q_{0,N}\big(\lambda^{[N]}_k\big)}{Q_{N,N}\big(\lambda^{[N]}_k\big)B^{\prime}_{N+1} \big(\lambda^{[N]}_k \big)} &>0, &k&\in\{1, \dots, N+1\}.
\end{align*}
From this using \eqref{eq:kcomponentelefteigenI} and \eqref{eq:kcomponentelefteigenII}, we get that this expression is equal to
$\mu^{[N]}_{k,1}$, $ k\in\{1, \dots, N+1\}$, and we get the statement for this case.

To discuss the Christoffel coefficient $\mu_{k,2}^{[N]} $, for $k\in\{1, \dots, N+1\}$ we use Lemma \ref{lemma:poly}. Recall that $\hat{T}^{[N,1]}$ is an oscillatory matrix with characteristic polynomial $B_{N+1}$. Then, a left eigenvector for the eigenvalue $\lambda^{[N]}_k $ can be chosen as 
\begin{align*}
\omega^{\langle N\rangle}_k L^{[N]}_1=\begin{bNiceMatrix}
 \frac{Q_{0,N}(\lambda^{[N]}_k)}{Q_{N,N}(\lambda^{[N]}_k\big)} &\frac{Q_{1,N} (\lambda^{[N]}_k)}{Q_{N,N}(\lambda^{[N]}_k)}& \Cdots & 1 
\end{bNiceMatrix}L_1^{[N]} 
= \begin{bNiceMatrix}
 \frac{(Q_{0,N}+ \alpha_2 Q_{1,N})(\lambda^{[N]}_k)}{Q_{N,N}(\lambda^{[N]}_k)} & \Cdots & 1 
\end{bNiceMatrix}.
\end{align*}
Again using the sign properties of the left eigenvector associated to an oscillatory matrix we get 
\begin{align*}
\frac{(Q_{0,N}+ \alpha_2 Q_{1,N})(\lambda^{[N]}_1)}{Q_{N,N}(\lambda^{[N]}_1)}&>0, &
\frac{(Q_{0,N}+ \alpha_2 Q_{1,N})(\lambda^{[N]}_2)}{Q_{N,N}(\lambda^{[N]}_2)}&<0, &\frac{(Q_{0,N}+ \alpha_2 Q_{1,N})(\lambda^{[N]}_3)}{Q_{N,N}(\lambda^{[N]}_3)}&>0, 
\end{align*}
and so on, alternating the sign. As before, using the interlacing properties of the polynomial $B^{\prime}_{N+1}$ with respect to the polynomial $B_{N+1}$, and after normalization it holds that
\begin{align*}
\tilde \mu_{k,2}^{[N]}\coloneq \frac{(\frac{1}{ \alpha_2}Q_{0,N}+Q_{1,N})(\lambda^{[N]}_k)}{Q_{N,N}(\lambda^{[N]}_k)B^{\prime}_{N+1} \big(\lambda^{[N]}_k \big)} &>0, & k&\in\{1, \dots, N+1\}.
\end{align*}

%This last quantity is not $\mu_{k,2}^{[N]} $, so we will denote it by $$. Using that
%\begin{align*}
%\mu_{k,2}^{[N]} = - \nu_{1,1} \frac{Q_{0,N}\big(\lambda^{[N]}_k\big)}{B^{\prime}_{N+1} \big(\lambda^{[N]}_k \big)Q_{N,N}\big(\lambda^{[N]}_k\big)} + \frac{Q_{1,N} \big(\lambda^{[N]}_k\big)} {B^{\prime}_{N+1} \big(\lambda^{[N]}_k \big)Q_{N,N}\big(\lambda^{[N]}_k\big)}
%\end{align*}
%it holds that choosing $- \nu_{1,1} = \frac{1}{ \alpha_2}$, we get that
%$$
%\mu_{k,2}^{[N]} = \frac{1}{ \alpha_2}\tilde \mu_{k,2}^{[N]} >0, \, k=1, \dots, N+1.
%$$

Let us discuss on $\mu_{k,3}^{[N]}$, for $k\in\{1, \dots, N+1\}$. Again we use Lemma \ref{lemma:poly}, with $\hat{T}^{[N,2]}$ an oscillatory matrix with characteristic polynomial $B_{N+1}$. Then, a corresponding left eigenvector for the eigenvalue $\lambda^{[N]}_k $ can be chosen as 
\begin{align*}
 \omega^{\langle N\rangle}_k L^{[N]}_1L^{[N]}_2=\begin{bNiceMatrix}
 \frac{Q_{0,N}\big(\lambda^{[N]}_k\big)}{Q_{N,N}\big(\lambda^{[N]}_k\big)} &\frac{Q_{1,N} \big(\lambda^{[N]}_k\big)}{Q_{N,N}\big(\lambda^{[N]}_k\big)}& \Cdots & 1 
 \end{bNiceMatrix}L_1^{[N]} L_2^{[N]} 
 = \begin{bNiceMatrix}
 \frac{(Q_{0,N}+( \alpha_2 + \alpha_3)Q_{1,N}+ \alpha_3 \alpha_{3+p}Q_{2,N})(\lambda^{[N]}_k)}{Q_{N,N}\big(\lambda^{[N]}_k\big)} & \Cdots & 1 
 \end{bNiceMatrix}.
\end{align*}
Again using the sign properties of the left eigenvector associated to an oscillatory matrix we get 
\begin{align*}
 \frac{(Q_{0,N}+( \alpha_2 + \alpha_3)Q_{1,N}+ \alpha_3 \alpha_{3+p}Q_{2,N})(\lambda^{[N]}_1)}{Q_{N,N}\big(\lambda^{[N]}_1\big)} &>0, &
\frac{(Q_{0,N}+( \alpha_2 + \alpha_3)Q_{1,N}+ \alpha_3 \alpha_{3+p}Q_{2,N})(\lambda^{[N]}_2)}{Q_{N,N}\big(\lambda^{[N]}_2\big)} &<0, %&\frac{(Q_{0,N}+( \alpha_2 + \alpha_3)Q_{1,N}+ \alpha_3 \alpha_{3+p}Q_{2,N})(\lambda^{[N]}_3)}{Q_{N,N}\big(\lambda^{[N]}_3\big)} &>0, 
\end{align*}
and so on, alternating the sign. As above we get
\begin{align*}
\tilde \mu^{[N]}_{k,3}\coloneq \frac{\big(\frac{1}{ \alpha_3 \alpha_{3+p}}Q_{0,N}+\frac{ \alpha_2 + \alpha_3}{ \alpha_3 \alpha_{3+p}}Q_{1,N}+Q_{2,N}\big)(\lambda^{[N]}_k)}{
 Q_{N,N}\big(\lambda^{[N]}_k\big)B^{\prime}_{N+1} \big(\lambda^{[N]}_k \big)}&>0, & k&\in\{1, \dots, N+1\}.
\end{align*}
%This last quantity is not $\mu_{k,2}^{[N]} $, so we will denote it by $\tilde \mu_{k,2}^{[N]}$. Using that
%\begin{align*}
% \mu_{k,2}^{[N]} = - \nu_{1,1} \frac{Q_{0,N}\big(\lambda^{[N]}_k\big)}{B^{\prime}_{N+1} \big(\lambda^{[N]}_k \big)Q_{N,N}\big(\lambda^{[N]}_k\big)} + \frac{Q_{1,N} \big(\lambda^{[N]}_k\big)} {B^{\prime}_{N+1} \big(\lambda^{[N]}_k \big)Q_{N,N}\big(\lambda^{[N]}_k\big)}
%\end{align*}
%Using now inductively this idea, we are generating the numbers
%$$
%\tilde \mu_{k,1}^{[N]}, \tilde\mu_{k,2}^{[N]}, \dots, \tilde \mu_{k,p}^{[N]}
%$$
%that are strictly positive, given by

We see that we have gotten quantities $\mu_{k,1}^{[N]},\tilde \mu_{k,2}^{[N]},\tilde \mu_{k,3}^{[N]}$ which are linear combinations of the entries of $w^{\langle N\rangle}_k$ and positive. Proceeding recursively we can get positive numbers $\tilde \mu_{k,a}^{[N]}>0$, for $a\in\{1, \dots, p\}$. We can see that the matrices providing the adequate linear combinations are the matrices $\mathscr L_j$ given in Definition \ref{def:L_alpha}.
These positive numbers can be proven to be
\begin{align*}
 \tilde \mu_{k,1}^{[N]} &\coloneq
 \begin{bNiceMatrix}
 w^{\langle N\rangle}_{k,1} &
 \Cdots &
 w^{\langle N\rangle}_{k,p} 
 \end{bNiceMatrix}
 \begin{bNiceMatrix}
 1\\0\\\Vdots\\0
 \end{bNiceMatrix},&
\tilde \mu_{k,j}^{[N]} &\coloneq 
\begin{bNiceMatrix}
 w^{\langle N\rangle}_{k,1} &
\Cdots &
w^{\langle N\rangle}_{k,p} 
\end{bNiceMatrix}
 \mathscr L_j,
 & j&\in\{2,\dots, p\}.
\end{align*}
Therefore,
$ \begin{bNiceMatrix}[small]
 \tilde \mu^{[N]}_{k,1} &
 \tilde \mu^{[N]}_{k,2} &
 \Cdots
 &
 \tilde \mu^{[N]}_{k,p}
 \end{bNiceMatrix}=\begin{bNiceMatrix}[small]
 w^{\langle N\rangle}_{k,1} &
 \Cdots &
 w^{\langle N\rangle}_{k,p} 
\end{bNiceMatrix}\mathscr L
$
is a entrywise positive row vector. 

We connect now this set of positive numbers $\big\{ \tilde \mu^{[N]}_{k,a}\big\}_{a=1}^p$ with our set of Christoffel coefficients $ \big\{ \mu^{[N]}_{k,a}\big\}_{a=1}^p$.
For that aim, let us write Equation \eqref{eq:kcomponentelefteigenII} in the alternative form
 \begin{align*}
 \begin{bNiceMatrix}
 \mu^{[N]}_{k,1} &
 \mu^{[N]}_{k,2} &
 \Cdots
 &
 \mu^{[N]}_{k,p}
 \end{bNiceMatrix}
 &= \begin{bNiceMatrix}
 w^{\langle N\rangle}_{k,1} &
 w^{\langle N\rangle}_{k,2} &
 \Cdots &
 w^{\langle N\rangle}_{k,p}
 \end{bNiceMatrix} \nu^{-\top} ,
\end{align*} 
and let us assume that $\mathscr C$ is a nonnegative upper unitriangular matrix. Then, we choose the initial conditions such that
$\nu^{-\top}=\mathscr L \mathscr C$.
Thus,
 \begin{align*}
 \begin{bNiceMatrix}
 \mu^{[N]}_{k,1} &
 \mu^{[N]}_{k,2} &
 \Cdots
 &
 \mu^{[N]}_{k,p}
 \end{bNiceMatrix}&= \begin{bNiceMatrix}
 w^{\langle N\rangle}_{k,1} &
 w^{\langle N\rangle}_{k,2} &
 \Cdots &
 w^{\langle N\rangle}_{k,p}
\end{bNiceMatrix} \mathscr L\mathscr C= \begin{bNiceMatrix}
 \tilde \mu^{[N]}_{k,1} &
 \tilde \mu^{[N]}_{k,2} &
 \Cdots
 &
 \tilde \mu^{[N]}_{k,p}
 \end{bNiceMatrix}\mathscr C
\end{align*} 
is an entrywise positive row vector.
%To finish the proof let us notice that we can take a diagonal matrix $D$ such that
%we get an upper triangular matrix, with entries equal $1$ in the main diagonal.
%\begin{align*}
% \left[ e_1^T, L_1^{[N]} e_1^T,L_1^{[N]}L_2^{[N]} e_1^T \dots, L_1^{[N]} \dots L_{p-1}^{[N]} e_1^T \right] D 
%\end{align*}
%So, now we choose the initial values, such that
%$$
%\nu^{-T} = \left[ e_1^T, L_1^{[N]} e_1^T,L_1^{[N]}L_2^{[N]} e_1^T \dots, L_1^{[N]} \dots L_{p-1}^{[N]} e_1^T \right] D
%$$
%and it is proven, noticing that the elements of the matrix $D$ are strictly positive.
 \end{proof}

\subsection{Truncated polynomials and recursion polynomials of the second kind}

\begin{defi}\label{def: k-truncated}
We consider
\begin{enumerate}
 \item
The submatrices 
 \begin{align}\label{eq:Hessenberg_k_truncation}
 T^{[N,k]}\coloneq
\begin{bNiceMatrix}[columns-width = auto]
 T_{k,k} & 1& 0 & \Cdots& &&&0\\
 T_{k+1,k} & T_{k+1,k+1} & 1&\Ddots& &&&\Vdots\\
 \Vdots & & \Ddots& \Ddots&&&&\\
 T_{p,k}& & & & &&&\\[7pt]
 0&\Ddots&&&&&\\
 \Vdots &\Ddots & &&&&&0\\
 \Vdots & & &&&&&1 \\
 0 & \Cdots&&0&T_{N,N-p}& \Cdots & &T_{N,N} 
 %\CodeAfter\line{4-1}{8-5} \line{5-1}{8-4} 
\end{bNiceMatrix}.
 \end{align}
\item The \emph{truncated} polynomials
 \begin{align}\label{eq:det_B-[k]}
 B^{[k]}_{N+1}(x)&\coloneq\det\big(x I_{N+1}- T^{[N,k]}\big).
% =\begin{vNiceMatrix}%[columns-width = 5pt]
% x- T_{k,k} & -1& 0 & \Cdots& &&&0\\
% -T_{k+1,k} & x-T_{k+1,k+1} & -1&\Ddots& &&&\Vdots\\
% \Vdots & & \Ddots& \Ddots&&&&\\
% -T_{p,k}& & & & &&&\\[3pt]
% 0&&&&&&\\
% \Vdots & & &&&&&0\\
% \Vdots & & &&&&&-1 \\
% 0 & \Cdots&&0&-T_{N,N-p}& \Cdots & &x-T_{N,N} 
%% \CodeAfter\line{4-1}{8-5} \line{5-1}{8-4} 
% \end{vNiceMatrix}.
\end{align}
\end{enumerate}
\end{defi}

Notice that $B_{N+1}=B^{[0]}_{N+1}$. For convenience, we also introduce
$B^{[N+1]}_{N+1}\coloneq 1$ and $B^{[N+k]}_{N+1}\coloneq 0$, $k\in\{2,\dots,p\}$,
with $\deg B^{[k]}_{N+1}=N+1-k$.
Expanding this determinant along its first column we find:

\begin{lemma}[Dual recursion for truncated polynomials]
For the truncated polynomials the following $(p+2)$ terms linear homogeneous recurrence holds
\begin{align}\label{eq:recurrence_bitruncation}
T_{p,k}B^{[k+p+1]}_{N+1}+\cdots+ T_{k+1,k}B^{[k+2]}_{N+1}+T_{k,k}B^{[k+1]}_{N+1}+ B^{[k]}_{N+1} &= x B^{[k+1]}_{N+1}, &k &\in\{0,1,\dots,N\}.
\end{align}
\end{lemma}

% Let us introduce the notation
%
%\begin{align*}
% \Omega^{\langle N\rangle}\coloneq \begin{bNiceMatrix}
% B^{[1]}_{N+1}& \Cdots & B^{[N+1]}_{N+1}
% \end{bNiceMatrix}.
%\end{align*}

\begin{pro}
The vectors
 \begin{align}\label{eq:left_eigenvectors}
 \omega^{\langle N\rangle}_{n}&\coloneq 
\begin{bNiceMatrix}
 B^{[1]}_{N+1}& \Cdots & B^{[N+1]}_{N+1}
\end{bNiceMatrix}\bigg|_{x=\lambda^{[N]}_{n}}, &n&\in\{1,\dots,N+1\},
 \end{align}
are the left eigenvectors of $T^{[N]}$ with last entry normalized to $1$. 
\end{pro}
\begin{proof}
Equation \eqref{eq:recurrence_bitruncation} implies
 \begin{align*}
 \begin{bNiceMatrix}
 B^{[1]}_{N+1}& \Cdots & B^{[N+1]}_{N+1}
 \end{bNiceMatrix}T^{[N]}+ \begin{bNiceMatrix}
 B^{[0]}_{N+1} & 0&\Cdots & 0
 \end{bNiceMatrix}=x \begin{bNiceMatrix}
 B^{[1]}_{N+1} & \Cdots & B^{[N+1]}_{N+1}
\end{bNiceMatrix},
 \end{align*}
and the result follows.
\end{proof}

As byproduct we get:

\begin{pro}[Christoffel--Darboux type formulas for truncated polynomials]\label{pro:CD2}
 \begin{enumerate}
 \item For the truncated polynomials the following Christoffel--Darboux type relation holds
 \begin{align}\label{eq:CD1}
 \sum_{n=0}^{N}B_{N+1}^{[n+1]}(x)B_{n}(y)=\frac{B_{N+1}(x)-B_{N+1}(y) }{x-y}.
 \end{align}
 \item The following confluent Christoffel--Darboux type formula is satisfied
 \begin{align}\label{eq:CD1_confluent}
 \sum_{n=0}^{N}B_{N+1}^{[n+1]}B_{n}=B'_{N+1}.
 \end{align}
 \end{enumerate}
\end{pro}
\begin{proof}
 \begin{enumerate}
 \item
It follows from \eqref{eq:truncated_Hessenberg} and \eqref{eq:left_eigenvectors} that
 \begin{align*}
 \begin{bNiceMatrix}
 B^{[1]}_{N+1}(x)& \Cdots & B^{[N+1]}_{N+1}(x)
 \end{bNiceMatrix}T^{[N]}+ \begin{bNiceMatrix}
 B^{[0]}_{N+1} (x)& 0&\Cdots & 0
 \end{bNiceMatrix}&=x \begin{bNiceMatrix}
 B^{[1]}_{N+1}(x) & \Cdots & B^{[N+1]}_{N+1}(x)
 \end{bNiceMatrix},\\
 T^{[N]}\begin{bNiceMatrix} B_0(y)\\\Vdots\\
 B_N(y)
 \end{bNiceMatrix}+\begin{bNiceMatrix}
 0\\\Vdots\\0\\B_{N+1} (y)
 \end{bNiceMatrix}&=y \begin{bNiceMatrix} B_0(y)\\\Vdots\\
 B_N(y)
 \end{bNiceMatrix},
 \end{align*}
 so that
 \begin{multline*}
 \begin{bNiceMatrix}
 B^{[0]}_{N+1} (x)& 0&\Cdots & 0
 \end{bNiceMatrix}\begin{bNiceMatrix} B_0(y)\\\Vdots\\
 B_N(y)
 \end{bNiceMatrix}- \begin{bNiceMatrix}
 B^{[1]}_{N+1}(x)& \Cdots & B^{[N+1]}_{N+1}(x)
 \end{bNiceMatrix}\begin{bNiceMatrix}
 0\\\Vdots\\0\\B_{N+1} (y)
 \end{bNiceMatrix}\\=(x-y ) \begin{bNiceMatrix}
 B^{[1]}_{N+1}(x)& \Cdots & B^{[N+1]}_{N+1}(x)
 \end{bNiceMatrix}\begin{bNiceMatrix} B_0(y)\\\Vdots\\
 B_N(y)
 \end{bNiceMatrix},
 \end{multline*}
 and \eqref{eq:CD1} follows immediately.
 \item To get \eqref{eq:CD1_confluent} take the limit $x\to y$ in \eqref{eq:CD1}.
 \end{enumerate}
\end{proof}

\begin{pro}
 For $k\in\{0,1,\dots,N\}$, the following holds
 \begin{align*}
 \frac{Q_{n,N}\big(\lambda^{[N]}_k\big)}{Q_{N,N}\big(\lambda^{[N]}_k\big)} &=B^{[n+1]}_{N+1}\big(\lambda^{[N]}_k\big),&
 w^{\langle N\rangle}_{k,n}&=\frac{ B_{N+1}^{[n]}\big(\lambda^{[N]}_k\big)
 }{B'_{N+1}\big(\lambda^{[N]}_k\big)}.
 \end{align*}
\end{pro}
\begin{proof}
	Use the normalized left eigenvectors $\omega^{\langle N\rangle}_k$ given in \eqref{eq:recurrence_bitruncation} and Equation \eqref{eq:kcomponentelefteigenI}.
\end{proof}

\begin{defi}[Second kind polynomials]
 The second kind polynomials $B^{(k)}_{N+1}$, $k\in\{1,\dots, p\}$, are the entries of the following row vector
 \begin{align*}
 \begin{bNiceMatrix}
 B^{(1)}_{N+1}& \Cdots & B^{(p)}_{N+1}
 \end{bNiceMatrix}= \begin{bNiceMatrix}
 B^{[1]}_{N+1}& \Cdots & B^{[p]}_{N+1}
\end{bNiceMatrix}
\nu^{-\top}.
 \end{align*}
\end{defi}
If $\{e_1,\dots,e_p\}$ is the canonical basis of $\R^p$ we have the modified basis
\begin{align*}
& \nu^{-\top}e_a, &a&\in\{1,\dots,p\}.
\end{align*}
Now, to get $ e_a^\nu$ we embed these vectors into $\R^N$, by adding zeros after the $p$-th entry.
For example,
$ e_1^\nu=e_1$, $e_2^{\nu}=e_2-\nu^{(1)}_{1}e_1$ and $ e_3^{\nu}= e_3-\nu^{(2)}_{2}e_2+\big(\nu^{(1)}_{1}\nu^{(2)}_{2}-\nu^{(1 )}_{2}\big)e_1$.
In the next result $\operatorname{adj}A$ denotes the adjugate matrix of the matrix $A$, see \cite{Horn-Johnson}.
\begin{pro}
 The second kind polynomials are given by
 \begin{align*}
 B^{(a)}_{N+1}(x)&\coloneq e_1^\top \operatorname{adj}(xI_{N+1}-T^{[N]})e^\nu_a, &a&\in\{1,\dots,p\}.
 \end{align*}
\end{pro}

The following technical result will be useful.
\begin{lemma}\label{eq:truco_o_trato}
	For matrices $\mathscr U$ and $\mathscr W$ as given in \eqref{eq:UW} we find
	\begin{align*}
		e_1^\top \mathscr U&=\begin{bNiceMatrix}
			1 &\Cdots &1
		\end{bNiceMatrix}, &
	\mathscr W e_a^\nu&= \begin{bNiceMatrix}
		\mu^{[N]}_{1,a}\\\Vdots\\ \mu^{[N]}_{N+1,a}
	\end{bNiceMatrix}, & a&\in\{1,\dots,p\}.
	\end{align*}
\end{lemma}
\begin{proof}
On the one hand, the entries in the first row of $\mathscr U$ are all $1$. On the other hand,
$	\mathscr W e_a^\nu$, is a adequate linear combination of the first $p$ columns of $\mathscr W$ with coefficients dictated by $\nu^{-\top}$. Now use \eqref{eq:mu_W_e_nu} to identify the corresponding column.
\end{proof}

The moments of these $p$ discrete measures are linked to the components of the powers of $T^{[N]}$:
\begin{pro}[Discrete moments]\label{pro:eq:moments_discrete}
	For the discrete moments we have
	\begin{align}\label{eq:moments_discrete}
		\left	\langle \mu^{[N]}_a, x^n\right\rangle&=\sum_{k=1}^{N+1}\mu_{k,a}^{[N]}\big( \lambda_k^{[N]}\big)^n
		=e_1^\top \big(T^{[N]}\big)^ne^\nu_a, & a&\in\{1,\dots, p\}.
	\end{align}
	\begin{proof}
		We have that $	e_1^\top \big(T^{[N]}\big)^ne^\nu_a=	e_1^\top \mathscr U D^n \mathscr We^\nu_a$, and \eqref{eq:truco_o_trato} leads to
		\begin{align*}
			e_1^\top \big(T^{[N]}\big)^ne^\nu_a=\begin{bNiceMatrix}
				1&\Cdots&1
			\end{bNiceMatrix}D^n
			\begin{bNiceMatrix}
				\mu^{[N]}_{1,a}\\\Vdots\\ \mu^{[N]}_{N+1,a}
			\end{bNiceMatrix}
		\end{align*}
		and the result follows.
	\end{proof}

We now discuss on the resolvent matrix $ R^{[N]}_{z}$ of the leading principal submatrix $ T^{[N]}$; i.e.,
\begin{align*}
 R_{z}^{[N]}\coloneq \big(z I_{N+1}- T^{[N]}\big)^{-1}=\frac{\operatorname{adj}\big(z I_{N+1}- T^{[N]}\big)}{\det(z I_{N+1}-T^{[N]})}.
\end{align*}
Notice that, from the spectral decomposition of the matrix $ T^{[N]}$, we obtain 
 \begin{align}\label{eq:spectral_resolvent}
 R_{z}^{[N]}
 & =\mathscr U(zI_{N+1}-D)^{-1}\mathscr W. \end{align}
\begin{defi}[Weyl's functions]
 We consider the set of Weyl functions $\big\{S_a^{[N]}\big\}_{a=1}^p$ defined by
\begin{align*}
 S_{a}^{[N]}(z)&\coloneq e_1^\top\big(zI_{N+1}- T^{[N]}\big)^{-1} e^\nu_a.
\end{align*}
\end{defi}

\begin{pro}\label{pro:Weyl_functions}
The Weyl functions can be expressed as follows
 \begin{align*}
 S_{a}^{[N]}(z)&= \frac{B^{(a)}_{N+1}(z)}{B_{N+1}(z)}=\sum_{n=1}^{N+1}\frac{\mu^{[N]}_{n,a}}{z-\lambda^{[N]}_n}=\left\langle
 \mu_a^{[N]},\frac{1}{z-x}
 \right\rangle, & a&\in\{1,\dots, p\}.
 \end{align*}
\end{pro}
\begin{proof}
The first equalities follow from adjugate expressions in
Proposition \ref{pro:determintal_second_kind}.
The second expressions can be deduced from \eqref{eq:spectral_resolvent}. Indeed, recalling \eqref{eq:truco_o_trato} we get that
 the Weyl functions are
\begin{align*}
 S_a^{[N]}(z)&= \begin{bNiceMatrix}
 1&\Cdots&1
 \end{bNiceMatrix}(zI_{N+1}-D)^{-1}
 \begin{bNiceMatrix}
 \mu^{[N]}_{1,a}\\\Vdots\\ \mu^{[N]}_{N+1,a}
 \end{bNiceMatrix}=\sum_{n=1}^{N+1}\frac{\mu^{[N]}_{n,a}}{z-\lambda^{[N]}_{n}}.
\end{align*}
\end{proof}
\begin{rem}
 The Christoffel coefficients are residues at the simple poles of these Weyl functions.
\end{rem}

	\end{pro}

\subsection{Favard theorem}

At this point we are ready to give one of the main results of the paper, that establish the existence of multiple orthogonal polynomials and corresponding positive
Lebesgue--Stieltjes
measures for a given bounded banded Hessenberg matrix that admit a positive bidiagonal factorization. The result is based in the positivity of the Christoffel coefficients established in Theorem \ref{theorem:Christoffel_positivy}.

Being $T^{[N]}$ oscillatory, so that the zeros of $B_{N+1}(x)$ strictly interlaces the ones of $B_N(x)$, we get that the positive sequences of eigenvalues $\{\lambda_1^{[N]}\}_{N=1}^\infty$ ($\{\lambda_{N+1}^{[N]})\}_{N=1}^\infty$ are strictly increasing (decreasing) sequences.
Moreover, for bounded operators $\|T\|_\infty <\infty$ we have $\|T^{[N]}\|_\infty <\|T\|_\infty <\infty$. Therefore, there exists the limits
\begin{align*}
	\xi&\coloneq \lim_{N\to\infty }\lambda_{N+1}^{[N]}, & \eta&\coloneq \lim_{N\to\infty }\lambda_1^{[N]}, 
\end{align*} 
with $\xi\geqslant 0$ and $\eta\leqslant \|T\|_\infty$. Following \cite{Chihara} we call $\Delta\coloneq [\xi,\eta]\subseteq [0,\|T\|_\infty]$ the true interval of orthogonality, that is the smallest interval containing all zeros of the type II polynomials $B_n$, i.e. eigenvalues of the truncations of $T$. By construction the measures in clear that $\operatorname{supp}\d\psi_a\subseteq \Delta$. However, the reverse inclusion is not as clear as in the simple orthogonality. For example it holds for AT systems.
%However, 
%for Angelescu and AT systems the zeros of the orthogonal polynomials belong 

\begin{teo}[Favard spectral representation]\label{theorem:spectral_representation_bis}
 Let us assume that
 \begin{enumerate}
 	\item The banded Hessenberg matrix $T$ is bounded and admit the positive bidiagonal factorization \eqref{eq:bidiagonal}.
 \item The sequences $\big\{A^{(1)}_n,\dots,A^{(p)}_n\big\}_{n=0}^\infty$ of recursion polynomials of type I, see Definition \ref{def:typeI}, are determined by the initial condition matrix $\nu$ such that $\nu^{-\top}=\mathscr L\mathscr C$, where $\mathscr L$ is given in Definition \ref{def:L_alpha} and $\mathscr C\in \R^{p\times p}$ is a nonnegative upper unitriangular matrix.
 \end{enumerate}
%Then, 
% The corresponding sequence $\{B_n\}_{n=0}^\infty$ of recursion polynomials of type II, see Definition\ref{def:typeII}, 
 Then, there exists $p$ non decreasing functions $\{\psi_k\}_{k=1}^p$, and corresponding positive Lebesgue--Stieltjes measures $\d\psi_k$ with compact support $\Delta$ such that the following biorthogonality between the sequence $\{B_n\}_{n=0}^\infty$ of recursion polynomials of type II and the mentioned type I recursion polynomials, see Definition\ref{def:typeII}, holds
 \begin{align*}
 \int_{\Delta} \big( A^{(1)}_{k}(x)\d \psi_1(x)+\cdots+A^{(p)}_{k}(x)\d \psi_p(x)\big)B_l(x)&= \delta_{k,l}, &k,l&\in\N_0.
 \end{align*}
\end{teo}
\begin{proof}
 The sequences $\big\{\psi_{k}^{[N]}\big\}_{N=0}^\infty$, $k\in\{1,\dots,p\}$ are uniformly bounded and nondecreasing.
 Consequently, following Helly's results, see \cite[\S II]{Chihara} there exist subsequences that converge when $N\to\infty$ to nondecreasing functions $\psi_1,\dots,\psi_p$. Being $T$ bounded its eigenvalues lay in a bounded set $\Delta$, and we deduce that these measures have compact support.
\end{proof}

Consequently,

\begin{coro}\label{cor:mor}
The multiple orthogonal relations of type II 
\begin{align*}
 \int_\Delta x^m B_{kp +j}\d \psi_r(x)&=0, & m&=0,\dots,k,& r=1, \dots, j ,\\
 \int_\Delta x^m B_{kp +j}\d \psi_r(x)&=0, & m&=0,\dots,k-1,& r=j+1, \dots, p ,\end{align*}
 with $j = 0 , \ldots , p-1$,
and of type I
\begin{align*}
 \int_{\Delta} \big( A^{(1)}_{kp+j}(x)\d \psi_1(x)+\cdots + A^{(p)}_{kp+j}(x)\d \psi_p(x)\big)x^n&=0, & n&\in\{0,1,\dots,kp+j-1\},
\end{align*}
are satisfied.
\end{coro}

\begin{pro}[Spectral representation of moments and Stieltjes--Markov functions]\label{pro:spectral_moments_Weyl}
 In terms of the spectral functions $\psi_1,\dots,\psi_p$ we find
\begin{align*}
 e_1^\top T^n e^\nu_a&=\int_\Delta x^n\d\psi_a(x), &\hat \psi_a\coloneq e_1^\top (z I- T)^{-1}e^\nu_a&=\int_\Delta\frac{\d\psi_a(x)}{z-x}, & a&\in\{1,\dots,p\}.
\end{align*}
\end{pro}
\begin{proof}
Propositions \ref{pro:eq:moments_discrete} and \ref{pro:Weyl_functions} and	Helly's second theorem leads to the spectral representation for the moments and Stieltjes--Markov functions of $T$.
\end{proof}

\begin{rem}
	For the Weyl functions in Proposition \ref{pro:Weyl_functions} we have in $\bar\C\setminus \Delta$ uniform convergence to the Stieltjes--Markov functions 
	\begin{align*}
		S^{[N]}_a&\xrightrightarrows [N\to\infty]{}\hat \psi_a, & a&\in\{1,\dots,p\}.
	\end{align*}
\end{rem}

Multiple quadrature formulas have been studied in \cite{Borges, Coussement-VanAssche,Ulises-Illan-Guillermo}. In particular, in \cite{Ulises-Illan-Guillermo} the convergence properties of simultaneous quadrature rules of a given function with respect to different weights is studied.

From our developments we can derive such multiple Gauss quadrature formulas easily. 
\begin{defi}
	The degrees of precision or orders $d_a(N)$, $a\in\{1,\dots,p\}$, are the largest natural numbers such that
	\begin{align*}
		e_1^\top T^n e^\nu_a &=e_1^\top \big(T^{[N]}\big)^n e^\nu_a, & 0&\leqslant n\leqslant d_a(N), & a&\in\{1,\dots,p\}.
	\end{align*}
\end{defi}
We us the ceiling $\lceil x\rceil$ function maps $x$ to the least integer greater than or equal to $x$.
\begin{pro}
In terms of the number of nodes $\mathcal N\coloneq N+1$, 	for $a\in\{1,\dots,p\}$, the degrees of precision~are
	\begin{align*}
	d_a&=\mathcal N-1+\upnu_a, &\upnu_a& \coloneq \left\lceil\frac{\mathcal N+1-a}{p}\right\rceil.
\end{align*}
\end{pro}
\begin{proof}
	For $p=1$ our matrix $T$ is tridiagonal, i.e. the nonzero entries of $T$ are of the form $T_{n,n+1}, T_{n,n}$ and $T_{n+1,n}$ and we are handling a Jacobi matrix, and we are asking for which entries of $T$ are involved in $(T^n)_{0,0}$. 
The degrees of precision we are seeking are determined by the property that in $(T^n)_{0,0}$ only factors that are entries of 
	$T^{[N]}$ contribute. 	For $N=1$, we have the following nonzero possibilities $T_{0,0} $, ($n=1$), $T_{0,1}T_{1,0}$ ($n=2$) and $T_{0,1}T_{1,1}T_{1,0}$ ($n=3$), and we stop here as for $n=4$ you have a nonzero term of the type $T_{0,1}T_{1,2}T_{2,1}T_{1,0}$ that involves entries from the $N=2$ truncation, and $d(1)=3$.
	For $N=2$, you have the previous term coming for $n=4$ but also for $n=5$ no entries from the truncation $N=3$ appear in $(T^5)_{0,0}$. However,
	for $n=6$ you have nonzero terms of the form $T_{0,1}T_{1,2}T_{2,3}T_{3,2}T_{2,1}$, involving elements from $T^{[3]}$,
	thus $d(2)=5$.
	For $N=3$, we have that up to $n=7$ no nonzero terms from $T^{[4]}$ appear, but do appear for $n=8$, consequently $d(3)=7$. By induction, we can show that
	$d(N)=2N+1$. 
	
	For $p=2$ the matrix $T$ is tetradiagonal with nonzero entries $T_{n,n+1}, T_{n,n}$, $T_{n+1,n}$ and $T_{n+2,n}$. 
	First we discuss $(T^n)_{0,0}$. 	For $N=1$ we find that $n=2$. When $N=2$ we have the powers up to $n=4$ are fine, however for $n=5$ the nonzero term
	$T_{0,1}T_{1,2}T_{2,3}T_{3,1}T_{1,0}$ tell us that $d_1(2)=4$. For $N=3$ the power $n=5$ is okay, but for $n=6$ we have nonzero terms such as
	$T_{0,1}T_{1,2}T_{2,3}T_{3,4}T_{4,2}T_{2,0}$, involving $T^{[4]}$, so that $d_1(3)=5$. For $N=4$, we can consider powers up to $n=7$, as for $n=8$ we have nonzero terms like 
	$T_{0,1}T_{1,2}T_{2,3}T_{3,4}T_{4,5}T_{5,3}T_{3,1}T_{1,0}$ and $d_1(4)=7$. The truncation $N=5$ admits up to $n=8$ as for $n=9$ we have nonzero terms such as
	$T_{0,1}T_{1,2}T_{2,3}T_{3,4}T_{4,5}T_{5,6}T_{6,4}T_{4,2}T_{2,0}$., then $d_1(5)=8$. For $N=6$ we can go up to $n=10$ as indicates the nonzero term
$T_{0,1}T_{1,2}T_{2,3}T_{3,4}T_{4,5}T_{5,6}T_{6,7}T_{7,5}T_{5,3}T_{3,1}T_{1,0}$. All these numbers can be gathered together by the formula
\begin{align*}
	d_1(N)=\left\lceil\frac{3N+1}{2}\right\rceil.
\end{align*}
Similar arguments for $(T^n)_{0,1}$ leads to the numbers $\left\lceil\frac{3N}{2}\right\rceil$. But $e_1^\top T^ne^\nu_1=(T^n)_{0,0}$ and
$e_1^\top T^ne^\nu_2=(T^n)_{0,2}-\nu^{(1)}_1(T^n)_{0,1}$. Consequently, the degrees of precision is the smaller of $\left\lceil\frac{3N+1}{2}\right\rceil$ and $\left\lceil\frac{3N}{2}\right\rceil$, so that
\begin{align*}
	d_2 (N)=\left\lceil\frac{3N}{2}\right\rceil.
\end{align*}

Induction on $p$ leads to 
\begin{align*}
	d_1(N)&=\left\lceil\frac{(p+1)N+1}{p}\right\rceil, &
	d_2(N)&=\left\lceil\frac{(p+1)N}{p}\right\rceil, &&\dots,
	& d_p(N)&=\left\lceil\frac{(p+1)N+2-p}{p}\right\rceil.
\end{align*}
Using that the number of nodes is $\mathcal N=N+1$, we find for $a\in\{1,\dots,p\}$
	\begin{align*}
	d_a=\left\lceil\frac{(p+1)(\mathcal N-1)+2-a}{p}\right\rceil=\mathcal N-1+\left\lceil\frac{\mathcal N+1-a}{p}\right\rceil, 
\end{align*}
and the result follows immediately.

\end{proof}

\begin{teo}[Multiple Gaussian quadrature formulas]
The following Gauss quadrature formulas hold
\begin{align}\label{eq:multiple_Gauss_quadrature}
	\int_\Delta x^n\d\psi_a(x)&=\sum_{k=1}^{N+1}\mu_{k,a}^{[N]}\big( \lambda_k^{[N]}\big)^n, & 0&\leqslant n\leqslant d_a(N),
	& a&\in\{1,\dots,p\}.
\end{align}
Here the degrees of precision $d_a$ are optimal (for any power largest than $n$ a positive remainder appears, an exactness is lost). 
\end{teo}
\begin{proof}
On the one hand, from Proposition \ref{pro:spectral_moments_Weyl} we have that $ e_1^\top T^n e^\nu_a=\int_\Delta x^n\d\psi_a(x)$. On the other hand, from Proposition \ref{pro:eq:moments_discrete}, we know that
$e_1^\top \big(T^{[N]}\big)^ne^\nu_a=\sum_{k=1}^{N+1}\mu_{k,a}^{[N]}\big( \lambda_k^{[N]}\big)^n$. Hence, as we have
\begin{align*}
e_1^\top T^n e^\nu_a &=e_1^\top \big(T^{[N]}\big)^n e^\nu_a, & 0&\leqslant n\leqslant d_a(N), & a&\in\{1,\dots,p\},
\end{align*}
we get \eqref{eq:multiple_Gauss_quadrature}. Notice that for $n>d_a(N)$ a positive remainder will appear and exactness will be lost. 
Observe $T^n$ is oscillatory and that $e_a^\nu=\mathscr L\mathscr C e_a$ is a positive vector, so all the objects involved imply positive contributions.

%\begin{align*}
%\int_\Delta t^n\d\psi_a(t)&=\sum_{k=1}^{N+1}\mu_{k,a}^{[N]}\big( \lambda_k^{[N]}\big)^n, & a&\in\{1,\dots,p\},
%\end{align*}
%with degree of precision $d_a(N)$.
\end{proof}

\begin{rem}
	As $d_a\geqslant \mathcal N-1$ the quadrature is interpolating \cite{Coussement-VanAssche}.
\end{rem}

\begin{rem}
	For $p=1$, the simple or non-multiple situation, we get the well-known Gauss quadrature that holds for $0\leqslant n\leqslant 2\mathcal N-1$, see \cite{Chihara, Ismail}.
	\end{rem}
\begin{rem}
A remarkable achievement here is to determine the degree of precision and not just lower bounds for them. 
The determination of the numbers $\upnu_a$, $a\in\{1,\dots,p\}$ is missing in the works \cite{Coussement-VanAssche,Ulises-Illan-Guillermo} .
\end{rem}

%We observe that from $T^nB=x^n B$ we get 
%\begin{align*}
%	x^n = (T^n)_{0,0}+(T^n)_{0,1} B_1(x)+\cdots+ (T^n)_{0,n}B_n(x),
%\end{align*}
%so that
%\begin{align*}
%	\int_{\Delta} x^n \d\psi_a(x)&=(T^n)_{0,0}	\int_{\Delta} \d\psi_a(x)+(T^n)_{0,1} \int_{\Delta} B_1(x)\d\psi_a(x)+\cdots+ (T^n)_{0,n}\int_{\Delta} B_n(x)\d\psi_a(x), 
%\end{align*}
%for $a\in\{1,\dots,p\}$. Recalling the multiple orthogonality relations of type II in Corollary \ref{cor:mor} we get
%\begin{align*}
%	\int_{\Delta} x^n \d\psi_1(x)&=(T^n)_{0,0}	\int_{\Delta} \d\psi_1(x), \\
%		\int_{\Delta} x^n \d\psi_1(x)&=(T^n)_{0,0}	\int_{\Delta} \d\psi_2(x)+(T^n)_{0,1} \int_{\Delta} B_1(x)\d\psi_2(x), \\
%		&\hspace*{5pt}\vdots\\
%			\int_{\Delta} x^n \d\psi_p(x)&=(T^n)_{0,0}	\int_{\Delta} \d\psi_p(x)+(T^n)_{0,1} \int_{\Delta} B_1(x)\d\psi_p(x)+\cdots+ (T^n)_{0,p-1}\int_{\Delta} B_{p-1}(x)\d\psi_p(x),
%\end{align*}

\section{Applications to Markov chains}

%We will discuss on the applications of the spectral representation we have given to bounded banded lower Hessenberg matrices that admit a positive bidiagonal factorization to 

We apply previous results to banded stochastic matrices 
 of the type II
\begin{align}\label{eq:PII}
\hspace{-.285cm} P =
\begin{bNiceMatrix}[columns-width =auto]
 P_{0,0} & P_{0,1}& 0 & \Cdots& &&&\phantom{X}\\
 P_{1,0} & P_{1,1} & P_{1,2}&\Ddots& &&&\phantom{X}\\
 \Vdots & & \Ddots& \Ddots&&&&\phantom{X}\\
 & & & &&&&\phantom{X}\ \\
 P_{p,0}& \Cdots& && P_{p,p}& P_{p,p+1} &&\phantom{X}\\%[10pt]
 0&P_{p+1,1}&\Cdots&&&P_{p+1,p+1}&P_{p+1,p+2}&\phantom{X}\\
 \Vdots & \Ddots& \Ddots&&&\phantom{X}&\Ddots&\Ddots&\phantom{X}\\
 \phantom{X} &\phantom{X} &\phantom{X}&\phantom{X}&\phantom{X}& \phantom{X}& \phantom{X}&\phantom{X}
% \CodeAfter\line{6-2}{8-4} 
% \line{6-1}{8-3}
% \line{1-3}{6-8}
% \line{2-3}{5-6}\line{2-2}{5-5}
% \line{6-6}{8-8}
% \line{6-7}{7-8}
 %\SubMatrix[{1-1}{7-8}]
\end{bNiceMatrix}
\end{align}
and of the type I
\begin{align}\label{eq:PI}
 P&=
\begin{bNiceMatrix}[columns-width =auto]
 P_{0,0} & P_{0,1}& \Cdots& & P_{0,p} &0&\Cdots&\phantom{X}\\
 P_{1,0} & P_{1,1} & && \Vdots&P_{1,p+1}&\Ddots&\phantom{X}\\
 0& P_{2,1}&\Ddots & &&\Vdots&\Ddots&\phantom{X}\\
 \Vdots& \Ddots& \Ddots& &&&&\phantom{X} \\
 & & &&P _{p,p}& &&\phantom{X}\\%[10pt]
 &&&&P_{p+1,p}&P_{p+1,p+1}&&\phantom{X}\\
 & & &&&P_{p+1,p+2}&\Ddots&\phantom{X}&\phantom{X}&\phantom{X}\\
&\phantom{X} &\phantom{X}&\phantom{X}&\phantom{X}& \phantom{X}& \Ddots&\phantom{X}
 \\
 \phantom{X} &\phantom{X} &\phantom{X}&\phantom{X}&\phantom{X}& \phantom{X}& \phantom{X}&\phantom{X}
% \CodeAfter
% \line{2-2}{5-5} \line{3-2}{6-5} 
% \line{3-1}{9-6}
%% \line{1-3}{6-9}
% \line{1-6}{3-8}
% \line{2-6}{4-8}
% \line{6-6}{8-8}
% \line{7-6}{9-8}
% \line{1-6}{1-8}
% \SubMatrix[{1-1}{8-8}]
\end{bNiceMatrix}.%\hspace*{-3.95cm}.
\end{align}
Here all the entries in each row are nonnegative and satisfy that its sum is $1$. 

In order to apply our previous analysis we need to understand the previous banded Hessenberg stochastic matrices in terms of a monic Hessenberg matrix $T$ or its transpose $T^\top$.

These stochastic matrices of Hessenberg type are connected to the monic banded Hessenberg matrix $T$ given in \eqref{eq:monic_Hessenberg}.
In fact, assuming that $P_{n,n+1}>0$, we seek for $\uppi_n$ such that
$	\uppi_n^{-1} P_{n,n+1}=\uppi_{n+1}^{-1}$ 
so that 
\begin{align*}
	\uppi_{n}&=\frac{1}{P_{0,1}P_{1,2}\cdots P_{n-1,n}},
\end{align*}
and we find
\begin{align}\label{eq:stochasticII}
	P&=\uppi T \uppi^{-1}, & \uppi&=\diag(1,\uppi_{1},\uppi_{2},\dots).
\end{align}
Similarly, assuming that $P_{n+1,n}>0$, for $\uppi_{n}=\frac{1}{P_{1,0}P_{2,1}\cdots P_{n,n-1}}$, then
\begin{align}\label{eq:stochasticI}
	P&= \uppi^{-1} T^\top \uppi , & \uppi &=\diag(1,\uppi_{1},\uppi_{2},\dots).
\end{align}

The stochastic banded Hessenberg matrix $P$ acts on the space of complex sequences $f=\{f(n)\}_{n=0}^\infty$ by means of $(Pf)(n)=\sum_{m=0}^\infty P_{n,m}f(m)$. As the matrix is stochastic, $\|P\|_\infty=1$, in where the norm is induced by the Banach space $\ell_\infty$ of bounded complex sequences. 
The set of the sequences $u_n=\{\delta_{n,m}\}_{m=0}^\infty$ is a Hilbert basis for the Hilbert space $\ell_2$.
The Banach space $\ell_\infty(\uppi)$ of sequences $\{f(n)\}_{n=0}^\infty$ such that the sequence 
$\{f(n)\uppi_n\}_{n=0}^\infty$ is bounded is of particular interest. The sequences 
 $u^\uppi_n=\frac{u_n}{\uppi_n}$ are normalized $\| u_n^\uppi\|_{\ell_\infty(\pi)}=1$ and allow to identify $T$ with $P$ as follows
\begin{align*}
	Tu^\uppi_n=\sum_{m=0}^\infty P_{m,n}u_m^\uppi.
\end{align*}
From $\|\sum_{n=0}^\infty x_nu_n^\pi\|_{\ell_\infty(\pi)}=\|\sum_{n=0}^\infty x_nu_n\|_{\infty}$ and 
\begin{align*}
	\|Tx\|_{\ell_\infty(\pi)}=		\left\|T\sum_{n=0}^\infty x_nu_n^\pi\right\|_{\ell_\infty(\pi)} =\left\|\sum_{n,m=0}^\infty x_n P_{m,n}u_m^\pi\right\|_{\ell_\infty(\pi)} =\left\|\sum_{n,m=0}^\infty x_n P_{m,n}u_m\right\|_\infty
%	\|Tx\|_{\ell_\infty(\pi)}\|_{\ell_\infty(\pi)}=\\|P x\|_\infty
\end{align*}we get
$\|T\|_{\ell_\infty(\pi)}=\|P\|_\infty=1$.
%An orthonormal Hilbert basis for the Hilbert space $\ell_2(\pi)$ of complex sequences $\{f(n)\}_{n=0}^\infty$ such that $\sum_{n=0}^\infty |f(n)|^2\pi_n<\infty$ is given by

Hence, the corresponding Markov chains can be described by the spectral methods we have constructed for monic Hessenberg semi-infinite matrices with positive bidiagonal factorization. We will see now that this is equivalent to the existence of a positive stochastic bidiagonal factorization. For that aim we introduce:

\begin{defi}[Positive stochastic bidiagonal factorization]\label{def:positive_stochastic_bidiagonal_factorization}
	A banded stochastic matrix $P_{II}$ of the form \eqref{eq:PII} has a positive stochastic bidiagonal factorization if we can write
	\begin{align*}
		P_{II}= \Pi_1 \cdots \Pi_p \Upsilon,
	\end{align*}
	with stochastic bidiagonal factors
	\begin{align*}
		\Pi_a&\coloneq \,\, \begin{bNiceMatrix}[columns-width=1cm]
			(\Pi_a)_{0,0} &0&\Cdots&&&\phantom{0}\\
			(\Pi_a)_{1,0} & (\Pi_k)_{1,1} &\Ddots&& &\\
			0& (\Pi_a)_{2,1} & (\Pi_a)_{2,2}& &&\\
			\Vdots&\Ddots& (\Pi_a)_{3,2} & (\Pi_a)_{3,3} & &\phantom{h}\\
			&& &\Ddots& \Ddots&\phantom{h}\\
			\phantom{h}&\phantom{h}&\phantom{h}&\phantom{h}&\phantom{h}&\phantom{h}&\phantom{h}
%			\CodeAfter
%			\line{3-1}{6-4} \line{4-3}{6-5}\line{4-4}{6-6} 
%			\line{1-2}{1-6}
%			\line{1-2}{5-6} 
		%	\SubMatrix[{1-1}{6-6}]
		\end{bNiceMatrix}, & a&\in\{1,\dots,p\},\\
		\Upsilon& \coloneq
		\left[\begin{NiceMatrix}[columns-width = .9cm,]
			\Upsilon_{0,0} & \Upsilon_{0,1} &0&\Cdots&\\
			0& \Upsilon_{1,1} & \Upsilon_{1,2} &\Ddots&\\
			\Vdots&\Ddots& \Upsilon_{2,2}&\Ddots&\\
			& & \Ddots &\Ddots &
		\end{NiceMatrix}\right],
	\end{align*}
	i.e., with $(\Pi_a)_{0,0}=1$, $(\Pi_a)_{k,k-1}+(\Pi_a)_{k,k}=1$, $(\Pi_a)_{k,k-1}> 0$, $(\Pi_a)_{k,k}> 0$, $\Upsilon_{k,k}+\Upsilon_{k,k+1}=1$, $\Upsilon_{k,k}> 0$, and $\Upsilon_{k,k+1}> 0$.
	
	A similar definition holds for finite matrices.
\end{defi}
Then, we have an identification between positive bidiagonal factorizations of monic banded Hessenberg matrices and positive stochastic bidiagonal factorizations of stochastic banded Hessenberg matrices. We discuss only the type $II$, as type $I$ is obtained by simple transposition.

\begin{pro}
A finite matrix with a positive stochastic bidiagonal factorization is oscillatory.
\end{pro}

\begin{proof}
	Each factor in the product belongs to InTN, hence the product belongs to InTN. The entries in the first superdiagonal and subdiagonal are all positive, and the Gantmacher--Krein Criterion gives the result.
\end{proof}
\begin{teo}
	Let us assume a type II banded stochastic matrix, $P$, is of the form \eqref{eq:PII}. Then, $P$ has a positive stochastic bidiagonal factorization, as in Definition \ref{def:positive_stochastic_bidiagonal_factorization},
	if and only if it is similar, via a positive diagonal matrix, to a monic Hessenberg matrix $T$ with positive bidiagonal factorization, as in \eqref{eq:bidiagonal}, consequently, 
	%using that itholds \eqref{eq:stochasticII} then
\begin{align*}
		P&= \Pi_1 \cdots \Pi_p \Upsilon, & & \xLeftrightarrow[]{\phantom{holoholad}}& & T=L_1\cdots L_p U,
	\end{align*}
with $L_1,\dots, L_p$ positive lower bidiagonal matrices and $U$ an upper positive bidiagonal matrix, as in \eqref{eq:bidiagonal_factors}. 
\end{teo}

%	Here,
%	\begin{align} \label{eq:HII}
%		\uppi_{II} &%= \diag (1,\uppi_{II,1} , \uppi_{II, 2} \cdots) 
%	=\diag\left(1,\frac{1}{P_{0,1}} , \frac{1}{P_{0,1}P_{1,2} }, \cdots\right),\\\label{eq:rH}
%		P_{0,1}&=\Upsilon_{0,1}, & P_{i,i+1} &=(\Pi_1)_{i,i} \cdots (\Pi_p)_{i,i} \Upsilon_{i,i+1},&
%	\end{align}
%ssssss	
%	
%	
%and
%	\begin{align}\label{eq:LPi}
%		(L_a)_{i-1,i} &= \frac{\uppi_{II,i}}{\uppi_{II,i-1}} \frac{(\Pi_1)_{i-1,i-1}}{(\Pi_1)_{i,i}} \frac{(\Pi_2)_{i-1,i-1}}{(\Pi_2)_{i,i}} \cdots \frac{(\Pi_a)_{i-1,i-1}}{(\Pi_a)_{i,i}},& a&\in\{1,\dots,p\},\\
%		\label{eq:UPi}
%		U_{i,i} &=
%		\frac{\Upsilon_{i,i}}{\Upsilon_{i,i+1}} \frac{\uppi_{II,i}}{\uppi_{II,i+1}} .
%	\end{align}
%\end{teo}
%

\begin{proof}	
	%with $ \Pi^k_{i+1,i} >0, \Pi^k_{i,i}>0 , \, i=0, 1, \cdots $, and $\Upsilon_{i,i} >0, \Upsilon_{i,i+1} >0, i= 0,1, 2 \cdots$, then $L_1,\dots, L_p,U\in \operatorname{InTN}$.
Let us assume that $P$ is of the form \eqref{eq:PII} and has a positive stochastic bidiagonal factorization
$P= \Pi_1 \cdots \Pi_p \Upsilon$, then we can consider the associated monic Hessenberg matrix $T$, such that
$P= \uppi T \uppi^{-1}$. A simple computation shows that $\uppi_{0}=1$ and $\uppi_{In}=\frac{1}{\prod_{k=0}^{n-1} P_{k,k+1}}$, for $n\in\N$.
	To connect these diagonal entries of the similarity matrix $\uppi_{II}$ with the bidiagonal stochastic factors with first notice that
	$(\Pi_p \Upsilon)_{i,i+1} = (\Pi_p)_{i,i} \Upsilon_{i,i+1}$ and that $(\Pi_k)_{0,0} = 1$, for $k=1, \ldots, p$. Then, we use this in a recursive manner
	to get 
	\begin{align*}
	P_{0,1}&=\Upsilon_{0,1}, & P_{i,i+1} &=(\Pi_1)_{i,i} \cdots (\Pi_p)_{i,i} \Upsilon_{i,i+1}.
\end{align*}
	%so the elements
	%$(\Pi_1 \cdots \Pi_p \Gamma)_{i,i+1}$, for $i=0, 1, \dots $ are given by
	%$$
	%\Pi^1_{i,i} \cdots \Pi^p_{i,i} \gammT_{i,i+1}
	%$$
	%

The stochastic bidiagonal factors of the factorization of $P_{II}$ and the positive monic bidiagonal factors of the factorization of $T$ are connected through
	\begin{align*}
		\Pi_1 & =\uppi L_1 D_1^{-1}, & \Pi_2 &=D_1 L_2 D_2^{-1}, & &\ldots ,& \Pi_p &=D_{p-1}L_p D_p^{-1},& \Upsilon &= D_p U \uppi^{-1},
	\end{align*}
where $D_k = \diag( d_{k,0},d_{k,1}, \ldots)$. 
	Let us notice that
	$L_1 = \uppi_{II}^{-1} \Pi_1 D_1$
	so, by direct computation it holds that
	\begin{align*}
		(L_1)_{i,i} &= 1 &&\xLeftrightarrow[]{\phantom{holoholad}} &\frac{1}{\uppi_{i}} (\Pi_1)_{i,i} d_{1,i} &= 1,
	\end{align*}
	and we are lead firstly to $d_{1,i} = \frac{\uppi_i}{ (\Pi_1)_{i,i}} $ and secondly to
	\begin{align*}
		(L_1)_{i-1,i} = \frac{1}{\uppi_{i-1}} (\Pi_1)_{i,i} d_{1,i} 
		= \frac{\uppi_{i}}{\uppi_{i-1}} .
	\end{align*}
	For the next step, let us consider
	$L_2 = D_1^{-1} \Pi_2 D_2$, and, by direct computation, conclude that it holds that
	\begin{align*}
		(L_2)_{i,i} &= 1 &&\xLeftrightarrow[]{\phantom{holoholad}} & \frac{1}{d_{1,i}} (\Pi_2)_{i,i} d_{2,i} &= 1.
	\end{align*}
	Thus, we deduce that
	\begin{align*}
		d_{2,i} = \frac{d_{1,i}}{ (\Pi_2)_{i,i}} = \frac{\uppi_{i}}{ (\Pi_1)_{i,i} (\Pi_2)_{i,i}},
	\end{align*}
and, consequently, that
	\begin{align*}
		(L_2)_{i-1,i} = \frac{1}{d_{1,i-1}} (\Pi_2)_{i,i} d_{2,i} =
		\frac{\uppi_{II,i}}{\uppi_{II,i-1}} \frac{ (\Pi_1)_{i-1,i-1}}{(\Pi_1)_{i,i} }.
	\end{align*}
Hence, recursively, we are lead to
	\begin{align*}
		d_{a,i} &= \frac{\uppi_{i}}{ (\Pi_1)_{i,i} (\Pi_2)_{i,i} \cdots (\Pi_a)_{i,i}}, &a&\in\{1,\dots,p\},
	\end{align*}
	and
	\begin{align*}
		(L_a)_{i-1,i} &= \frac{\uppi_{i}}{\uppi_{i-1}} \frac{(\Pi_1)_{i-1,i-1}}{(\Pi_1)_{i,i}} \frac{(\Pi_2)_{i-1,i-1}}{(\Pi_2)_{i,i}} \cdots \frac{(\Pi_a)_{i-1,i-1}}{(\Pi_a)_{i,i}},& a&\in\{1,\dots,p\}.
		\end{align*}
Finally, from 
	$U = D_p^{-1} \Upsilon \uppi_{II}$ 
we get $
		U_{i,i+1} = \frac{1}{d_{p,i}} \Upsilon_{i,i+1}\uppi_{i+1} = 1$
and
	$U_{i,i} = \frac{1}{d_{p,i}} \Upsilon_{i,i}\uppi_{i}$, so that	
	$U_{i,i} =
		\frac{\Upsilon_{i,i}}{\Upsilon_{i,i+1}} \frac{\uppi_{i}}{\uppi_{i+1}} $.
	
		Now, let us assume that $P$ has a associated PBF monic Hessenberg matrix $T$, then we can prove that
the stochastic matrix $P$ has the following stochastic bidiagonal factorization
%\label{eq:stochastic_factorization_II}
$ P=\Pi_{1}\Pi_{2}\cdots \Pi_{p}\Upsilon$, in terms of the stochastic matrices 
 \begin{align*}
\Pi_1&\coloneq \pi L_1 D_{p}^{-1}, & 
\Pi_2 &\coloneq D_{p} L_2 D_{p-1}^{-1}, &&\dots, &
\Pi_p&\coloneq D_{2} L_p D_{1}^{-1} &
\Upsilon&\coloneq D _{1} U \pi^{-1} .
 \end{align*}
Given any semi-infinite vector $v=\left[\begin{NiceMatrix}[]v_0 &v_1&\Cdots
\end{NiceMatrix}\right]^\top$ we construct a corresponding diagonal matrix $\delta(v)=\diag(v_0,v_1,\dots)$.
We use the notation $\textbf{1}\coloneq\left[\begin{NiceMatrix}[]
	1 &
	1&
	\Cdots
\end{NiceMatrix} \right]^\top$.
 Then
$ D _{1} = \delta (U\uppi^{-1}\textbf{1})^{-1}$, $ D _{2} = \delta (L_p D_{1}^{-1}\textbf{1})^{-1}$,
 and recursively, for $j\in\{1, \dots, p-1\}$,
$ D _{j+1} = \delta( L_{p-j+1} D_{j}^{-1} \textbf{1})^{-1}$.
 Notice that the definition of each of the $D_{j}$ matrices are given in order to get that the corresponding matrices $\Upsilon, \Pi_{p}, \dots ,\Pi_2$ to be stochastic.
To have 
$	 P=\Pi_{1} \Pi_{2}\cdots \Pi_{p} \Upsilon$,
 we need the factor
$\Pi_{1}^L \coloneq \uppi L_1 D_{p}^{-1}$ to be stochastic, an this is the case. Indeed, by~construction this last factor is a nonnegative matrix. Moreover, being $P$ and the other factors stochastic we have $\textbf{1}=P\textbf{1}=\Pi_1 \Pi_{2}\cdots \Pi_{p} \Upsilon =\Pi_{1} \textbf{1}$. Hence, this last factor is a stochastic matrix.
%We can assert that this matrix is stochastic using that the product of stochastic matrices is a stochastic matrix.%and that the inverse of a stochastic matrix is again a stochastic matrix.
\end{proof}

\begin{rem}
	This stochastic factorization appeared in \cite{grunbaum_de la iglesia} in where an urn model was proposed for the Jacobi Markov chain and in \cite{Grunbaum_Iglesia} for the Jacobi--Piñeiro situation, see \cite{nuestro1}. In those papers a factorization $P_{II}=P_LP_U$ with $P_{U}$ being a stochastic upper triangular matrix with only the first superdiagonal nonzero, i.e. an stochastic matrix describing a pure birth Markov chain, and a matrix $P_L$ a stochastic lower triangular matrix, with zero as an absorbing state and only the two first subdiagonals nonzero.
	
	The stochastic factorization provided here for the oscillatory situation in $p$ simple stochastic oscillatory factors is in terms of a pure birth factor and $p$ pure death factors for the type II, and in terms of one pure death factor and $p$ pure birth factors for the type I case. The construction of the corresponding urn models, once the stochastic factorization is provided, will be given by an appropriate choice of $p+1$ urns, one urn per factor, with $p+1$ different experiments. 
\end{rem}

From hereon we will assume that the stochastic matrices $P_{II}$ or $P_I$ have a positive stochastic bidiagonal factorization, this will allow for the application of our spectral Favard theorem and to spectral formulas for significant probability quantities of the Markov chain involved.

\subsection{Finite Markov chains}

Let us assume two semi-infinite stochastic matrices $P_{II}$ or $P_I$ as in~\eqref{eq:stochasticII} and \eqref{eq:stochasticI}, respectively. Here $T$ is a bounded $(2p+2)$-diagonal matrix that we assume to be PBF. Hence, recalling that $ \lambda_1^{[N]}>0$ was the largest eigenvalue of $T^{[N]}$, 
the normalized truncated matrix $	\frac{1}{ \lambda_{1}^{[N]} }T^{[N]}$
has as its largest eigenvalue $1$, with corresponding entrywise positive right 
eigenvectors given by
\begin{align*}
	u_1^{\langle N\rangle}&=\begin{bNiceMatrix}
		1
		& u^{\langle N\rangle}_{1,2}
		& \Cdots 
		& u^{\langle N\rangle}_{1,N+1}
	\end{bNiceMatrix}^\top, & u^{\langle N\rangle}_{1,n}&=B_{n-1}( \lambda_1^{[N]}), 
\end{align*}
and left eigenvectors given by
\begin{align}
\label{eq:Omega1}	\Omega_1^{\langle N\rangle}&=\begin{bNiceMatrix}
		1&\Omega^{\langle N\rangle}_{1,2}&\Cdots&\Omega^{\langle N\rangle}_{1,N+1}
	\end{bNiceMatrix}, 
&
\begin{aligned}[t]\Omega_{1,n}^{\langle N\rangle}
		&=\frac{B^{[n]}_{N+1}( \lambda_{1}^{[N]})}{B^{[1]}_{N+1}( \lambda_{1}^{[N]})}\\ &=A^{(1)}_{n-1}( \lambda_{1}^{[N]})+A_{n-1}^{(2)}( \lambda_1^{[N]})\frac{\mu_{1,2}^{[N]}}{\mu_{1,1}^{[N]}}+\cdots+A_{n-1}^{(p)}( \lambda_{1}^{[N]})\frac{\mu_{1,p}^{[N]}}{\mu_{1,1}^{[N]}}.
	\end{aligned}
\end{align}

Let us take 
\begin{align*}
\tilde \uppi^{[N]}_{II,n}&=\frac{1}{B_n( \lambda_{1}^{[N]}) }, &
	\tilde \uppi_{I,n}^{[N]}&=\frac{B^{[n+1]}_{N+1}( \lambda_{1}^{[N]}) }{B^{[1]}_{N+1}( \lambda_{1}^{[N]}) }=A^{(1)}_{n}( \lambda_{1}^{[N]}) +A_{n}^{(2)}( \lambda_{1}^{[N]}) \frac{\mu_{1,2}^{[N]}}{\mu_{1,1}^{[N]}}+\cdots
		+A_{n}^{(p)}( \lambda_{1}^{[N]}) \frac{\mu_{1,p}^{[N]}}{\mu_{1,1}^{[N]}},
	\end{align*}
 and consider the matrices
\begin{align*}
	\tilde P_{II}^{[N]}&=\tilde \uppi_{II}^{[N]} \frac{T^{[N]}}{\lambda^{[N]}_{1}} (\tilde \uppi_{II}^{[N]})^{-1}, & \tilde \uppi_{II}^{{[N]}}&=\diag\big(1,\tilde \uppi^{{[N]}}_{II,1},\tilde \uppi^{{[N]}}_{II,2},\dots,\tilde \uppi^{{[N]}}_{II,N}\big),\\
\tilde	P_I^{[N]}&= (\tilde \uppi_{I}^{[N]})^{-1}\frac{(T^{[N]})^\top }{\lambda^{[N]}_{1}}\tilde \uppi_{I}^{[N]}, & \tilde \uppi_{I}^{{[N]}}&=\diag\big(1,\tilde \uppi^{{[N]}}_{I,1},\tilde \uppi^{{[N]}}_{II,2},\dots,\tilde \uppi^{{[N]}}_{I,N}\big).
\end{align*}
The stochastic property $\tilde P_{II}^{[N]} \textbf{1}^{[N]} = \textbf{1}^{[N]} $ follows from $T^{[N]} \tilde \uppi_{II}^{-1} \textbf{1}^{[N]} =\lambda^{[N]}_{1} \tilde \uppi_{II}^{-1} \textbf{1}^{[N]} $, recall that $\tilde \uppi_{II}^{-1} \textbf{1}^{[N]} =u_1$. 
Similarly, $\tilde P_{I}^{[N]} \textbf{1}^{[N]} = \textbf{1}^{[N]} $ follows from $\big(T^{[N]}\big)^\top \tilde \uppi_{I}\textbf{1}^{[N]} =\lambda^{[N]}_{1} \tilde \uppi_{I}\textbf{1}^{[N]} $, recall that $ \tilde \uppi_{I} \textbf{1}^{[N]}=\Omega_1^\top$.
Hence we get finite Markov chains, in where the $(p+2)$-diagonal stochastic matrices are given by
\begin{align*}
\tilde P_{II}^{[N]}&=
 \begin{bNiceMatrix}[columns-width =.8cm]
 P_{0,0} & P_{0,1}& 0 & \Cdots& &&&&&0\\
 P_{1,0} & P_{1,1} & P_{1,2}&\Ddots& &&&&&\Vdots\\
 \Vdots & & \Ddots & \Ddots&&&&&\phantom{X}\\
 & & & &&&&&\phantom{X} \\
 P_{p,0}& \Cdots& && P_{p,p}& P_{p,p+1} &&&\phantom{X}\\%[10pt]
 0&P_{p+1,1}&\Cdots&&&P_{p+1,p+1}&P_{p+1,p+2}&&\phantom{X}\\\\
 \Vdots &&\Ddots&\Ddots&&&&\Ddots&\Ddots&0\\
 & & &&&&&&&P_{N-1,N}\\
 0 &\Cdots&&&0&P_{N,N-p}&\Cdots& & &P_{N,N}
% \CodeAfter\line{6-2}{8-4} 
% \line{6-1}{8-3}
% \line{1-3}{6-8}
% \line{2-3}{5-6}\line{2-2}{5-5}
% \line{6-6}{8-8}
% \line{6-7}{7-8}
 % \SubMatrix[{1-1}{7-8}]
 \end{bNiceMatrix}, &
 \end{align*}
with $P_{0,1},P_{1,2}\cdots P_{N-1,N}>0$ and
\begin{align*}
 \tilde P_{I}^{[N]}&=
 \begin{bNiceMatrix}[columns-width =.7cm]
 P_{0,0} & P_{0,1} & \Cdots & & P_{0,p} & 0 & \Cdots && &0\\
 P_{1,0} & P_{1,1} & & & \Vdots & P_{1,p+1} & \Ddots && &\Vdots\\
 0 & P_{2,1} & \Ddots & & & \Vdots & \Ddots && &\\
 \Vdots& \Ddots & \Ddots & & & & & & &\\
 & & & & r _{p,p} & & && &0\\%[10pt]
 & & & & P_{p+1,p} &P_{p+1,p+1} & && &P_{N-p,p}\\
 & & & & &P_{p+1,p+2} & \Ddots && & \Vdots\\
 & & & & & & \Ddots && & \\
 & & & & & & && & \\
 0 & \Cdots & & & & & &0& P_{N,N-1} & P_{N,N}
% \CodeAfter
% \line{3-1}{10-8}
% \line{2-2}{5-5} \line{3-2}{6-5} 
% \line{3-1}{9-6}
% % \line{1-3}{6-9}
% \line{1-6}{3-8}
% \line{2-6}{4-8}
% \line{6-6}{8-8}
% \line{7-6}{9-8}
% \line{1-6}{1-8}
% \SubMatrix[{1-1}{8-8}]
 \end{bNiceMatrix},
 \end{align*}
with $P_{1,0},P_{2,1}\cdots P_{N,N-1}>0$, and
 \begin{align*}
\tilde \uppi^{[N]}_{II,n}&=\frac{1}{P_{0,1},P_{1,2}\cdots P_{n-1,n}},&
\tilde \uppi^{[N]}_{I,n}&=\frac{1}{P_{1,0},P_{2,1}\cdots P_{n,n-1}},\\
 P_{n,n+1}&=\frac{B_{n+1}(\lambda^{[N]}_{1})}{B_{n}(\lambda^{[N]}_{1})},&P_{n,n-1}&=\frac{B^{[n]}_{N+1}(\lambda^{[N]}_{1})}{B^{[n+1]}_{N+1}(\lambda^{[N]}_{1})}=\frac{
 	A^{(1)}_{n-1}(\lambda^{[N]}_{1})\mu_{1,1}^{[N]}+\cdots+A_{n-1}^{(p)}(
\lambda^{[N]}_{1})\mu_{1,2}^{[N]}
 }{A^{(1)}_{n}(\lambda^{[N]}_{1})\mu_{1,1}^{[N]}+\cdots+
 	A_{n}^{(p)}(\lambda^{[N]}_{1})\mu_{1,2}^{[N]}}.
 \end{align*}
\begin{rem}
	We stress that these two finite stochastic matrices $\tilde P^{[N]}_{II}$ and $\tilde P^{[N]}_I$ are not the leading submatrices $ P^{[N]}_{II}$ and $ P^{[N]}_I$ of $
	P_{II}$ and $P_I$, respectively. These truncations are not stochastic matrices, but semistochastic matrices. Also de matrices $\tilde \uppi^{[N]}_{II}$ and $\tilde \uppi^{[N]}_{I}$ are not the truncations $\uppi^{[N]}_{II}$ and $\uppi^{[N]}_I$.
\end{rem}
\begin{rem}
	 For $N>1$ we have $\| P^{[N]}_{II}\|_\infty=1$ %(and for $N>p$ we have $\| \tilde P_{II}\|_infty=1$) 
	 so that its eigenvalues belong to $(0,1]$ (being oscillatory its eigenvalues are positive). But $P^{[N]}_{II}$ is similar to $T^{[N]}$, and therefore for the corresponding eigenvalues we find $0<\lambda^{[N]}_k\leqslant 1$. Hence $[\xi,\eta]\subseteq[0,1]$.
\end{rem}
\begin{rem}\label{rem:1}
	As $\tilde P^{[N]}$ is stochastic $1$ is the largest eigenvalue, being the only eigenvalue with an eigenvector that can be chosen entrywise positive.
	Then, we have two finite rank operator sequences $P^{[N]}$ and $\tilde P^{[N]}$ that converge to $P$ in any Banach space $\ell_q$, $q\in \N$.
	The highest eigenvalue of the first sequence goes to $\eta\in(0,1]$ while for the second sequence is constant and equal to $1$. This strongly suggest that
	$\eta=1$.
\end{rem}

We introduce the notation
\begin{align*}
 \Theta^{[N]}_{II,k,l}&\coloneq \frac{\uppi^{[N]}_{II,k}}{\uppi^{[N]}_{II,l}}=\frac{B_l(\lambda^{[N]}_{1})}{B_k(\lambda^{[N]}_{1})}=\begin{cases}
 \dfrac{1}{P_{l,l+1}\cdots P_{k-1,k}},& l<k,\\
 1, & l=k,\\
 P_{k,k+1}\cdots P_{l-1,l},& l>k,
 \end{cases}\\
 \Theta^{[N]}_{I,l,k}&\coloneq \frac{\uppi^{[N]}_{I,l}}{\uppi^{[N]}_{I,k}}=\frac{B^{[l+1]}_{N+1}(\lambda^{[N]}_{1})}{B^{[k+1]}_{N+1}(\lambda^{[N]}_{1})}=\frac{
 A^{(1)}_{l}(\lambda^{[N]}_{1})\mu_{1,1}^{[N]}+\cdots+
 A_{l}^{(p)}(\lambda^{[N]}_{1})\mu_{1,p}^{[N]}
}{A^{(1)}_{k}(\lambda^{[N]}_{N+1})\mu_{1,1}^{[N]}+\cdots+
 A_{k}^{(p)}(\lambda^{[N]}_{N+1})\mu_{1,p}^{[N]}}=\begin{cases}
 P_{l+1,l}\cdots P_{k+1,k},& l<k,\\
1, & l=k,\\
\dfrac{1}{P_{k+1,k}\cdots P_{l,l-1}} ,& l>k.
 \end{cases}
\end{align*}

We now present three results regarding the spectral representation of important probability elements of the finite Markov chains involved.
We use the notation 
\begin{align*}
	x_N\coloneq \dfrac{x}{\lambda^{[N]}_{1}}. 
\end{align*}

\begin{pro}[Karlin--McGregor representation formula]
 The iterated probabilities have the following spectral representation
\begin{align*}
\left(\big(\tilde P_{II}^{[N]}\big)^n\right)_{k,l}&=\Theta^{[N]}_{II,k,l}
\left\langle A^{(1)}_{l}\mu_1^{[N]}+ \cdots+A^{(p)}_{l}\mu_p^{[N]}, x_N^nB_k\right\rangle
,\\
\left(\big(\tilde P_{I}^{[N]}\big)^n\right)_{k,l}&= \Theta^{[N]}_{I,l,k}
\left\langle A^{(1)}_{k}\mu_1^{[N]}+ \cdots+A^{(p)}_{k}\mu_p^{[N]} , x_N^nB_l\right\rangle.
\end{align*}
\end{pro}
\begin{proof}
	It follows from Theorem \ref{theorem:birothoganality}.%and the consequent spectral representation of $\big( T^{[N]}\big)^n$.
\end{proof}

The generating functions of the probability $P_{ij}^n$ and first-passage-time probability $f_{ij}^n$ are given by
$P_{ij}(s)\coloneq \sum_{n=0}^\infty P_{ij}^ns^n$, $F_{ij}(s)=\sum_{n=1}^\infty f_{ij}^n s^n$. They 
are connected by
$P_{ij}(s)=F_{ij}(s)P_{jj}(s)$, for $ i\neq j$, and $P_{jj}(s)=1+F_{jj}(s)P_{jj}(s)$.
Hence, the generating functions of the first time passage distributions after~$n$ transitions are expressed in terms of the generating functions for the transition probability after~$n$ transitions.

Then, it immediately follows that:
\begin{coro}[Spectral representation of generating functions]
 For $|s|<1$, the corresponding transition probability generating functions are
\begin{align*}
(\tilde P_{II}^{[N]}(s))_{k,l}&= \Theta^{[N]}_{II,k,l}
\left\langle A^{(1)}_{l}\mu_1^{[N]}+\cdots+A^{(p)}_{l}\mu_p^{[N]}, \frac{B_k(x)}{1-sx_N}\right\rangle,\\
(\tilde P_{I}^{[N]}(s))_{k,l}&= \Theta^{[N]}_{I,l,k}
\left\langle A^{(1)}_{k}\mu_1^{[N]} + \cdots+A^{(p)}_{k}\mu_p^{[N]}, \frac{B_l(x)}{1-sx_N}\right\rangle.
\end{align*}
For $k\neq l$, the first passage generating functions are
\begin{align*}
(F_{II}^{[N]}(s))_{k,l}&= \Theta^{[N]}_{II,k,l}
\frac{\left\langle A^{(1)}_{l}\mu_1^{[N]}+\cdots+ A^{(p)}_{l}\mu_p^{[N]}, \frac{B_k(x)}{1-sx_N}\right\rangle}{\left\langle A^{(1)}_{l} \mu^{[N]}_1 +\cdots+ A^{(p)}_{l} \mu^{[N]}_p, \frac{B_l(x)}{1-sx_N}\right\rangle}
,\\
(F_{I}^{[N]}(s))_{k,l}&= \Theta^{[N]}_{I,l,k}\frac{\left\langle A^{(1)}_{k}\mu^{[N]}_1 + \cdots+A^{(p)}_{k}\mu^{[N]}_p , \frac{B_l(x)}{1-sx_N}\right\rangle}{\left\langle A^{(1)}_{l}\mu^{[N]}_1+ \cdots+A^{(p)}_{l} \mu^{[N]}_p, \frac{B_l(x)}{1-sx_N}\right\rangle}.
\end{align*}
For $k=l$ the first passage generating functions are the same for type I and II, namely
\begin{align*}
F_{ll}^{[N]}(s)=1-\dfrac{1}{\left\langle A^{(1)}_{l}\mu^{[N]}_1 + \cdots+A^{(p)}_{l}\mu^{[N]}_p , \dfrac{B_l(x)}{1-sx_N}\right\rangle}.
\end{align*}
\end{coro}

Let us show that the corresponding finite Markov chain is recurrent, for the concept of recurrence see \cite{Gallager,Karlin-McGregor,Karlin-Taylor1}.

%\textcolor{red}{Esta afirmacion nos parece falsa!!!}
\begin{teo}[Recurrent process]
 These finite Markov chains are recurrent.
\end{teo}

\begin{proof}
 Notice that 
$\lim_{s\to 1^-}F_{ll}^{[N]}(s)=1$, 
and all the states are recurrent. This is a consequence of the presence of a mass at $\lambda^{[N]}_{N+1}$ in the linear form $A^{(1)}_{l-1}\mu^{[N]}_1 + \cdots+A^{(p)}_{l-1}\mu^{[N]}_p$, then the contribution at that zero, which has all its factors strictly positive numbers, will have a term proportional to $\frac{1}{1-s}\to\infty$ as $s\to 1^-$ in the denominator of the second term in the right hand side and, consequently, $F_{ll}^{[N]} \to 1$. 
\end{proof}

 Both Markov chains have the same stationary state:
 
\begin{teo}[Stationary states]\label{teo:stationary_states_finite}
 The stationary distribution $ \mathscr \pi^{[N]}=\begin{bNiceMatrix}
 \pi_{1}^{[N]}&\Cdots&\pi^{[N]}_{N+1}
 \end{bNiceMatrix}, $ of both stochastic matrices $\tilde P_{II}^{[N]}$ and $\tilde P_{I}^{[N]}$ has the following components
 \begin{align*}
 \pi^{[N]}_n=u^{[N]}_{1,n}w^{[N]}_{1,n}=\Big(A^{(1)}_{n-1}(\lambda^{[N]}_{1})\mu^{[N]}_{1,1}+ \cdots+A^{(p)}_{n-1}(\lambda^{[N]}_{1})\mu^{[N]}_{1,p}\Big)B_{n-1}(\lambda^{[N]}_{1})=
 \dfrac{B^{[n]}_{N+1}(\lambda^{[N]}_{1})B_{n-1}(\lambda^{[N]}_{1})}{
 B'_{N+1}(\lambda^{[N]}_{1})}.
 \end{align*}
\end{teo}

\begin{proof}
We peek the right and left eigenvectors, $u_1^{[N]}$ and $w_1^{[N]}$, respectively, normalized by $w_1^{[N]}u_1^{[N]}=1$. If we take
\begin{align*}
\pi^{[N]}=(u_1^{[N]})^\top \diag(w_{1,1},\dots, w_{1,N+1})=w_1^{[N]}\diag (u_{1,1},\dots, u_{1,N+1}),
\end{align*}
we deduce that
\begin{align*}
\pi^{[N]} \tilde P_{II}^{[N]}&=w_1^{[N]}\diag (u_{1,1},\dots, u_{1,N+1})\diag (u_{1,1},\dots, u_{1,N+1})^{-1} \frac{T^{[N]} }{\lambda^{[N]}_{1}}\diag (u_{1,1},\dots, u_{1,N+1})=\pi^{[N]},\\
\pi^{[N]}\tilde P_{I}^{[N]}&=(u_1^{[N]})^\top \diag(w_{1,1},\dots, w_{1,N+1}) \diag(w_{1,1},\dots, w_{1,N+1})^{-1}\frac{ (T^{[N]}) ^\top }{\lambda^{[N]}_{1}}\\
& \phantom{olaolaolaolaolaoalolaolaolaolaolaolaolaola} \times \diag(w_{1,1},\dots, w_{1,N+1})=\pi^{[N]}.
\end{align*}
\end{proof}

Notice that $\pi^{[N]}_n>0$, $n\in\{1,\dots,N+1\}$ and that $\displaystyle \pi^{[N]} \textbf{1}^{[N]} =\sum_{n=1}^{N+1}\pi^{[N]}_n=\sum_{n=1}^{N+1} u_{1,n}w_{1,n}=1$.

 For the notion of stationary state see \cite{Gallager,Karlin-McGregor,Karlin-Taylor1}.
\subsection{Countable infinite Markov chains}

We now discuss the case of a countable Markov chain with an infinite countable set of states.
In previous papers \cite{nuestro2,nuestro1} we discuss the construction of Markov chain beyond birth and death Markov chains, and give as explicit examples the well known Jacobi--Piñeiro multiple orthogonal polynomials, see \cite{VanAssche2}, and the recently found hypergeometric multiple orthogonal polynomials \cite{Lima-Loureiro}.%\footnote{See \cite{Grunbaum_Iglesia} for further developments regarding Jacobi--Piñeiro random walks.}
%\textcolor{red}{For both mentioned examples the recursion matrix $T$ is an oscillatory banded Hessenberg matrix.}

The basic idea of \cite{nuestro1} is to construct a Markov chain beyond birth and death for a given sequence of multiple orthogonal polynomials in the step line. The recursion matrix of these multiple orthogonal polynomials happens to be a banded Hessenberg matrix, for which we find a suitable similarity and scaling to get two stochastic matrices.
In those papers the spectral question was not touched. This is an important issue as now the stochastic matrices are essentially non normal operators and are beyond the well known spectral theory for symmetric or normal operators. We address this issue now.

We have two stochastic matrices $P_{II}$ and $P_I$ (admitting a positive stochastic bidiagonal factorization) describing countable infinite Markov chains. There are similarity transformations so that \eqref{eq:stochasticII} and \eqref{eq:stochasticI} hold, i.e. both Markov chains are described by some oscillatory monic banded Hessenberg matrix $T$ (admitting a positive bidiagonal factorization). 

%The scaled leading principal submatrices $\frac{T^{[N]}}{\lambda^{[N]}_{N+1}}$ of $T$, as described in the previous subsection, lead to stochastic matrices $P^{[N]}_{II}$ and $P^{[N]}_I$, which are not leading principal submatrices of $P_{II}$ and $P_I$, but recover them in the large $N$-limit, 
%$P^{[N]}_{II}\xrightarrow[N\to\infty]{}P_{II}$ and $P^{[N]}_{I}\xrightarrow[N\to\infty]{}P_{I}$. Hence, we must have
%\begin{align*}
%\frac{T^{[N]}}{\lambda^{[N]}_{N+1}}\xrightarrow[N\to\infty]{} T,
%\end{align*}
%so that 
%\begin{align*}
%\lambda^{[N]}_{N+1}\xrightarrow[N\to\infty]{} 1.
%\end{align*}
%These is also indicated by the stochastic condition $P_{II}\textbf{1}=\textbf{1}$ and $P_{I}\textbf{1}=\textbf{1}$ 

% Moreover, for the diagonal matrices we have $H^{[N]}_{II}\xrightarrow[N\to\infty]{}\uppi_{II}$ and 
%$H^{[N]}_{I}\xrightarrow[N\to\infty]{}\uppi_{I}$.

The coeffcients $\uppi_{II,n}$ can be expressed in terms of the type II multiple orthogonal evaluated at $1$. Indeed, from the stochastic property expressed as an eigenvalue property $P_{II}\textbf{ 1}=\textbf{1}$ we get that $T \uppi_{II}^{-1}\textbf{1}=\uppi_{II}^{-1}\textbf{1}$. Hence recalling that $TB(x)=xB(x)$, so that $TB(1)=B(1)$, and the fact that all zeros of the type II polynomials belong to $(0,1]$ and, consequently, that $B_n(1)>0$, leads to that up to a multiplicative positive constant $\uppi_{II}^{-1}\textbf{1}=B(1)$, so that 
\begin{align*}
\uppi_{II,n}&=\frac{1}{B_{n}(1)}, & P_{n-1,n}=\frac{B_n(1)}{B_{n-1}(1)}
\end{align*}

 For the type I case the discussion is more involved. Now, the stochastic property $P_I \textbf{1}=\textbf{1}$ leads to the conclusion that $\Omega\coloneq\textbf{1}^\top \uppi_I$ is a entrywise positive left eigenvector of $T$, with eigenvalue $1$, that is $\Omega T=\Omega$. 
 Let us assume that $ \eta=1$, as was discussed in Remark \ref{rem:1} and the existence of the limits
 \begin{align*}
 	F_a&\coloneq \lim_{N\to\infty }\frac{\mu_{1,a}^{[N]}}{\mu_{1,1}^{[N]}}, &a&\in\{1,\dots, p\}.
 \end{align*}
 Recalling \eqref{eq:Omega1} we see that $\lim_{N\to\infty}\Omega^{\langle N\rangle}$, is an entrywise nonnegative left eigenvector of $T$ with eigenvalue $1$. The entries of this left eigenvector are
\begin{align*}
	 \lim_{N\to\infty }\Omega^{\langle N\rangle}_{1,n}=A^{(1)}_{n-1}(1)+A_{n-1}^{(2)}(1) F_2	+\cdots
	 +A_{n-1}^{(p)}(1)	 F_p.
\end{align*}
If none of these entries vanishes we can take, up to a positive multiplicative constant, 
 \begin{align*}
	\uppi_{I,n}=A_{n}^{(1)}(1)+F_2A_{n}^{(2)}(1)+\cdots +F_pA_{n}^{(p)}(1).
\end{align*}
 These last constructions can be understood after the limit is taken and the weights are given as follows. Let us assume for $a\in\{2,\dots, p\}$ that $\d\psi_a\ll\d\psi_1$, i.e., $\d\psi_k$ is absolutely continuous with respect $\d\psi_1 $ ($\d\psi_a(A)=0$ whenever $\d\psi_1(A)=0$). Then, the Radon--Nikodym derivatives are $F_a=\frac{\d\psi_a}{\d\psi_1}$.

Let us use the notation
\begin{align*}
 \Theta_{II,k,l}&\coloneq \frac{\uppi_{II,k}}{\uppi_{II,l}}
 =\begin{cases}
 \dfrac{1}{P_{l,l+1}\cdots P_{k-1,k}},& l<k,\\
 1, & l=k,\\
 P_{k,k+1}\cdots P_{l-1,l},& l>k,
 \end{cases}&
 \Theta_{I,l,k}&\coloneq \frac{\uppi_{I,l}}{\uppi_{I,k}}=\begin{cases}
 P_{l,l+1}\cdots P_{k,k+1},& l<k,\\
 1, & l=k,\\
 \dfrac{1}{P_{k+1,k}\cdots P_{l,l-1}} ,& l>k.
\end{cases}
\end{align*}

Notice that $ \Theta_{II,k,l}=\frac{B_l(1)}{B_k(1)}$ and when the assumptions discussed above for the type I case hold we also have
\begin{align*}
	\Theta_{I,l,k}&=\frac{
	A^{(1)}_{l}(1)+
	A_{l}^{(2)}(1)F_2+\cdots+A_{l}^{(p)}(1)F_p
}{A^{(1)}_{k}(1)+ A_{k}^{(2)}(1)F_2+\cdots+
	A_{k}^{(p)}(1)F_p}.
\end{align*}
As we did with the finite case we proceed to state the main results regarding countable Markov chains with $(p+2)$-diagonal transition matrices with positive bidiagonal stochastic factorization.

\begin{teo}[Spectral representation]\label{theorem:KMG}
 Let us consider a countable Markov chain with transition matrix an $(p+2)$-diagonal matrix such that admits a positive stochastic bidiagonal factorization as in Definition~\ref{def:positive_stochastic_bidiagonal_factorization} (or its transpose version). Then, there is sequence of multiple orthogonal polynomials of type II, $\{B_n\}_{n=0}^\infty$, and of type I, $\{A^{(1)}_n,\dots,A^{(p)}_n\}_{n=0}^\infty$, associated with positive Lebesgue--Stieltjes measures $\d\psi_1,\dots,\d\psi_p$ such that:
\begin{enumerate}
 \item \textbf{Karlin--McGregor representation formula.}
The iterated probabilities have the following spectral representation
 \begin{align*}
 \left(\big(P_{II}\big)^n\right)_{k,l}&= \Theta_{II,k,l}
 \int_0^1
 (A^{(1)}_{l} \d\psi_1(x)+ \cdots+A^{(p)}_{l}\d\psi_p(x))x^nB_k(x)
 ,\\
 \left(\big(P_{I}^{[N]}\big)^n\right)_{k,l}&= \Theta_{I,l,k}
 \int_0^1
 (A^{(1)}_{k} \d\psi_1(x)+ \cdots+A^{(p)}_{k}\d\psi_p(x))x^nB_l(x).
 \end{align*}
\item \textbf{Spectral representation of generating functions.}
For $|s|<1$, the corresponding transition probability generating functions are
\begin{align*}
 (P_{II}(s))_{k,l}&= \Theta_{II,k,l}
 \int_0^1
 \big(A^{(1)}_{l} \d\psi_1(x)+ \cdots+A^{(p)}_{l}\d\psi_p(x) \big)\frac{B_k(x)}{1-sx}
 ,\\
 (P_{I}(s))_{k,l}&= \Theta_{I,l,k}
 \int_0^1
 \big(A^{(1)}_{k} \d\psi_1(x)+ \cdots+A^{(p)}_{k}\d\psi_p(x)\big) \frac{B_l(x)}{1-sx}.
\end{align*}
For $k\neq l$, the first passage generating functions are
\begin{align*}
 (F_{II}(s))_{k,l}&= \Theta_{II,k,l}
 \dfrac{\int_0^1
 \big(A^{(1)}_{l} \d\psi_1(x)+ \cdots+A^{(p)}_{l}\d\psi_p(x) \big)\frac{B_k(x)}{1-sx}}{\int_0^1
 \big(A^{(1)}_{l} \d\psi_1(x)+ \cdots+A^{(p)}_{l}\d\psi_p(x) \big)\frac{B_l(x)}{1-sx}}
 ,\\
 (F_{I}(s))_{k,l}&= \Theta_{I,l,k}\dfrac{\int_0^1
 \big(A^{(1)}_{k} \d\psi_1(x)+\cdots+ A^{(p)}_{k}\d\psi_p(x) \big)\frac{B_l(x)}{1-sx}}{\int_0^1
 \big(A^{(1)}_{l} \d\psi_1(x)+ \cdots+A^{(p)}_{l}\d\psi_p(x) \big)\frac{B_l(x)}{1-sx}}.
\end{align*}
For $k=l$ the first passage generating functions are the same for type I and II, namely
\begin{align*}
 F_{ll}^{[N]}(s)=1-\dfrac{1}{\int_0^1
 \big(A^{(1)}_{l} \d\psi_1(x)+ \cdots+A^{(p)}_{l}\d\psi_p(x) \big)\frac{B_l(x)}{1-sx}}.
\end{align*}

\item \textbf{Recurrent Markov chains.}
The Markov chain is recurrent if and only if the integral
\begin{align*}
 \int_0^{1}\frac{\d\psi_1(x)}{1-x}
\end{align*}
diverges. Otherwise is transient.
\end{enumerate}
\end{teo}
%\textcolor{red}{Discutir !!!!}

Regarding stationary states we find:
\begin{teo}[Ergodic Markov chains]\label{theorem:ergodic}
 The Markov chain described in previous Theorem \ref{theorem:KMG} is ergodic (or positive recurrent) if and only $1$ is a mass point of $\d\psi_1,\d\psi_2,\dots,\d\psi_p$ with masses $m_1>0$ and $m_2,\dots,m_p\geqslant 0$, respectively. In that case, the corresponding stationary distribution is
 \begin{align*}
 \pi&=\left[\begin{NiceMatrix}
 \pi_1 &\pi_2 &\Cdots
 \end{NiceMatrix}\right], &
 \pi_{n+1}&= (A^{(1)}_n(1)m_1+\cdots+A^{(p)}_n(1) m_p)B_n(1).
 \end{align*}
\end{teo}
\begin{proof}
 
Following \cite{Karlin-McGregor}, se also \cite{Gallager}, for a recurrent chain, the expected first passage times are all finite if
 $\displaystyle
 0<\lim_{n\to\infty} P_{00}^{2n}<+\infty$, that is
 \begin{align*}
 \int_0^1\d\psi_1 (x)x^{2n}<+\infty.
 \end{align*}
 The argument goes as in \cite{Karlin-McGregor}. Since $x^{2n}\xrightarrow[n\to\infty]{} 0$ monotonically in $0<x<1$ the chain is ergodic if and only if $\psi_1$ has a jump at $x=1$, and \cite{Karlin-McGregor} shows that the jump occurs at $x=1$.
Therefore, from
Theorem \ref{teo:stationary_states_finite}
we see that 
 \begin{align*}
 \lim_{N\to\infty} \pi^{[N]}_{n+1}=\Big(A^{(1)}_{n}(1) \lim_{N\to\infty} \mu^{[N]}_{1,1}+ \cdots+A^{(p)}_{n}(1) \lim_{N\to\infty} \mu^{[N]}_{1,p}\Big)B_{n}(1),
 \end{align*}
and the large $N$-limits of the Christoffel coefficients $\lim_{N\to\infty} \mu^{[N]}_{1,k}=m_k$ exists due to fact that $1$ is a mass point.

As we have a mass at $1$, we have a simple pole at $1$ in the resolvent; i.e. $1$ is an eigenvalue. Hence, $\sum_{n=1}^\infty\pi_n^2<\infty$.
\end{proof}

 For the notion of ergodic state see \cite{Gallager,Karlin-McGregor,Karlin-Taylor1}.

%\begin{rem}
%Using Theorem \ref{theorem:ergodic} and results from \cite{nuestro1,nuestro2} we conclude that for the Jacobi--Piñeiro, in the semi-band $0<\alpha-\beta<1$, and for the hypergeometric case, as there are no masses, the random walks when recurrent are null recurrent, they are not ergodic, and the return time is infinity. 
%\end{rem}

\begin{rem}
 Krein--Rutman~\cite{Krein_Rutman} seminal theorem on the existence of a positive eigenvector with eigenvalue given by the spectral radius for a compact positive operator, extended previous results in finite dimensions of Perron~\cite{perron1} and Frobenius~\cite{Frobenius,Frobenius2,Frobenius3}. Recall that our banded operators are compact whenever $P_{n,n-i}\xrightarrow[n\to\infty]{}0$, $i\in\{-1,0,1,\dots,p\}$, see \cite{vanassche0}. This is clearly incompatible with the stochastic requirement $\sum_{i=-1}^pP_{n,n-i}=1$. Hence, the Krein--Rutman theorem is not applicable for the case under consideration, oscillatory banded Hessenberg matrices. 
 
 In \cite{Karlin} the results of Krein and Rutman where extended to bounded positive operators in the quasi compact class, i.e., those operators $T$ such that for some integer $k$ we have $\|T^k-K\|<1$ for some compact operator~$K$.
 %\footnote{This is equivalent to the existence of a sequence $K_n$ of compact operators such that $\lim_{n\to\infty} \|T^n-K_n\|=0$.}
 Hence, if the bounded positive matrix is quasi compact, then $1$ is an eigenvalue with totally positive right and left eigenvectors in $\mathscr l^2$. Following \cite{Nussbaum} we know that a strongly positive matrix, a positive matrix $A$ taking totally positive vectors into totally positive vectors, with spectral radius ($1$ in our case) strictly bigger than the essential spectral radius (the supreme of those $\lambda$ such that $\lambda I-A$ being not a Fredholm operator), then $1$ is a simple eigenvalue with a totally positive eigenvector and all the other eigenvalues absolute values less than unity. The existence of these spectral gaps is equivalent to quasi compactness.
% For the Jacobi--Piñeiro and hypergeometric cases mentioned earlier, we know that \cite{nuestro1,nuestro2} the stochastic matrices of type II (a similar discussion holds for type I) can be written as $P=\Theta+K$ with
% \begin{align*}
% \Theta&\coloneq \frac{1}{27}\left( \begin{NiceMatrix}[columns-width = 0.5cm]
% 12 & 8 & 0 & 0 & 0 & \Cdots\\ 
% 6 &12& 8& 0 & 0 & \Ddots \\ 
% 1& 6& 12& 8&0&\Ddots\\ 
% 0&1 &6 & 12& 8& \Ddots\\ \Vdots&\Ddots&\Ddots&\Ddots&\Ddots&\Ddots
% \end{NiceMatrix}\right), &
% K &\coloneq \left( \begin{NiceMatrix}[columns-width = 0.5cm]
% \delta r_0 & \delta s_0 & 0 & 0 & 0 & \Cdots\\ 
% \delta q _1& \delta r_1 & \delta s_1 & 0 & 0 & \Ddots \\ 
% \delta p_2& \delta q_2& \delta P_{2}& \delta s_{2}&0&\Ddots\\ 
% 0&\delta p_3& \delta q_3& \delta P_{3}& \delta s_{3}& \Ddots\\ 
% \Vdots&\Ddots&\Ddots&\Ddots&\Ddots&\Ddots
% \end{NiceMatrix} \right),
% \end{align*}
% with $K$ a compact operator, i.e., $\delta p_n, \delta q_n ,\delta r_n, \delta s_n\xrightarrow[n\to\infty]{}0$, and $\Theta$ a Toeplitz matrix with $\|\Theta^n\|=1$ for any $n\in\N$ and the previous extension to quasi-compact operators can not be applied. 
\end{rem}

\begin{rem}
\begin{enumerate}
 \item For the positive recurrent case the operator $T$ has $1$ as an eigenvalue and a corresponding entrywise positive right eigenvector 
 \begin{align*}
 u_1&=\left[\begin{NiceMatrix}
 1 
 & B_1(1)
 & B_2(1)
 & \Cdots
 \end{NiceMatrix}\right]^\top, &Tu_1&=u_1,
 \end{align*}
with finite norm $\|u_1\|^2=\sum_{n=0}^\infty \big(B_n(1)\big)^2<\infty$, and entrywise positive left eigenvector
\begin{align*}
 w_1&= \left[\begin{NiceMatrix}
 A^{(1)}_0(1)m_1+\cdots+A^{(p)}_0(1)m_p&
 A^{(1)}_1(1)m_1+\cdots+A^{(p)}_1(1)m_p&
 %A^{(1)}_2(1)m_1+\cdots+A^{(p)}_2(1)m_p&
 \Cdots
 \end{NiceMatrix}\right], & w_1T&=w_1,
\end{align*}
with $\|w_1\|^2=\sum_{n=0}^\infty \big(A^{(1)}_n(1)m_1+\cdots+A^{(p)}_n(1)m_p\big)^2<\infty$.

\item Notice that $\pi_n>0$, $n\in\N$ and that $\lim_{N\to\infty}\sum_{n=1}^{N+1}\pi_n=1$ (proof: it holds for each $N$ and the large $N$-limit exists).
 Observe also that for the masses we have
 \begin{align*}
m_k&=\lim_{N\to\infty}\frac{B^{(k)}_{N+1}(1)}{B_{N+1}'(1)}.
 \end{align*}
\end{enumerate}
\end{rem}

\section*{Conclusions and outlook}
The main result we achieve in this paper is that bounded banded Hessenberg matrices that admit a positive bidiagonal factorization have a set positive measures, and can be spectrally described by multiple orthogonal polynomials. This extends to the non-normal scenario the spectral Favard theorem for Jacobi matrices, see \cite{Simon}, that we can show that for an adequate shift becomes oscillatory. 
%We have given examples of this construction, the Toeplitz case leading to 2-orthogonal Chebyshev polynomials, and the hypergeometric recursion matrix, described in \cite{Lima-Loureiro}, also fits in our theory. We have applied these results to Markov chains beyond birth and death with tetradiagonal stochastic matrices. The Jacobi--Piñeiro multiple orthogonal polynomials illustrates the fact that this theory do not exhaust all possibilities, as the weights exist beyond the oscillatory situation.
%We suggest that maybe a more refined requirement is oscillation for a retracted complementary matrix. 
Open questions~are:
\begin{enumerate}
	\item What happens when the banded recursion matrix has several superdiagonals as well as subdiagonals? What about the corresponding Markov chains?
	\item Chebyshev (T) systems appear in \cite{Gantmacher-Krein} in relation with influence kernels and oscillatory matrices. 
	Is there any connection between the AT property and the oscillation of the matrix or some submatrix of it?
\end{enumerate}

\section*{Acknowledgments}
AB acknowledges Centro de Matemática da Universidade de Coimbra UID/MAT/00324/2020, funded by the Portuguese Government through FCT/MEC and co-funded by the European Regional Development Fund through the Partnership Agreement PT2020.

AF acknowledges CIDMA Center for Research and Development in Mathematics and Applications (University of Aveiro) and the Portuguese Foundation for Science and Technology (FCT) within project UIDB/MAT/UID/04106/2020 and UIDP/MAT/04106/2020.

MM acknowledges Spanish ``Agencia Estatal de Investigación'' research projects [PGC2018-096504-B-C33], \emph{Ortogonalidad y Aproximación: Teoría y Aplicaciones en Física Matemática} and [PID2021- 122154NB-I00], \emph{Ortogonalidad y Aproximación con Aplicaciones en Machine Learning y Teoría de la Probabilidad}.

%\enlargethispage{1cm}


\begin{thebibliography}{99}
 

 \bibitem{afm}
 Carlos Álvarez-Fernández, Ulises Fidalgo, and Manuel Mañas,
 \emph{Multiple orthogonal polynomials of mixed type: Gauss-Borel factorization and the multi-component 2D Toda hierarchy},
 Advances in Mathematics~\textbf{227} (2011) 1451–1525.
 
 

 \bibitem{Aptekarev_Kaliaguine_VanIseghem} 
 Alexander Aptekarev, Valery Kaliaguine, and Jeannette Van Iseghem, 
 \emph{The Genetic Sums' Representation for the Moments of a System of Stieltjes Functions and its Application}, Constructive Approximation \textbf{16} (2000) 
487–524.
 
\bibitem{Aptekarev_Kaliaguine_Lopez}
Alexander Aptekarev, Valery Kaliaguine, Guillermo L\'opez Lagomasino, and Ignacio Rocha,
\emph{On the limit behavior of recurrence coefficients for multiple orthogonal polynomials},
Journal of Approximation Theory \textbf{139} (2006), no. 1-2,
346–370.


\bibitem{Aptekarev_Kaliaguine_Saff}
Alexander Aptekarev, Valery Kaliaguine, and Edward Saff,
\emph{Higher-order three-term recurrences and asymptotics of multiple orthogonal polynomials},
Constructive Approximation \textbf{30} (2009), no. 2,
175–223.

\bibitem{dolores}
Dolores Barrios-Rolanía, Amílcar Branquinho, and Ana Foulquié-Moreno,
\emph{On the full Kostant--Toda system and the discrete Korteweg–de Vries equations},
Journal of Mathematical Analysis and Applications \textbf{401} (2013)
811–820.
 
\bibitem{Beckermann_Osipov}
Bernhard Beckermann and Andrey Osipov,
\emph{Some spectral properties of infinite band matrices},
Numerical Algorithms \textbf{34} (2003) 
173–185.

\bibitem{Borges} Carlos F. Borges, \emph{On a class of Gauss-like quadrature rules}, Numerische Mathematik \textbf{67} (1994) 271–288.


\bibitem{nuestro2} 
Amílcar Branquinho, Juan E. Fernández-Díaz, Ana Foulquié-Moreno, and Manuel Mañas,
\emph{Hypergeometric Multiple Orthogonal Polynomials and Random Walks}, \hyperref{https://arxiv.org/pdf/2107.00770.pdf}{}{}{\texttt{arXiv:2107.00770}}.
 
 
 
 \bibitem{bfm}
 Amílcar Branquinho, Ana Foulquié-Moreno, and Manuel Mañas, 
% \emph{Multiple orthogonal polynomials on the step-line}, \hyperref{https://arxiv.org/abs/2106.12707}{}{}{\texttt{arXiv:2106.12707 [CA]}}.
 \emph{Multiple orthogonal polynomials: Pearson equations and Christoffel formulas},
Analysis and Mathematical Physics. \textbf{12} (2022) 129.

 
 \bibitem{bidiagonal_factorization_paper}
 Amílcar Branquinho, Ana Foulquié-Moreno, and Manuel Mañas,
 \emph{Positive Bidiagonal Factorization of Tetradiagonal Hessenberg Matrices}, \hyperref{https://arxiv.org/pdf/2210.10728.pdf}{}{}{\texttt{arXiv:2210.10728}}.
 
 \bibitem{proximo}
 Amílcar Branquinho, Ana Foulquié-Moreno, and Manuel Mañas, 
 % \emph{Multiple orthogonal polynomials on the step-line}, \hyperref{https://arxiv.org/abs/2106.12707}{}{}{\texttt{arXiv:2106.12707 [CA]}}.
 \emph{Bidiagonal factorization of tetradiagonal matrices and Darboux transformations}, 
 \hyperref{https://arxiv.org/pdf/2210.10727.pdf}{}{}{\texttt{arXiv:2210.10727}}.
 

\bibitem{nuestro1} 
Amílcar Branquinho, Ana Foulquié-Moreno, Manuel Mañas, Carlos Álvarez-Fernández, and Juan E. Fernández-Díaz,
\emph{Multiple Orthogonal Polynomials and Random Walks}, 
\hyperref{https://arxiv.org/pdf/2103.13715.pdf}{}{}{\texttt{arXiv:2103.13715}}.


 \bibitem{Chihara} 
 Theodore S. Chihara, 
 \emph{An Introduction to Orthogonal Polynomials}, 
 Gordon \& Breach, 1978, New York. Reprinted by Dover, 2011.



 \bibitem{Coussement-VanAssche} 
 Jonathan Coussement and Walter Van Assche, 
 \emph{Gaussian quadrature for multiple orthogonal polynomials}, 
 Journal of Computational and Applied Mathematics \textbf{178} (2005) 131–145.
 
% \bibitem{Douak}
% Khalfa Douak and Pascal Maroni,
% \emph{On d-orthogonal Tchebychev polynomials, I},
% Applied Numerical Mathematics \textbf{24} (1997) 23–53.
 
% \bibitem{Edrei} Albert Edrei, \emph{Proof of a Conjecture of Schoenberg on the Generating Function of a Totally Positive Sequence}, Canadian Journal of Mathematics \textbf{5} (1953). 86-94. 
 
% \bibitem{Elaydi}
% Saber Elaydi,
% \emph{An Introduction to Difference Equations},
% Third Edition, Undergraduate Texts in Mathematics, Springer-Verlag New York, 2005.
 
 
 
 \bibitem{Fallat-Johnson}
 Shaun M. Fallat and Charles R. Johnson, 
 \emph{Totally Nonnegative Matrices}, 
 Princeton Series in Applied Mathematics,
 Princeton University Press, 2011, Princeton.
 
 
 
\bibitem{Ulises-Illan-Guillermo} 
Ulises Fidalgo Prieto, Jesús Illán, and Guillermo López-Lagomasino, 
\emph{Hermite--Padé approximation and simultaneous quadrature formulas},
 Journal of Approximation Theory \textbf{126} (2004) 171–197.
 
 
 
 \bibitem{Frobenius}
 Georg Frobenius,
 \emph{Über Matrizen aus positiven Elementen, 1},
 Sitzungsberichte der Königlich Preussischen Akademie der Wissenschaften (1908) 471–476.
 
 
 \bibitem{Frobenius2}
 ----------, 
 \emph{Über Matrizen aus positiven Elementen, 2},
 Sitzungsberichte der Königlich Preussischen Akademie der Wissenschaften (1909) 514–518.
 
 
 \bibitem{Frobenius3}
 ----------, 
 \emph{Über Matrizen aus nicht negativen Elementen},
 Sitzungsberichte der Königlich Preussischen Akademie der Wissenschaften (1912) 456–477.
 
 
 
 
 \bibitem{Gallager} 
 Robert G. Gallager, 
 \emph{Stochastic Processes. Theory for Applications}, 
 Cambridge University Press, 2013.
 
 
 
 \bibitem{Gantmacher} 
 Felix P. Gantmacher, 
 \emph{Matrix Theory, volume two}, 
 Chelsea Publishing Company, New York, 1974.
 
 
 
 \bibitem{Gantmacher-Krein} 
 Felix P. Gantmacher and Mark G. Krein, 
 \emph{Oscillation and Kernels and Small Vibrations of Mechanical Systems},
 revised second edition, AMS Chelsea Publishing, American Mathematical Society, Providence, Rhode Island.
 
 
 
\bibitem{grunbaum_de la iglesia}
F. Alberto Grünbaum and Manuel D. de la Iglesia,
\emph{Stochastic LU factorizations, Darboux transformations and urn models},
Journal of Applied Probability \textbf{55} (2018) 862–886. 
 
 
 
\bibitem{Grunbaum_Iglesia} 
 ----------, 
\emph{An urn model for the Jacobi--Piñeiro polynomials}, 
Proceedings of the American Mathematical Society 150 (2022), no. 8,
3613–3625.
%\hyperref{https://doi.org/10.1090/proc/15910}{}{}{DOI:10.1090/proc/15910}.
 
 
 \bibitem{Horn-Johnson} 
Roger A. Horn and and Charles R. Johnson, 
\emph{Matrix Analysis},
 Second Edition, Corrected reprint,
 Cambridge University Press, 2018, Cambridge.
 
 
 
 \bibitem{Ismail} 
 Mourad E. H.Ismail, 
 \emph{Classical and Quantum Orthogonal Polynomails in One Variable}, 
 Encyclopedia of Mathematics and its Applications \textbf{98}, Cambridge University Press, 2009.
 
 
% \bibitem{Jones} William B. Jones and Wolfgang J. Thron, \emph{Continued Fractions: Analytic Theory and Applications}, Encyclopedia of Mathematics and its Applications \textbf{11}, Cambridge University Press, Cambridge, 1984.
 
 
 \bibitem{Kalyagin}
Valery Kalyagin, 
\emph{Hermite-Padé Approximants and Spectral Analysis of Nonsymmetric Operators},
Sbornik: Mathematics \textbf{82} (1995) 199–216.



\bibitem{Kaliaguine}
Valery Kaliaguine, 
\emph{The operator moment problem, vector continued fractions and an explicit form of the Favard theorem for vector orthogonal polynomials}, 
Journal of Computational and Applied Mathematics \textbf{65} (1995) 181–193.

\bibitem{KaliaguineII}
 ----------, 
\emph{On operators associated with Angelesco systems},
East Journal on Approximations \textbf{2} (1995)
157–170.




\bibitem{Karlin}
Samuel Karlin, 
\emph{Positive Operators},
Journal of Mathematics and Mechanics \textbf{8} (1959) 907–937.



\bibitem{Karlin-McGregor}
Samuel Karlin and James McGregor, 
\emph{Random walks}, 
llinois Journal of Mathematics \textbf{3} (1959) 66–81.

\bibitem{Karlin-Taylor1}
Samuel Karlin and Howard M. Taylor,
\emph{A first course in stochastic processes},
Academic Press, 1975.



\bibitem{Krein_Rutman}
Mark G. Krein and Mark Rutman,
\emph{Linear operators leaving invariant a cone in a Banach space},
Uspekhi Matematicheskikh Nauk \textbf{3} (1948) 1–95 (in Russian),
American Mathematical Society Translations \textbf{10} (1950) 199–235.


\bibitem{Lima-Loureiro}
 Hélder Lima and Ana Loureiro,
\emph{Multiple orthogonal polynomials with respect to Gauss’ hypergeometric function},
Studies in Applied Mathematics \textbf{148} (2022) 154–185. 


%\bibitem{Lorentzen} Lisa Lorentzen and Haakon Waadeland,
%\emph{Continued Fractions: 1. Convergence Theory} (2nd Edition), Atlantis Press/World Scientific, Paris (2008).

%\bibitem{manas} Manuel Mañas, \emph{Revisiting Biorthogonal Polynomials. An LU factorization discussion} in 
%\emph{Orthogonal Polynomials: Current Trends and Applications}, edited by E. Huertas and F. Marcellán, SEMA SIMAI Springer Series, \textbf{22} (2021) 273-308, Springer.


\bibitem{nikishin_sorokin}
Evgenii M. Nikishin and Vladimir N. Sorokin,
\emph{Rational Approximations and Orthogonality},
Translations of Mathematical Monographs, \textbf{92},
American Mathematical Society, Providence, 1991.


\bibitem{Nussbaum}
Roger D. Nussbaum, 
\emph{Eigenvectors of nonlinear positive operators and the linear Krein-Rutman theorem},
Fixed Point Theory,
Lecture Notes in Mathematics \textbf{886}, 309–330. Springer-Verlag, 1981. 


\bibitem{perron1}
Oskar Perron,
\emph{Zur Theorie der Matrizen},
Mathematische Annalen \textbf{64} (1907) 248–263.

%\bibitem{Petreolle-Sokal-Zu} Mathias Pétréolle, Allan D. Sokal and Bao-Xuan Zu. \emph{Lattice paths and branched continued fractions: An infinite sequence of generalizations of the Stieltjes-Rogers and Thron-Rogers polynomials, with coefficient wise Hankel-total positivity}. To appear in Memoirs of the American Mathematical Society (2023). \hyperref{https://arxiv.org/pdf/1807.03271.pdf}{}{}{arXiv:1807.03271v2 [math.CO]} (2021).



%\bibitem{Schmudgen} Konrad Schmüdgen, \emph{The Moment Problem}, Graduate Texts in Mathematics \textbf{277}, Springer, 2017.


%\bibitem{Schoenberg} Isaac J. Schoenberg, \emph{Some analytical aspects of the problem of smoothing}, Studies and Essays presented to R. Courant on his 60th Anniversary, 351-370, Interscience Publishers,
%New York, 1948.


%\bibitem{SimonCD}
%Barry Simon,
%\emph{The Christoffel--Darboux kernel}, Proceedings of Symposia in Pure Mathematics \textbf{79}: \emph{Perspectives in Partial Differential Equations, Harmonic Analysis and Applications: A Volume in Honor of Vladimir G. Maz’ya’s 70th Birthday}, 295–336,
%American Mathematical Society, 2008.



\bibitem{Simon} 
Barry Simon, 
\emph{Operator Theory. A Comprehensive Course in Analysis, Part 4},
American Mathematical Society, Providence, Rhode Island, 2015.

\enlargethispage{1cm}

\bibitem{Sorokin_Van_Iseghem_1} 
Vladimir N. Sorokin and Jeannette Van Iseghem, 
\emph{Algebraic Aspects of Matrix Orthogonality for Vector Polynomials},
Journal of Approximation Theory \textbf{90} (1997) 97–116.



\bibitem{Sorokin_Van_Iseghem_2}
----------, 
\emph{Matrix Continued Fractions}, 
Journal of Approximation Theory \textbf{96} (1999), 237–257.



\bibitem{Sorokin_Van_Iseghem_3}
----------, 
\emph{Matrix Hermite--Padé problem and dynamical systems}, 
Journal of Computational and Applied Mathematics \textbf{122} (2000) 275–295.



%\bibitem{vanassche-1}
%Walter Van Assche, \emph{Orthogonal polynomials, associated polynomials and functions of the second kind}, Journal of Computational and Applied Mathematics \textbf{37} (1991) 237-24.


\bibitem{vanassche0} 
Walter Van Assche, 
\emph{Compact Jacobi matrices: from Stieltjes to Krein and $M(a, b)$},
Annales de la faculté des sciences de Toulouse \textbf{6e série, tome S5} (1996) 195–215.

%\bibitem{VanAssche} ----------, \emph{Nonsymmetric Linear Difference Equations for Multiple Orthogonal Polynomials}, Centre de Recherches Mathématiques, CRM Proceedings and Lecture Notes \textbf{25} (2000) 415-429.



\bibitem{VanAssche2} ----------, 
\emph{Multiple Orthogonal Polynomials} in Mourad E. H. Ismail,
\emph{Classical and Quantum Orthogonal Polynomials in one Variable}, 
Encyclopedia of Mathematics and its Applications \textbf{98},
Cambridge University Press, 2005.



%\bibitem{Wall} Hubert S. Wall, \emph{Analytic Theory of Continued Fractions}, American Mathematical Society Chelsea, 1948.

%\bibitem{Zu} Bao-Xuan Zhu, \emph{Coefficientwise Hankel-total positivity of row-generating polynomials for the m-Jacobi-Rogers
% triangle}, \hyperref{https://arxiv.org/pdf/2202.03793v1.pdf}{}{}{arXiv:2202.03793v1
% [math.CO]} (2022).
 

\end{thebibliography}
\end{document}